\documentclass[11pt]{article}       % onecolumn (second format)
\usepackage{graphicx}
\usepackage{latexsym}
\usepackage{amsmath,amsfonts}
\usepackage{amssymb,amsthm}

\usepackage{graphicx}
\usepackage[english]{babel}
\usepackage{color}
\usepackage{dsfont}
\usepackage{geometry}
\usepackage[nottoc]{tocbibind}
\usepackage{indentfirst}
\usepackage{float}

\graphicspath{{./figure/}}

\usepackage{pgfplots}

\usepackage{tikz}

\usepackage{amsmath,amsfonts,amssymb,amsthm}
\usepackage{latexsym}

\renewcommand*{\Re}{\mathop{}\!\mathrm{Re}\,} %for real part

\numberwithin{equation}{section}

\newcommand{\eu}[1]{\textcolor{black}{#1}}

\newtheorem{theorem}{Theorem}[section]{\bf}{\it}

\newtheorem{lemma}[theorem]{Lemma}{\bf}{\it}
{\bf}{\it}
\newtheorem{proposition}[theorem]{Proposition}{\bf}{\it}
{\bf}{\it}

\newtheorem{reminder*}[theorem]{Reminder}

\newtheorem{details*}[theorem]{Details}

\newtheorem{comm*}[theorem]{Comment}
\newtheorem{definition}[theorem]{Definition} 
\newtheorem{definition*}[theorem]{Definition}

\newtheorem{notation*}[theorem]{Notation}
\newtheorem{assumption}[theorem]{Assumption}

\newtheorem{example}[theorem]{Example}
\title{Modelling physiologically structured populations: renewal equations and  partial differential equations}
\begin{document}
	\date{}
	
	
	\author{Eugenia Franco \and Odo Diekmann \and  Mats Gyllenberg}

	%\institute{Eugenia Franco \at  University of Helsinki, Department of Mathematics and Statistics, Helsinki, Finland 
	%          \\\   \email{eugenia.franco@helsinki.fi}        \\\   orcid:0000-0002-5311-2124   %  \\
	%      %  if needed     
	%\and Mats Gyllenberg \at
	%University of Helsinki, Department of Mathematics and Statistics, Helsinki, Finland \\\ orcid:0000-0002-0967-8454
	%\and Odo Diekmann \at Utrecht University, Mathematical Institute, Utrecht, Netherlands\\\ orcid:0000-0003-4695-7601
	%}
	%
	%\date{}

	\maketitle

\begin{abstract}
We analyse the long term behaviour of the measure-valued solutions of a class of linear renewal equations modelling physiologically structured populations. 
The renewal equations that we consider are characterised by a regularisation property of the kernel. This regularisation property allows to deduce the large time behaviour of the measure-valued solutions from the asymptotic behaviour of their absolutely continuous, with respect to the Lebesgue measure, component. We apply the results to a model of cell growth and fission and to a model of waning and boosting of immunity.
For both models we relate the renewal equation (RE) to the partial differential equation (PDE) formulation and draw conclusions about the asymptotic behaviour of the solutions of the PDEs.
\end{abstract}

%%%%%%%%%%%%%%%%%%%%%%%%%%%%%%%%%%%%%%%%%%%%%%%%%%%%%%%%%

\section{Introduction}

Models of physiologically structured populations can take various forms. If the individual states are discrete stages in which individuals sojourn for an exponentially distributed amount of time, then it is natural to formulate the model at the population level as a system of ordinary differential equations describing the rate of change of the number of individuals in the different stages \cite{Fanetal}, \cite{thieme2018mathematics}. 
If the individual states form a continuum, like in e.g. age-size structured populations, there are several popular modelling approaches. 

A very natural, and historically the oldest, approach is to formulate an integral equation of renewal type for the population birth rate.  The method is based on the observation that those who are born at the current time are the children of individuals who were themselves born in the past, have survived up to the current time and give birth at the current time. This approach was formalised by Lotka  \cite{Lotka1907} and Sharpe and Lotka \cite{SharpeLotka1911} for age structured populations, using ideas going back to Euler \cite{Euler1790}.

Another way to model the dynamics of structured populations is to first write down a partial differential equation (PDE) of transport-degradation type describing development (movement in the individual state space) and survival. After that, the PDE is augmented by a rule for reproduction. This can either lead to extra non-local terms in the PDE, like in models of individuals reproducing by fission, or to non-local boundary conditions, like in age-size structured models in which newborns enter the individual state space at the boundary where age is zero.  The PDE-approach was first introduced by McKendrick \cite{McKendrick1926} for age-structured populations and later adapted to age-size-structured populations by Tsuchiya {\em et al.} \cite{Tsuchiya1966}, Bell and Anderson \cite{BellAnderson1967} and others. For a data oriented discrete time variant, see \cite{ellner2016data} and \cite{caswell2000matrix}.

Here we focus on the integral equation for the population birth rate, which in general is a measure representing the rate at which individuals are born in different subsets of the individual state space. If there are no dependencies between individuals such as competition for resources, that is, if the environmental condition is given, then the integral equation is linear and of renewal type.
Here we assume that the environmental condition is constant in time. 

In a recent paper, \cite{franco2021one}, we considered the question when such a renewal equation can be reduced to a one dimensional renewal equation in the sense that the measure-valued solution of the original equation can be recovered from the solution of the one dimensional reduction.
In the present paper we analyse the renewal equation when a different, less restrictive, assumption on the kernel is satisfied and prove asynchronous exponential growth/decline for the solution of the renewal equation. 

\eu{Our interest in measure-valued solutions is motivated by the fact that it allows to consider in a unified way the case in which $\Omega $ is a discrete set and the case in which we have a continuum of states. 
Moreover, considering measure-valued solutions  for REs allows us to draw the connection with measure-valued solutions of PDEs, that have gained much interest in the last years, see for instance  \cite{dull2021spaces}.}

We apply our results to two concrete population models: a model of cell growth and fission (into equal or unequal parts) and a model of waning and boosting of the immunity level.
As anticipated above, these models can be also formulated as PDEs, see \eqref{PDE growth-fission den}, \eqref{eq:equal fission den}, \eqref{eq:waning and boosting den} and could also be analysed in the PDE framework as has been done in \cite{pichor2020dynamics}, \cite{bertoin2019feynman}, \cite{canizo2021spectral}. 

The aim of this paper is twofold. 
On one hand we reiterate the message presented in \cite{franco2021one}, i.e., renewal equations are suitable when dealing with measure-valued solutions. 
The reason is that the existence and uniqueness of their solution can be proven constructively as in the case of scalar equations and, moreover, since renewal equations are integral equations, no regularity assumption with respect to the time variable is required for the concept of solution. This is in contrast with what happens in the PDE framework, where it is typically necessary to work with weak solutions in the measure sense, i.e., weak solutions of the dual equation, see \eqref{PDE}.
 
The second aim of the paper is to provide applicable techniques, based on the work presented in \cite{heijmans1986dynamical} and \cite{inaba2017age}, to study the asymptotic behaviour of the measure-valued solutions of renewal equations. 

In the case of age-structured populations the relationship between the renewal equation and PDE approaches is well understood and discussed in an abstract setting in \cite{Diekmannetal1991}. In the case of size-structured populations the relationship between the two formulations has been investigated in \cite{calsina2016structured} and \cite{barril2021nonlinear}. 
In the closing section of the present paper we show that the measure-valued solutions of the renewal equation yield a solution of  a corresponding PDE and we deduce the asymptotic behaviour of the solution of the PDE from the behaviour of the solution of the RE.

The paper is organised as follows: in Section \ref{sec:well posedness} we provide conditions on the kernel that guarantee the existence of a unique solution for the renewal equation. 
In Section \ref{sec:regularization} we introduce the main assumption of this work: the kernel has a regularising effect on the initial condition. We also motivate heuristically the assumption. 

In Section \ref{sec:method} we prove asynchronous exponential growth for the solution of the renewal equation when the kernel satisfies the assumption presented in Section \ref{sec:regularization}.
We do this by adapting the methods presented in \cite{heijmans1986dynamical} and \cite{inaba2017age}.
The aim of Section \ref{sec:examples} is to show that kernels satisfying the regularisation assumption arise in applications. We analyse the corresponding models in Section \ref{sec:asy models}.
Finally, as anticipated above, Section \ref{sec:PDE} is devoted to the connection between REs and PDEs.  

In Appendix \ref{sec:notation} we collect explanations of the notational conventions, while in Appendix \ref{sec:well posed PDE} we collect results on the existence of a unique solution for the PDE that corresponds to the renewal equation we study.

%%%%%%%%%%%%%%%%%%%%%%%%%%%%%%%%%%%%%%%%%%%%%%%%%%%

\section{The Renewal Equation: existence and uniqueness of the solution} \label{sec:well posedness} 
In this paper we study linear physiologically structured population models that can be formalised via a renewal equation with a  measure-valued solution. 
More precisely, we consider a population of individuals characterised by a structuring variable, $i$-state. We assume that the individual state evolves in time due to different individual level mechanisms that might be continuous and deterministic, as is growth, or discontinuous and stochastic, as is fission. 

We denote with $\Omega $ the set of the possible $i$-states and we assume that $\Omega $ is a Borel subset of $\mathbb R^n.$ 
The set of the possible states at birth is $\Omega_0 \subset \Omega.$

We denote with $B(t, \omega) $ the population birth rate, that is the rate at which individuals appear in the population with state in the set $\omega \in \mathcal B(\Omega_0)$ at time $t$. 
Note that when an individual jumps from state A to another state B, we will say that an individual with state A has died and that an individual with state B is born. Likewise, in the case of cell fission, we consider the disappearance of the mother as ‘death’ and the appearance of the two daughters as ‘birth’.

If we assume that the population distribution at time zero is a given datum $M_0 \in \mathcal M_{+, b} (\Omega)$, then we deduce that $B$ solves
\begin{equation} \label{RE}
B(t, \omega) = \int_0^t \int_{\Omega_0} K(a, \xi, \omega ) B(t-a, d\xi)  da + B_0(t, \omega) \quad t >0, \omega \in \mathcal B(\Omega_0)
\end{equation} 
where $K(t,\xi, \omega)$ is interpreted as the rate at which an individual, having state $\xi$ time $t$ ago, gives birth to an individual with state in the set $\omega$ and
\begin{equation} \label{b0 function of m0}
B_0(t, \omega):=\int_{\Omega} K(t, x, \omega ) M_0( dx).
\end{equation}

Here we repeat the definition of locally bounded kernels, and of their convolution, from \cite{franco2021one}, but we refer to the Appendix of that paper for the proofs of the results presented below. 
\begin{definition} [Locally bounded kernel] \label{def: kernel of type 0} 
		A locally bounded kernel is a positive function $K: \mathbb{R}_+ \times \Omega \times \mathcal B(\Omega_0) \rightarrow \mathbb{R}_+$ with the following properties 
		\begin{enumerate}
			\item for every $(a, \xi) \in \mathbb R_+ \times \Omega$, $K(a,\xi, \cdot) \in M_+(\Omega_0)$ (space of positive Borel measures) 
			\item for every $\omega \in \mathcal B(\Omega_0)$, the function 
			\[
			(a,\xi) \mapsto K(a,\xi,\omega), \quad  (a,\xi) \in \mathbb{R^+} \times \Omega
			\] is measurable (with respect to the product Borel $\sigma$-algebra). 
\item 		for any $ T>0$
		\[
		\sup_{ (a,\xi )  \in [0,T] \times \Omega}   K(a,\xi,\Omega_0) < \infty.
		\]
	\end{enumerate}
	\end{definition} 

The middle argument $\xi$ of $K$ ranges over all of $\Omega$ only in connection with the initial condition, cf. (2.2). In connection with births it ranges over $\Omega_0$.
 We therefore define $\mathbb B_{loc}$, the \textit{set of the locally bounded kernels}, as the set of kernels defined on $\mathbb R_+ \times \Omega_0 \times \mathcal B(\Omega_0)$
such that the properties of Definition \ref{def: kernel of type 0} hold with $\xi$ restricted to $\Omega_0.$

\begin{definition}[Convolution product of kernels]\label{convolution of kernels}
We define the convolution product of $K_1, K_2 \in \mathbb B_{loc}$, as 
\begin{equation}\label{eq:convolution of kernels}
(K_2 * K_1)(t, x, \omega) := \int_0^t \int_{\Omega_0} K_2(t-s, \xi,\omega)K_1(s , x, d\xi) ds.
\end{equation}
\end{definition} 
\begin{definition}[Semiring]
	A semiring $R$ is a set endowed with two binary operations, addition $+$ and multiplication $*$, such that 
	\begin{itemize}
		\item $(R,+)$ is a commutative monoid with identity element $\textbf{0}$: i.e. for every element $a,b,c \in R $ we have that $(a+b)+c=a+(b+c)$, for every $a, b \in R $ we have that $a+b=b+a$ and for every $a \in R$ we have that $a+ \textbf{0}=a$; 
		\item  $(R,*)$ is a semigroup: for every $a,b,c \in R$ we have that $(a*b)*c=a*(b*c)$; 
		\item multiplication from the right and from the left is distributive over the addition, 
		\item  multiplication by $\textbf{0}$ annihilates $R$: for every $a \in R$ we have that $a * \textbf{0}= \textbf{0} *a =  \textbf{0}$. 
	\end{itemize}
\end{definition}
\begin{lemma}[Properties of the convolution]\label{convolution preserves type}
The convolution product $*$ of two locally bounded kernels is a locally bounded kernel. 
The set $\mathbb B_{loc}$, equipped with the sum and with the convolution product $*$, is a semiring.
\end{lemma}
Unlike the classical convolution of scalar functions, the convolution $*$ defined by \eqref{convolution of kernels} is not commutative. 
Moreover, $(\mathbb B_{loc},*)$ is a semigroup, but not a monoid. The reason is that the candidate identity element $\textbf{1}$ is a Dirac measure in the time/age variable, indeed $\textbf{1}(t, x, \omega)=\delta_0(t) \chi_\omega(x) $. Hence $\textbf{1} $ does not belong to $\mathbb B_{loc}.$ 

\begin{definition} \label{def:set X} 
$\mathcal X$ denotes  the set of functions $f: \mathbb R_+ \times \mathcal B(\Omega_0) \rightarrow \mathbb R_+ $ such that for every $a \in \mathbb R_+$, $f(a, \cdot)$ is a measure, the function $f(\cdot, \omega)$ is measurable for every $\omega \in \mathcal B(\Omega_0)$ and $ f(\cdot, \Omega_0)$ is locally integrable.
\end{definition}
\begin{definition}
Given $K \in \mathbb B_{loc}$ and $f \in \mathcal X$, we denote with $\mathcal L_K f $ the convolution of $K$ and $f$, defined by
\begin{equation}\label{def L_k}
(\mathcal L_K f) (t, \omega) := \int_0^t \int_{\Omega_0}  K(t-\sigma, x, \omega) f(\sigma, dx)d\sigma \quad t \geq 0 \quad \omega \in \mathcal B(\Omega_0).
\end{equation}
\end{definition} 

\begin{lemma}
	If $K \in \mathbb B_{loc}$, then
the operator $\mathcal L_K$ is a linear operator from $\mathcal X$ to itself.
If $K_1, K_2 \in \mathbb B_{loc}$, then $\mathcal L_{K_2} \mathcal L_{K_1}  = \mathcal L_{ K_2 * K_1}. $
\end{lemma}

We now interpret \eqref{RE} as the equation $B=\mathcal L_K B+B_0$ with given $B_0 \in \mathcal X$ and unknown $B \in \mathcal X.$

\begin{proposition}\label{prop:existence of a solution for the RE}
Let $K\in \mathbb B_{loc}$ and $B_0 \in \mathcal X$. Then, there exists a unique solution $B$ of equation \eqref{RE} and it is given by
\begin{equation}\label{formula for b}
B = B_0 + \mathcal L_R B_0 
\end{equation} 
where $R \in \mathbb B_{loc}$ is the \textit{resolvent} of the kernel $K$ defined by $R=\sum_{n=1}^\infty  K^{*n}$
where $K^{*1}=K$ and for every $n \geq 2 $
\[
K^{*n}= K^{*(n-1)} *K.
\]
\end{proposition}

We deduce that, if $K$ is a locally bounded kernel and if $B_0$ is given by \eqref{b0 function of m0}, where $M_0 \in \mathcal M_{+,b}(\Omega)$, then $B_0 \in \mathcal X$ and equation \eqref{RE} has a unique solution.

%%%%%%%%%%%%%%%%%%%%%%%%%%%%%%%%%%%%

\section{Reduction to densities} \label{sec:regularization}
In this subsection we present the assumptions on the kernel $K$ that allow us to study the asymptotic behaviour of the solution of equation \eqref{RE} by studying the asymptotic behaviour of the density of its non-singular component. 
We start with an heuristic explanation of the simplification achieved in this manner.  

We can rewrite equation \eqref{RE} in the following translation invariant form
\begin{equation} \label{RE trans inv}
B(t, \omega) = \int_0^\infty \int_{\Omega_0} K(a, \xi, \omega ) B(t-a, d\xi)  da 
\end{equation} 
where for every $\omega \in \mathcal B(\Omega_0)$, $ B(\theta, \omega) d \theta := \Phi (d\theta, \omega ) \text{ if } \theta < 0$ with $\Phi$ a given measure. We allow $\Phi$ to be a measure with respect to the time-of-birth variable simply because it does not harm. 
The fact that equation \eqref{RE trans inv} is translation invariant and linear suggests to look for exponential solutions of the form: 
\begin{equation} \label{expo ansatz}
B(t, \omega):=e^{\lambda t} \Psi(\omega) \text{ for every } t \in \mathbb R. 
\end{equation}
We want to investigate whether such exponential solutions exist and whether they are attractive, i.e., describe the long-term behaviour of $B.$

To guarantee the convergence of the relevant integrals we make the following assumption. 
\begin{assumption} \label{z_0 kernel}
	There exists a $z_0<0 $ and a constant $C>0$ such that for every $t\geq 0$  
	\begin{equation}\label{condition on k for reduction}
	\sup_{x  \in \Omega_0}  K (t,x,  \Omega_0) \leq C e^{ z_0 t}.
	\end{equation} 
\end{assumption}
We say that a locally bounded kernel that satisfies Assumption \ref{z_0 kernel} is a $z_0$-\textit{bounded kernel}.

Substituting the Ansatz \eqref{expo ansatz} in \eqref{RE trans inv} we obtain the following non-linear eigenproblem
\begin{equation}\label{ansatz char}
\Psi(\omega)=  \int_0^\infty  \int_{\Omega_0}  e^{-\lambda a} K(a, \xi, \omega ) \Psi(d \xi) da. 
\end{equation}
So we need to study the properties of the operator
\[
\Psi \mapsto \int_0^\infty  \int_{\Omega_0}  e^{-\lambda a} K(a, \xi, \cdot ) \Psi(d \xi)da. 
\]
 that maps $\mathcal M_{+,b}(\Omega_0) $ into itself. 
In particular we would like to prove its compactness, but this is a very difficult task when we deal with spaces of measures, see for instance \cite{thieme2020discrete}.

Therefore we introduce regularity assumptions on $K$ that allow us to reduce the non-linear eigenproblem \eqref{ansatz char} to measures that are absolutely continuous with respect to the Lebesgue measure, so to an associated non-linear eigenproblem in $L^{1}(\Omega_0).$
It is easiest to assume that for each $t$ and $x$ the measure $K(t,x,.)$ has a density. But as we shall see in Section \ref{sec:examples}, there are natural examples in which the ‘smoothing’ needs one more step. 
\begin{assumption} \label{ass:AC}
For every $x \in \Omega_0$, every $t \geq 0$ the measure
\begin{equation} \label{AC m_xt}
\omega \mapsto \int_0^t \int_{\Omega_0}  K(t-a,\xi, \omega) K(a,x, d \xi ) da 
\end{equation}
is absolutely continuous with respect to the Lebesgue measure. Moreover, for every $t \geq 0$ and for every $f \in L^1(\Omega_0)$ the measure
\begin{equation}\label{int AC}
\omega \mapsto \int_{\Omega_0} K(t,x,\omega) f(x) dx 
\end{equation}
is absolutely continuous with respect to the Lebesgue measure. 
\end{assumption} 
\begin{definition}\label{def:z0 regularizing}
We say that $K$ is a $z_0$-bounded regularizing kernel if it is a $z_0$-bounded kernel that satisfies Assumption \ref{ass:AC}.
\end{definition} 
The interpretation of the absolute continuity with respect to the Lebesgue measure of \eqref{AC m_xt} is that, when we focus on an individual with state $x$ and look $t$ time later at the distribution of the state-at-birth over $\Omega_0$ of grandchildren born at that time, it has a density. So we require that the distribution concentrated in $x$ is, by the combination of growth, survival and twice reproduction, transformed into an absolutely continuous distribution.

On the other hand, the absolute continuity, with respect to the Lebesgue measure, of \eqref{int AC}, guarantees that, if the distribution of the states at birth of a certain generation is absolutely continuous with respect to the Lebesgue measure, then the same is true for the future generations.  

We refer to Appendix \ref{sec:notation} for an explanation of the notation used in the formulation of the following theorem (whose proof is given at the end of the current section).
\begin{theorem}\label{thm:reduction} 
Let $K$ be a $z_0$-bounded regularizing kernel and let $B_0$ be given by \eqref{b0 function of m0} as a function of $K$ and $M_0 \in \mathcal M_{+,b}(\Omega)$. Then the solution $B$ of \eqref{RE} satisfies 
\begin{equation} \label{Bs tends to zero exponentially}
 \| B(t, \cdot)^s \| \leq c_1 e^{z_0 t}+ c_2 t e^{t z_0} \quad t > 0
\end{equation}
where $\| \cdot \|=\| \cdot \|_{TV} = \| \cdot \|_{\flat}$ and $c_1, c_2>0$. 
\end{theorem} 
Since the operator $\mathcal L_K$ is linear, equation \eqref{RE} can be rewritten as 
\begin{equation}\label{RE mix}
B^{AC} +B^s = \mathcal L_K B^{AC}+ \mathcal L_K B^{s} + B_0^{AC}+ B_0^s
\end{equation}
Equation \eqref{RE mix} can be decoupled in a system of two equations 
\begin{equation}\label{RE AC int}
B^{AC} = \mathcal L_K B^{AC}+ \left(\mathcal L_K B^{s}\right)^{AC} + B_0^{AC}
\end{equation}
and 
\begin{equation}\label{singular part}
B^{s} =  \left(\mathcal L_K B^{s}\right)^{s} + B_0^{s}.
\end{equation}
Thanks to Theorem \ref{thm:reduction} we can focus on the asymptotic behaviour of the density of $B^{AC}$  to gain information regarding $B.$

We next present a definition and two lemmas which will be applied in the proof of Theorem \ref{thm:reduction}. 
\begin{definition} \label{ideal}
$\mathcal I$ is the set of the locally bounded kernels $K$ such that 
\[
K(t,x, \cdot)  \in \mathcal M_{+,AC}(\Omega_0) 
\] 
for every $t \geq 0$ and $x \in \Omega_0.$
\end{definition}
\begin{definition}[Right semi-ideal]
	Let $R$ be a semiring with the binary operations $+$ and $*$.
A set $\mathcal J \subset R$ is a right semi-ideal if  $(\mathcal J,+)$ is a monoid and for every $i \in \mathcal J $  and every $K \in R$ we have that $j*K \in \mathcal J.$	
\end{definition} 
\begin{lemma} \label{lem:ideal}
The set $\mathcal I$ is a right semi-ideal.
Moreover, if $K \in \mathcal I$, then for every $t\geq 0$ and every $f \in \mathcal X$, we have that $(\mathcal L_K f) (t, \cdot) \in \mathcal M_{+, AC}(\Omega_0)$. 
\end{lemma} 
\begin{proof}
Assume that $K_2 \in \mathcal I$ and that $K_1 \in \mathbb B_{loc} $. 
Consider a set $A$ that has Lebesgue measure equal to zero. 
Since $K_2 \in \mathcal I$, we deduce that, for every $t\geq 0 $ and $x \in \Omega_0$,  
\[
K_2(t,x,A)=0. 
\]
By Definition \ref{convolution of kernels} and formula \eqref{eq:convolution of kernels}, we deduce that, for every $t\geq 0$ and $x \in \Omega_0$
\[
(K_2 * K_1)(t,x,A)=0.
\] 
We conclude that for every $t\geq 0$ and $x \in \Omega_0$ the measure $(K_2 * K_1)(t,x, \cdot) $ is absolutely continuous with respect to the Lebesgue measure, hence  $ K_2 * K_1 \in \mathcal I.$
The second statement of the proof follows analogously from formula \eqref{def L_k}
\qed\end{proof} 

\begin{lemma}\label{lem:r is AC} 
If $K$ satisfies Assumption \ref{ass:AC}, then $ \sum_{n=2}^\infty K^{*n} \in \mathcal I$. 
\end{lemma}
\begin{proof}
Thanks to Assumption \ref{ass:AC} we know that $K* K \in \mathcal I$. 
Since $\mathcal I$ is a right-ideal we deduce that, if $K^{*n} \in \mathcal I$, then $K^{*n} * K= K^{*(n+1)} \in \mathcal I$. We conclude by induction that $\sum_{n=2}^\infty K^{*n} \in \mathcal I$. 
\qed\end{proof} 
\begin{proof}[Proof of Theorem  \ref{thm:reduction}]
Since $B$ solves \eqref{RE}, then
\[
B=B_0+ \mathcal  L_R B_0= B_0 + \mathcal  L_K B_0+ \mathcal  L_{ \sum_{n=2}^\infty K^{*n}} B_0. 
\]
Thanks to Lemma \ref{lem:r is AC} we deduce that
\[
B^s=B_0^s + \mathcal  L_K B_0^s . 
\]
As a consequence of the fact that for every $t\geq 0$, $B_0(t, \Omega_0) \leq c_1 e^{z_0 t} $ and \eqref{condition on k for reduction}, we deduce that there exists $c_2>0$ such that
\[
\mathcal L_K B_0 (t, \Omega_0) \leq c_2 t e^{z_0 t} \quad \text{ for every } t \geq 0, 
\]
hence $B^s(t, \Omega_0) \leq c_1 e^{z_0 t} + c_2 t e^{z_0 t}. $
\qed\end{proof}

%%%%%%%%%%%%%%%%%%%%%%%%%%%%%%%%%%%%%%%%%%%%%%
\section{Asymptotic behaviour of the solution of the renewal equation} \label{sec:method}
In this section we denote with $X$ the Banach space $L^1(\Omega_0)$ endowed with the $L^1$ norm $\| \cdot \|_{1}.$
Moreover we denote with $X_+$ the cone of the positive functions in $L^1(\Omega_0)$ and we call the bounded linear operator $L: X \rightarrow X$ positive if $L: X_+ \rightarrow X_+$. 

To study the asymptotic behaviour of the solution $B$ of \eqref{RE} we adopt the following strategy: 
\begin{itemize}
\item in Section \ref{sec:RE for density} we introduce the renewal equation for the density of $B^{AC}$ and we prove that it has a unique solution; 
\item in Section \ref{sec:charc} we perform the Laplace transform to all the terms in the renewal equation for the density of $B^{AC}$. We derive in this way a non-linear eigenproblem;
\item in Section \ref{sec:positive op} we present some results on positive operators that are important to study the non-linear eigenproblem derived in \eqref{sec:charc}; 
\item in Section \ref{sec:r} we prove that there exists a unique, up to renormalisation, real eigencouple solving the non-linear eigenproblem derived in Section \ref{sec:charc}; 
\item in Section \ref{sec:asympt AC} we adapt the approach presented by Heijmans in \cite{heijmans1986dynamical} to prove that the solution of the non-linear eigenproblem is attractive, i.e., we deduce the asymptotic behaviour of the density of $B^{AC}$; 
\item in Section \ref{sec:gripenberg} we sketch a different approach to derive the asymptotic behaviour of the density of $B^{AC};$
\item in Section \ref{sec:asympt measure} we show that the behaviour of  the density of $B^{AC} $ determines the behaviour of $B$. 
\end{itemize}

\subsection{Renewal equation for the density} \label{sec:RE for density}
\begin{definition}\label{def: operator kernel}
A positive locally bounded operator kernel is a map $\tilde{K}: \mathbb R_+ \rightarrow \mathcal L (X)$ such that 
\begin{itemize}
\item $\tilde{K}(a)$ is a positive operator for every $a \geq 0$; 
\item the map $a \mapsto \tilde{K}(a) f $ is Bochner measurable  for every $f \in X $ ; 
\item  for every $T\geq 0$
\begin{equation}\label{kernel operator boundedness}
\sup_{a \in [0,T]} \| \tilde{K}(a)\|_{op} =
\sup_{a \in [0,T]} \sup_{\{f \in X: \|f\|_1 \leq 1\}} \| \tilde{K}(a) f \|_{1} < \infty.
\end{equation}
\end{itemize}
\end{definition}
Since in this paper we will only deal with operator kernels that are positive and locally bounded, in the following we use the term \textit{operator kernel} to refer to locally bounded operator kernels. 
\begin{lemma} \label{lem:for b}
Let $\tilde{K}$ be an operator kernel. 
Let $b_0 : \mathbb R_+ \rightarrow X $ be Bochner measurable and locally bounded.
Then the equation
\begin{equation}\label{eq b}
b(t) = \int_0^t \tilde{K}(t-a) b(a) da + b_0(t), \quad t\geq 0
\end{equation}
has a unique solution $b:  \mathbb R^*_+ \rightarrow X$, which is locally bounded and Bochner measurable. 
\end{lemma}
\begin{proof}
The main step of the proof consists in proving the existence of the resolvent of $\tilde{K}.$
To this end, we aim at proving that
\begin{align*}
\sup_{a \in [0,T]}  \sup_{\{f \in X: \|f\|_1 \leq 1\} } \| \tilde{R}(a) f \|_{1} =\sup_{a \in [0,T]}   \sup_{\{f \in X: \|f\|_1 \leq 1\} } \|\sum_{n=1}^\infty \tilde{K}^{\star n} (a) f \|_{1}  < \infty
\end{align*}
where for every $f \in X$ and every $a \geq 0$
\[
K^{\star 1}(a)f:=\tilde{K}(a) f 
\]
and for every $n \geq 2 $
\[
\tilde{K}^{\star n} (a ) f := \int_0^a \tilde{K}^{\star(n-1)}(a-s) \tilde{K}(s) f ds. 
\]
To ensure that the resolvent is well defined, we need to prove that, if $K_i$ are operator kernels, then $K_1 \star K_2: \mathbb R_+ \rightarrow \mathcal L(X)$, defined by
\[
K_1 \star K_2: t \mapsto \left( f \mapsto \int_0^t K_1 (t-a) K_2(a) f da \right)
\]
is also an operator kernel.
Inequality \eqref{kernel operator boundedness} follows by the boundedness properties of $K_1$ and $K_2$, while the Bochner measurability can be proven by an adaptation of the proof of the measurability of the classical convolution product. See the proof of Theorem 1 in \cite{gripenberg1987asymptotic} for more details.

Hence, if 
\[
  \sup_{\{f \in X: \|f\|_1 \leq 1\} } \left\| \int_0^T \tilde{K}(s) f ds \right\|_1 < 1  
\] 
then for every $0 \leq a \leq T$
\begin{align*}
  \sup_{\{f \in X: \|f\|_1 \leq 1\} } \| \tilde{R}(a) f \|_{1}&=  \sup_{\{f \in X: \|f\|_1 \leq 1\} } \left\|\sum_{n=1}^\infty \tilde{K}^{\star n} (a) f \right\|_{1} \\
  & \leq \sum_{n=1}^\infty \left(  \sup_{\{f \in X: \|f\|_1 \leq 1\} } \left\| \int_0^T \tilde{K} (s) f ds \right\|_{1} \right)^n < \infty.
\end{align*}
If, instead, 
\[
  \sup_{\{f \in X: \|f\|_1 \leq 1\} } \left\| \int_0^T \tilde{K}(a) f da \right\|_1 \geq 1  
\] 
the above argument can be adapted by considering a scaled version of $\tilde{K}$, $\tilde{K}_\lambda (a):= e^{-\lambda a} \tilde{K}(a) $, with $\lambda$ chosen such that
\[
  \sup_{\{f \in X: \|f\|_1 \leq 1\} } \left\| \int_0^T \tilde{K}_\lambda (a)f da \right\|_1 < 1.  
\] 
The uniqueness of the solution of equation \eqref{eq b} follows by standard arguments of renewal theory.
See for instance \cite{gripenberg1987asymptotic} or \cite[pp. 233-234]{gripenberg1990volterra}. 
\qed\end{proof}

\subsection{Laplace transformed equation} \label{sec:charc} 
In this section, we make the following assumptions on $\tilde K$ and $b_0$. 
\begin{assumption}\label{ass:on the operator Ktilde}
$\tilde K$  is an operator kernel.
Moreover, there exists a $z_0<0$ and a constant $C>0$ such that for every $t \geq 0$
\begin{equation}\label{eq:z_0 Ktilde}
\| \tilde{K} (t) f \|_{1} \leq C e^{z_0 t}  \|f \|_1.
\end{equation}
\end{assumption} 
\begin{assumption}\label{ass:on b_0}
 $b_0: \mathbb R_+ \rightarrow X $ is a Bochner measurable function and there exists a $c>0$ such that for every $t \geq 0$
\begin{equation} \label{bound b0}
\| b_0(t) \|_1 \leq c e^{z_0 t}+ c_1 t e^{z_0 t} + c_2 t^2 e^{z_0 t}.
\end{equation}
\end{assumption}
We denote with $b$ the solution of equation \eqref{eq b}. 
\begin{lemma}
There exists a $\beta \in \mathbb R$ such that $b(t)e^{-\lambda t}$ is integrable over $\mathbb R_+$ for every $ \lambda > \beta.$
\end{lemma}
\begin{proof}
This proof is an adaptation of the proof of Lemma 3.4 in \cite{heijmans1986dynamical}. 
Thanks to the fact that $\tilde{K}$ satisfies \eqref{eq:z_0 Ktilde} and $b_0$ satisfies \eqref{bound b0}, we know that there exists a $\beta \in \mathbb R$ such that both
\[
\int_0^\infty e^{-\beta  a} \sup_{\{f \in X: \| f\|_1=1\} } \left\| \tilde{ K}(a) f \right\|_1 da= k_1 < 1
\]
and
\[
\sup_{ t\geq 0} e^{-\beta   t } \|  b_0(t) \|_1 = k_2 <\infty 
\] hold.
Since $b$ satisfies \eqref{eq b}, then 
\begin{align*}
& e^{ - \beta  t } \|b(t)\|_1  \leq   e^{ - \beta t} \left\| \int_0^t \tilde{K}(a)b(t-a) da \right\|_1 + e^{- \beta  t }  \|b_0(t)\|_1 \\
& \leq  \left\| \int_0^t e^{ -\beta a} \tilde{K}(a) e^{-\beta  (t-a)} b(t-a) da\right\|_1+ k_2 \\
\end{align*}
Consider the map $M: \mathbb R_+ \rightarrow \mathbb R_+$ defined by $M(t):= \max_{a \in [0,t]} e^{-\beta  a} \| b(a)\|_1$
then we deduce that for every $t>0$
\begin{align*}
M(t) \leq M(t) k_1 + k_2
\end{align*}
this implies that $M(t) \leq \frac{k_2}{1-k_1}$. We deduce that $\| b (t)\|_1 \leq c e^{\beta  t}$
for a positive constant $c>0$, and the desired conclusion follows.
\qed\end{proof}

As a consequence the Laplace transform of $b$, 
\[
\hat{b}(\lambda):= \int_0^\infty e^{-\lambda t} b(t) dt 
\]  
is well defined for every $\lambda \in \mathbb C$ with $\Re \lambda > \beta.$

The \textit{next generation operator} corresponding to the operator kernel $\tilde{K}$ is the operator $\mathbb K_0 : X \rightarrow X$ defined by 
\begin{align}\label{NGO}
 \mathbb K_0 f: = \int_0^\infty \tilde{K}(a)f da. 
\end{align}
Notice that the integral in \eqref{NGO} is guaranteed to converge thanks to the fact that $\tilde{K} $ satisfies \eqref{eq:z_0 Ktilde}.

Motivated by the interpretation in the context of population models, we call 
\begin{equation}\label{basci repro n}
R_0:=\rho(\mathbb K_0),
\end{equation}
where $\rho(\mathbb K_0)$ denotes the spectral radius of $\mathbb K_0$, \textit{basic reproduction number}.

We denote with $\mathbb K_\lambda$ the \textit{discounted next generation operator}
\begin{align}\label{NGO lambda}
 \mathbb K_\lambda f: = \int_0^\infty e^{- \lambda a} \tilde{K}(a)f da. 
\end{align}
Notice that if $\lambda  \in \mathbb C$ then $\mathbb K_\lambda $ is a complex-valued function. This is the reason why we introduce the concept of complexification of a Banach space and of a linear operator. 

We denote with $X^{\mathbb C}$ the set of the functions $f : \Omega_0 \rightarrow \mathbb C$ such that  $f=f_1  +i  f_2$ for some $f_1 \in X$ and $f_2 \in X$. 
\begin{definition}
	The complexification of a linear operator $T: X \rightarrow X$ is the operator $T:X^{\mathbb C}\rightarrow X^{\mathbb C}$ defined by 
	\[
	T (f+ ig ) = Tf+i Tg.
	\]
\end{definition}

The operator $\mathbb K_\lambda$ is well defined for every $\lambda \in \mathbb C$ with $\Re \lambda > z_0$. The same holds for
\[
\hat{b_0}(\lambda):= \int_0^\infty e^{-\lambda t} b_0(t) dt. 
\]

The Laplace transformed version of equation \eqref{eq b}, is
\begin{equation} \label{eq:laplace transform RE AC} 
\hat{b}(\lambda)=\hat{b_0}(\lambda) + \mathbb K_\lambda \hat{b}(\lambda) \quad \Re \lambda > z_0.
\end{equation} 

Let 
\begin{equation}\label{sigma}
\Sigma:=\{ \lambda \in \Delta: 1 \in \sigma (\mathbb K_\lambda )  \}
\end{equation}
where 
\begin{equation}\label{delta} 
\Delta:=\{ \lambda \in \mathbb C: \Re \lambda > z_0 \}
\end{equation}
For $\lambda \in \mathbb C \setminus \Sigma$ it is possible to write 
\begin{equation}\label{lap transform of b}
\hat{b}(\lambda)= \left( I - \mathbb K_\lambda\right)^{-1} \hat{b_0}(\lambda).
\end{equation}
As will be explained later, applying the inverse Laplace transform formula to $\hat{b}$, we deduce the asymptotic behaviour of $b$. 

It is then clear that the first fundamental step to deduce the asymptotic behaviour of $\hat{b}$ is to study the non-linear eigenproblem 
\begin{equation}\label{eigenproblem}
f =\mathbb K_\lambda f, 
\end{equation}
which is, in a sense, the differentiated version of \eqref{ansatz char}.

If the non-linear eigenproblem \eqref{eigenproblem} has a unique, upon normalisation of $f$, real solution $(\lambda,f)=(r, \psi_r)$, then $r$ is called \textit{Malthusian parameter}, while the eigenvector $\psi_r$ is called the \textit{stable distribution}.

\subsection{Compact and non-supporting positive operators} \label{sec:positive op}
 The aim of this section is to present the results on positive compact and non-supporting operators that we need to study the non-linear eigenproblem \eqref{eigenproblem}. 
To this end we introduce the following notation: $X_+^* $ is the positive cone in the dual of $X^*$, represented by the set $L^\infty_+(\Omega_0)$.

We start this section by introducing the concept of non-supporting operators. 
\begin{definition}[Non-supporting operator]\label{def:nonsupporting}
	Let $L: X \rightarrow X$ be a positive bounded linear operator. The operator $L$ is non-supporting with respect to $X_+$ if for every $\psi \in X_+$, $\psi \neq 0$ and $F \in X_+^*$, $F \neq 0$, there exists an integer $p$ such that for every $n \geq p $ we have that $\langle F, L^n \psi\rangle >0$.
\end{definition}

The following result is fundamental for our purposes as it provides important information regarding the spectral radius of positive non-supporting operators. 
We do not write the statement in its most general form, i.e., for a generic Banach space $E$ with certain properties, but we state the result for $E=X=L^1(\Omega_0)$. 
\begin{theorem}[\cite{marek1970frobenius} and \cite{sawashima1965spectal}]\label{theorem 5.2}
	Let $T:X\rightarrow X $ be positive and  non-supporting (cf. Definition \ref{def:nonsupporting}) and suppose that $\rho(T)$
	is a pole of the resolvent, then 
	\begin{enumerate}
		\item $\rho(T)>0$ and $\rho(T)$ is  an  algebraically  simple  eigenvalue of $T$. 
		\item  The corresponding eigenvector $\psi$ is almost everywhere strictly positive.
		\item  The  corresponding dual eigenvector $F$ is strictly positive, i.e.$\langle F, \phi\rangle >0$ for every $\phi \in X_+$ with $\phi \neq 0.$
		\item  If $\{\lambda \in \sigma(T): |\lambda|=\rho(T)\}$  consists only of poles  of the  resolvent, then it consists only of $\lambda =\rho(T)$ and all the remaining elements $\lambda \in \sigma(T)$ satisfy $|\lambda |< \rho(T)$.
	\end{enumerate} 
\end{theorem}
The following result, proven by Marek, \cite[Theorem 4.3 and Theorem 4.5]{marek1970frobenius}, allows us to compare the spectral radius of two positive operators by comparing the operators.  
	Again, we do not write the statement in its most general form, but we state the result for $E=X=L^1(\Omega_0)$ and for the classes of operators we are interested in. 
\begin{proposition}\label{prop 10.12}
	Suppose that $S,T: X \rightarrow X$ are positive, bounded, linear operators. Then, the following holds:
	\begin{enumerate}
		\item if $S \leq T$, that is if $T-S: X_+ \rightarrow X_+ $, then $\rho(S)\leq \rho(T )$; 
		\item if $T,S$ are non-supporting and compact and  $S \leq T$ with $S \neq T$, then $\rho(S)<  \rho(T )$.
	\end{enumerate}
\end{proposition}

\subsection{The Malthusian parameter $r$} \label{sec:r}
In this section we make again Assumption \ref{ass:on the operator Ktilde} on $\tilde K$ and Assumption \ref{ass:on b_0} on $b_0$. 
We denote with $\mathbb K_\lambda $ the discounted next generation operator, \eqref{NGO lambda}.
We recall that $\Delta$ is given by \eqref{delta}. 

The aim of this section is to prove that there exists a unique, up to renormalisation, real solution to the non-linear eigenproblem \eqref{eigenproblem}. 
  \begin{theorem}\label{thm:eigencouple}
	Assume that for every $\lambda \in \Delta \cap \mathbb R$ the positive operator $\mathbb K_\lambda $ is non-supporting and compact. 
	Then, there exists a unique real eigencouple $(r, \psi_r)$, with $\psi_r \in X_+$ and $\| \psi_r\|_1=1$, that solves equation \eqref{eigenproblem}. 
	If $R_0 >1 $ then $r >0$, if $R_0=1$ then $r=0$, if $R_0 <1 $ then $r <0.$
\end{theorem}
To prove Theorem \ref{thm:eigencouple} we follow the approach presented by Heijmans in \cite{heijmans1986dynamical}.
The main steps of the proof consist in
\begin{enumerate} 
	\item proving that $\rho(\mathbb K_\lambda)$ is a positive eigenvalue of $\mathbb K_\lambda$ and that its corresponding eigenfunction is strictly positive: to this end we apply Theorem \ref{theorem 5.2}, hence we need the operator $\mathbb K_\lambda $ to be compact and non-supporting;
	\item proving that the function $\lambda \mapsto \rho(\mathbb K_\lambda )$ is strictly decreasing and continuous and that $\lim_{\lambda \rightarrow z_0} \rho(\mathbb K_\lambda ) \geq 1$ while $\lim_{\lambda \rightarrow \infty } \rho(\mathbb K_\lambda )=0$. To this end we will employ step $1$ and Proposition \ref{prop 10.12}.
\end{enumerate} 
Step $1$ is made in Lemma \ref{lem:spectral radius is eigenvector} and Step $2$ is made in Proposition \ref{prop:Klambda decreasing} and Proposition \ref{prop:caract}. 

\begin{lemma}\label{lem:spectral radius is eigenvector}
	Assume that the operator $\mathbb K_\lambda$ is compact and non-supporting for every $\lambda \in \Delta \cap \mathbb R$.  Then
	\begin{enumerate}
		\item  the spectral radius of $\mathbb K_\lambda$, denoted with $\rho(\mathbb K_\lambda) $, is a positive, algebraically simple eigenvalue of $\mathbb K_\lambda$; 
		\item the corresponding eigenvector $\psi_\lambda \in X$, with normalisation $\| \psi_\lambda\|_1=1$, satisfies $\psi_\lambda(x) >0 $ for a.e. $x \in \Omega_0$ 
		\item the dual eigenfunctional $F_\lambda \in X^*$, such that $\mathbb K_\lambda^* F_\lambda =\rho(\mathbb K_\lambda) F_\lambda$ where $\mathbb K_\lambda^* $ is the dual operator of $\mathbb K_\lambda$, is strictly positive, i.e. $\langle F_\lambda , \phi \rangle>0$ for every $\phi \in X_+$ with $\phi \neq 0$.
	\end{enumerate}
\end{lemma}
\begin{proof}
	By the fact that $\mathbb K_\lambda$ is positive we deduce that $\rho(\mathbb K_\lambda) \in \sigma (\mathbb K_\lambda)$. 
	Since $\mathbb K_\lambda$ is also compact we deduce that the spectral radius is a pole of the resolvent. 
	Hence, if we additionally assume that $\mathbb K_\lambda$ is non-supporting, we can apply Theorem \ref{theorem 5.2} and deduce the desired conclusion.
\qed\end{proof}

\begin{proposition} \label{prop:Klambda decreasing}
	Assume $\mathbb K_\lambda$ to be compact and non-supporting for every $\lambda \in \Delta \cap \mathbb R$. 
	The map $\lambda \mapsto \rho( \mathbb K_\lambda) $ is decreasing and continuous for $\lambda \in [z_0, \infty)$.
\end{proposition}
\begin{proof}
	To prove that the function $\lambda \mapsto \rho(\mathbb K_\lambda)$ is decreasing it is enough to notice that if $\lambda_2>\lambda_1$, then for every $f\in X_+ $ we have that $\left(\mathbb K_{\lambda_1}- \mathbb K_{\lambda_2}\right) f$ belongs to $X_+.$ 
	Hence, by Proposition \ref{prop 10.12} we deduce that $0<\rho(\mathbb K_{\lambda_2}) < \rho(\mathbb K_{\lambda_1})$ by Theorem \ref{theorem 5.2}. 
	
	To prove the continuity notice that thanks to Lemma \ref{lem:spectral radius is eigenvector} we have that for every $\lambda >z_0$ it holds that $ \frac{\langle \mathbb K_\lambda^* F_\lambda , \psi_\mu \rangle }{\langle F_\lambda, \psi_\mu \rangle } = \frac{ \rho(\mathbb K_\lambda) \langle F_\lambda , \psi_\mu \rangle }{\langle F_\lambda, \psi_\mu \rangle} =\rho(\mathbb K_\lambda)$
	and similarly that $ \frac{\langle  F_\lambda \mathbb K_\mu, \psi_\mu \rangle }{\langle F_\lambda, \psi_\mu \rangle }=  \frac{ \rho(\mathbb K_\mu) \langle F_\lambda , \psi_\mu \rangle}{\langle F_\lambda, \psi_\mu \rangle }= \rho(\mathbb K_\mu)$. Hence
	\begin{align*}
	& \rho(\mathbb K_\lambda) -\rho(\mathbb K_\mu) = \frac{\langle \mathbb K_\lambda^* F_\lambda , \psi_\mu \rangle}{\langle F_\lambda, \psi_\mu \rangle}-  \frac{\langle F_\lambda \mathbb K_\mu , \psi_\mu \rangle }{\langle F_\lambda, \psi_\mu \rangle}=  \frac{\langle \left( \mathbb K^*_\lambda -\mathbb K^*_\mu \right) F_\lambda, \psi_\mu \rangle }{\langle F_\lambda, \psi_\mu \rangle } \\
	&\leq \| \mathbb K^*_\lambda -\mathbb K^*_\mu \|_{{X}^*} =  \| \mathbb K_\lambda -\mathbb K_\mu \|_1. 
	\end{align*}
	Therefore, if we prove that $ \lim_{\lambda \rightarrow \mu} \| \mathbb K_\lambda -\mathbb K_\mu \|_1=0 $, then we deduce that $\lambda \mapsto \rho(\mathbb K_\lambda)$ is continuous.
	Since for every $x_1 ,x_2 \geq 0$ we have $|e^{-x_1} - e^{-x_2}|\leq |x_1-x_2 |$, then 
	\begin{align*}
	& \int_{\Omega_0} |\mathbb K_\lambda f(x) - \mathbb K_\mu f(x) | dx \leq \int_{\Omega_0} \int_0^\infty | e^{-\lambda a }-  e^{-\mu a }|  \left| \tilde{K}(a) f \right| (x) dx  \\
	& \leq \int_0^\infty | e^{-\lambda a }-  e^{-\mu a }|  \| \tilde{K}(a) f \|_{1} da  \leq  |\lambda -\mu |  \| f\|_{X}  \int_0^\infty a e^{z_0 a } da. 
	\end{align*} 
	Hence $\lambda \mapsto \mathbb K_\lambda $ is continuous and the desired conclusion follows.
\qed\end{proof} 

\begin{proposition}\label{prop:caract}
	Let $\mathbb K_\lambda$ be compact and non-supporting for every $\lambda \in \Delta \cap \mathbb R$.  
	If $R_0\geq 1$, then there exists a unique $r \geq 0$ such that 
	\begin{equation}\label{equazione caratteristica}
	\rho(\mathbb K_{r})=1. 
	\end{equation}
	If $R_0 <1 $ and there exists a $z \in [z_0, 0)$ such that $\rho(\mathbb K_z )\geq 1$, then there exists a unique $r<0$ such that \eqref{equazione caratteristica} holds. 
\end{proposition}
\begin{proof}
	The map $\lambda \mapsto \rho(\mathbb K_\lambda) $ is decreasing and continuous.
	First of all, in both cases, $R_0 <1 $ and $R_0 \geq 1 $, we have that $\rho(\mathbb K_\lambda)\rightarrow 0$ as $\lambda \rightarrow \infty$. To see this, it is enough to notice that
	\begin{align*}
	0 \leq \rho(\mathbb K_\lambda) \leq \| \mathbb K_\lambda \|_{op} \rightarrow 0 \text{ as  } \lambda \rightarrow \infty. 
	\end{align*} 
	If $R_0 \geq 1 $, then, by the comparison theorem of linear operators, i.e., Proposition \ref{prop 10.12}, the definition of $R_0$ and the continuity of $\lambda \mapsto \rho(\mathbb K_\lambda)$ we deduce that there exists a $r\geq 0$ such that \eqref{equazione caratteristica} holds. 
	
	When $R_0 < 1$ the proof is similar.
\qed\end{proof}
\begin{proof}[Proof of Theorem \ref{thm:eigencouple}]
	The proof is a direct consequence of Proposition \ref{prop:caract} and of the fact that the spectral radius of $\mathbb K_\lambda$ is an eigenvalue when positive. 
\qed\end{proof}

\subsection{Large time behaviour of the density}\label{sec:asympt AC} 

Most of the results of this section hold thanks to the assumption that the discounted next generation operator $\mathbb K_\lambda$ is non-supporting for real values of $\lambda.$ 

The aim of this section is to prove that the unique real eigensolution of \eqref{eigenproblem} is attracting. Namely we aim at proving the following theorem.

\begin{theorem} \label{thm:asymptotic behaviour AC}
	Assume that for every $\lambda \in \Delta \cap \mathbb R$ the operator $\mathbb K_\lambda $ is non-supporting and that its complexification $\mathbb K_\lambda$ is compact for every $\lambda \in \Delta$.
	Additionally, assume that if $\lambda \in \Sigma $, where $\Sigma $ is given by \eqref{sigma}, and $\lambda \neq r $, then $\Re \lambda < r.$
	Let $(r, \psi_r) $ be the unique real normalised eigencouple solving \eqref{eigenproblem}. 
	Then there exists $v>0$ such that
	\[
	\| e^{-r t } b(t) - c \psi_r(\cdot) \|_{1} \leq L e^{- v t} , \quad t >0
	\]
	for some constants $L,c >0 $. 
\end{theorem}
To prove Theorem \ref{thm:asymptotic behaviour AC} we follow the approach presented by Heijmans in \cite{heijmans1986dynamical}.
We need to prove that the exponential solution of the form \eqref{expo ansatz} with $\lambda=r$ and $\psi=\psi_r$, is attracting.
We can divide the proof in the following main steps.
\begin{enumerate} 
	\item We prove that $(I-\mathbb K_\lambda )^{-1}$ is meromorphic on the half-plane $ \Delta$ and that it has a pole of order $1$ in $\lambda=r$. The residue has the form: $R_{-1} \psi= C(\psi) \psi_r$ where $C(\psi)>0.$ 
	\item We prove that there exists a spectral gap: there exists an $\varepsilon$, with $0<\varepsilon < r $, such that $\Re \lambda \leq r-\varepsilon $ for every $\lambda \in \Sigma.$ To this end we apply the Riemann Lebesgue Lemma (i.e. Lemma \ref{lemma:rl}), and we exploit the fact that for every $\lambda \in \Sigma $ we have that $\Re \lambda < r $. The results proven in step 1 will also be used.
	\item We apply the Laplace transform inversion theorem (i.e. Lemma \ref{lem:laplace inversion}), to deduce the behaviour of $b$. To this end we apply some results of complex analysis (such as the Cauchy Theorem) and the results of the previous steps.
	\end{enumerate}
Step 1 is made in Proposition \ref{prop:mero} and \ref{prop:residue}. Step 2 is made in Lemma \ref{cor:epsilon}. Finally step 3 is made in Proposition \ref{prop:hardy space} and Theorem \ref{thm:asymptotic behaviour AC}.

\begin{proposition}\label{prop:mero}
Assume that the positive operator $\mathbb K_\lambda$ is compact for every $\lambda \in \Delta$.
The function $\lambda \mapsto (I- \mathbb K_\lambda)^{-1}$ is meromorphic on $ \Delta.$
\end{proposition}
 To prove this proposition we use the following result due to Steinberg and proven in  \cite{steinberg1968endomorphisms}.
\begin{theorem}\label{theorem 6.3}
	Let	$\Gamma$ be a subset  of the  complex plane  which  is  open  and  connected. If $\{T(\lambda): \lambda \in \Gamma \}$ is an analytic family of compact operators  on  $X^{\mathbb C}$, then  either $I-T(\lambda)$ is  nowhere invertible in  $\Gamma$ or $(I-T(\lambda))^{-1}$ is meromorphic in $\Gamma.$
	\end{theorem}
\begin{proof}[Proof of Proposition \ref{prop:mero}]
 From the the definition of $\mathbb K_\lambda$ we know that
\[
\| \mathbb K_\lambda \|_{op} \leq \| \mathbb K_{\Re\lambda} \|_{op}  \rightarrow 0 \text{ as } \Re \lambda \rightarrow \infty.
\] 
Hence $(I-\mathbb K^{\mathbb C}_\lambda)$ is invertible for $\Re \lambda $ large enough.
Since $\mathbb K_\lambda $ is compact for each real $\lambda $ with $\lambda >z_0$ and since $\lambda \mapsto \mathbb K_\lambda$ is analytic, we can apply Theorem \ref{theorem 6.3} to deduce that $\lambda \mapsto (I- \mathbb K_\lambda)^{-1}$ is meromorphic.
\qed\end{proof}

\begin{lemma}\label{cor:epsilon} 
Let $\mathbb K_\lambda$ be non-supporting for every $\lambda \in \Delta \cap \mathbb R$ and compact for every $\lambda \in \Delta$. Let $r$ be the Malthusian parameter.
	Moreover, assume that if $\lambda \in \Sigma$, where $\Sigma $ is given by \eqref{sigma}, and $\lambda \neq r$, then $\Re \lambda < r$.
There exists an $\varepsilon >0$ such that for every $\lambda \in \Sigma $ with $\lambda \neq r $, $\Re \lambda \leq r-\varepsilon$. 
\end{lemma}
The following lemma is taken from \cite{phillips1957functional}, see Theorem 6.4.2, and will be applied to prove Lemma \ref{cor:epsilon}.
\begin{lemma}[Riemann-Lebesgue]\label{lemma:rl}
	Let  $f \in L^1((0,\infty), X^{\mathbb C})$ and let $\hat{f}$ be its Laplace transform.
	Then $\lim_{|\eta| \rightarrow \infty } \hat f(\xi+i\eta)= 0$, uniformly 
	for $\xi$ in  bounded closed  subintervals  of $(0,\infty)$.
\end{lemma}
\begin{proof}[Proof of Lemma \ref{cor:epsilon}]
%	The function $\lambda \mapsto \left( I-\mathbb K_\lambda \right)^{-1}$ is a meromorphic function for every $\lambda \in \mathbb C$ with $\Re \lambda > z_0$ and it has a pole in $\lambda =r.$ As a consequence, there is a neighbourhood of $r$ in $\mathbb C$ such that the function $\lambda \mapsto (I-\mathbb K_\lambda)^{-1}$ is analytic.
Thanks to Lemma \ref{lemma:rl} we have that for every $ \overline r < r $ there exists an $\eta_0>0$ such that $\| \mathbb K_{s + i \eta } \|_{op}< 1$, hence $\left(I-\mathbb K_{s + i \eta }\right)^{-1}$ is analytic, for every $s \in [\overline r , r ]$ and $|\eta| > \eta_0 $. 
Since the function $\left(I - \mathbb K_\lambda \right)^{-1}$ is meromorphic, we deduce that the number of its poles contained in the compact set $\{ \lambda \in \mathbb C\ : | Im \lambda| \leq \eta_0 \text{ and } \Re \lambda \in [\overline r, r] \} $ is finite. 
This implies that the set $\Sigma \cap \{ \lambda \in \mathbb C \ : \Re \lambda \in [\overline r, r]\} $ has a finite number of elements.
Thanks to the assumption of Lemma \ref{cor:epsilon}, the only $\lambda \in \Sigma $ with $\Re \lambda =r $ is $\lambda =r $. From this we deduce that there exists an $\varepsilon >0$ such that $\Re \lambda \leq r - \varepsilon$ for every $\lambda \in \Sigma$ with $\lambda \neq r$. 
	\qed\end{proof} 
Since, from Proposition \ref{prop:mero}, we know that $\mathbb K_\lambda$ is analytic in a neighbourhood of $r$ we can write its Taylor expansion: 
\begin{equation}\label{Taylor}
\mathbb K_\lambda= \sum_{n=0}^\infty {\left( \lambda - r \right)}^n K_n. 
\end{equation}
Moreover, the map $\mathcal R_\lambda=(I- \mathbb K_\lambda)^{-1}$ can be represented by a Laurent series around the pole $r$ of order $p \geq 1$: 
\begin{equation} \label{Laurent}
\mathcal R_\lambda= \sum_{n=-p}^\infty {\left( \lambda - r \right)}^n R_n.
\end{equation}
\begin{proposition}\label{prop:residue} 
Let $\mathbb K_\lambda$ be non-supporting for every $\lambda \in \Delta \cap \mathbb R$ and let $\mathbb K_\lambda$ be compact for every $\lambda \in \Delta$.
Moreover assume that if $\lambda \in \Sigma$ and $\lambda \neq r$, then $\Re \lambda < r$.
 Let $r$ be the Malthusian parameter, $\psi_r$ be the stable distribution, $F_r$ be the corresponding dual eigenfunction. 
The function $ \lambda \mapsto {(I-\mathbb K_\lambda)}^{-1}$ has a pole of order one in $\lambda=r$ and the residue, $R_{-1}$, is given by 
\[
R_{-1} \psi=\frac{\langle F_r , \psi\rangle }{ \langle F_r, - K_{1} \psi_r \rangle }  \psi_r,  \quad \psi \in X.
\]
\end{proposition} 
\begin{proof}
The proof of this proposition is the same as the proof of Theorem 7.1 in \cite{heijmans1986dynamical}.
The fact that
\begin{equation}\label{identity}
\mathcal R_\lambda(I-\mathbb K_\lambda)=(I-\mathbb K_\lambda) \mathcal R_\lambda =I,
\end{equation}
 implies, together with \eqref{Taylor} and \eqref{Laurent} that 
 \begin{equation} \label{R zero}
 R_{-p}(I-K_0)=(I-K_0) R_{-p}=0. 
 \end{equation}
From \eqref{R zero} and from the fact that $K_0=\mathbb K_r $ we deduce that the range of $R_{-p}$ is equal to
$\{\gamma \psi_r : \gamma \in \mathbb C \}$. Similarly we deduce that the range of $R^*_{-p}$, which is the dual operator of $R_{-p}$, is equal to
$\{\gamma F_r : \gamma \in \mathbb C \}$. 
As a consequence there exist $\Phi$ and $H$ solving 
\begin{equation} \label{eq H F} 
R_{-p} \Phi=\psi_r \text{ and } 
R^*_{-p} H=F_r
\end{equation} 
 respectively. 

From the identity \eqref{identity} and formula \eqref{Laurent} and \eqref{Taylor} we can also deduce that if $p>1$, then 
\[
-R_{-p} K_1 + R_{-p+1}(I-K_0)=0.
\]
While, if $p=1$ we have that
\[
-R_{-1} K_1 + R_0(I-K_0)=I.
\]
Combining these two last equations with \eqref{R zero} we deduce that if $p>1$ then 
\begin{equation}\label{p maggiore a 1}
R_{-p}K_1 R_{-p}=0
\end{equation}
while if $p=1$
\begin{equation}\label{p uguale a 1}
R_{-1}K_1 R_{-1}=-R_{-1}.
\end{equation}
As a consequence of \eqref{p maggiore a 1} and \eqref{eq H F}, if $p>1$, then \[
\langle F_r, K_1 \psi_r \rangle=\langle R^*_{-p} H, K_1 R_{-p} \Phi \rangle =\langle H, R_{-p} K_1 R_{-p} \Phi \rangle=  0,
\] which is a contradiction with the fact that $F_r$ strictly positive and $- K_1 \psi_r=\left[ - \frac{d}{d\lambda} \mathbb K_\lambda \right]_{r} \psi_r $ is positive. Hence $p=1.$

Now let $R_{-1}\psi=f(\psi)\psi_r$ for some linear functional $f.$
Using the fact that $F_r = R_{-1}^* H_r $ and \eqref{p uguale a 1} we deduce that 
\begin{align*}
&\langle F_r , \psi  \rangle = \langle  R_{-1}^* H , \psi  \rangle =
\langle H ,   R_{-1}\psi  \rangle =
\langle  H ,  -R_{-1}K_{1} R_{-1} \psi  \rangle= \\
&\langle R_{-1}^* H ,  -K_{1} \left( f(\psi) \psi_r \right)  \rangle 
= f(\psi) \langle R_{-1}^* H ,  -K_{1}  \psi_r   \rangle=f(\psi) \langle F_r ,  -K_{1}  \psi_r   \rangle
\end{align*}
it follows that $f(\psi)= \frac{\langle F_r , \psi  \rangle}{ \langle F_r ,  -K_{1}  \psi_r   \rangle}$.
\qed \end{proof}
\begin{definition} \label{def:hardy} 
The Hardy-Lebesgue class $H_1(\alpha, X^{\mathbb C})$ is the class of functions $g: \mathbb C \rightarrow X$, which are analytic in $\Re \lambda > \alpha$ and satisfy the following conditions 
\begin{equation}\label{hardy bound}
\sup_{ \xi > \alpha} \int_{-\infty}^\infty  \| g(\xi+ i \eta)  \|_1 d \eta < \infty 
\end{equation}
and $g(\alpha+i \eta)=\lim_{\xi  \rightarrow \alpha} g(\xi+ i \eta) $ exists a.e. and is an element of $L^1((-\infty, \infty), X)$.
\end{definition} 
\begin{proposition}\label{prop:hardy space}
Let $\mathbb K_\lambda$ be compact for every $\lambda \in \Delta$ and $\mathbb K_\lambda$ non-supporting for every $\lambda \in \Delta \cap \mathbb R$.
Moreover, assume that if $\lambda \in \Sigma$ and $\lambda \neq r$, then $\Re \lambda < r$.
 Let $r$ be the Malthusian parameter. Then 
$\hat{b} \in H_1(\alpha, X^{\mathbb C})$ if $\alpha > r.$
\end{proposition}

\begin{proof}[Proof of Proposition \ref{prop:hardy space}]
For each fixed $\xi > z_0$ the map 
\[
\eta \mapsto \hat{b_0}(\xi + i \eta ) \quad 
\]
belongs to $L^1((-\infty, \infty), X^{\mathbb C})$, see for instance Theorem 6.3.2 in \cite{phillips1957functional}. From Lemma \ref{lemma:rl} we know that there exists a $\eta_0$ such that if $ |\eta| \geq \eta_0 $ then 
\[
\left\| \left(I - \mathbb K_{\xi+i\eta}\right)^{-1} \right\|_{op} \leq 2. 
\]

Since when $\xi > r $ the function  
$ \eta \mapsto \left(  I - \mathbb K_{\xi+i\eta}\right)^{-1} $ is continuous on the compact set $[-\eta_0, \eta_0]$ it follows that, if $\xi > r$  there exists a constant $C(\xi)>0$ such that \[
\left\| \left(I - \mathbb K_{\xi+i\eta}\right)^{-1} \right\|_{op}\leq C(\xi) \] 
 for all $\eta \in \mathbb R .$ 
Since $\hat{b}$ is given by \eqref{lap transform of b} we deduce, that 
\begin{equation*}
\|\hat{b}(\xi+i \eta )\|_1  \leq C(\xi) \|\hat{b_0}(\xi + i \eta ) \|_1 \quad \text{ for } \xi > r \text{ and } \eta \in \mathbb R. 
\end{equation*}
As a consequence of the positivity of $b$ and $\hat{b_0}$, we have that for every $\xi \geq \alpha> r$
\begin{equation}\label{ineq b b0}
\|\hat{b}(\xi+i \eta )\|_1 \leq \|\hat{b}(\alpha+i \eta )\|_1  \leq C(\alpha) \|\hat{b_0}(\alpha + i \eta ) \|_1
\end{equation}
Hence  $\|\hat{b}(\xi+i \eta )\|_1 $ is integrable with respect to $\eta $ over $(-\infty, \infty)$ and thanks to \eqref{ineq b b0} we deduce that $\hat{b}$ satisfies \eqref{hardy bound}.

Since the maps $\lambda \mapsto (I- \mathbb K_\lambda)^{-1}$ and $\lambda \mapsto \hat{b_0}(\lambda)$ are analytic when $\Re \lambda > r$ we deduce that the map $\lambda \mapsto \hat{b}(\lambda)$ is analytic for $\Re \lambda > r$. 
Hence the limit of $\hat{b}(\xi + i \eta )$ as $\xi \rightarrow \alpha $ exists and is equal to $\hat{b}(\alpha +i \eta ) $. 
The fact that $\hat{b}(\alpha +i \cdot ) \in L^1((-\infty,\infty ), X^{\mathbb C})$ follows from inequality \eqref{ineq b b0}.
\qed\end{proof}
The following lemma, useful for the proof of Theorem \ref{thm:asymptotic behaviour AC}, is taken from \cite{friedman1968nonlinear}.
\begin{lemma}\label{lem:laplace inversion}
	Let $\hat{g} \in H_1(\alpha,X^{\mathbb C})$,(cf. Definition \ref{def:hardy}) then the function 
	\[
	f(t)=\frac{1}{2\pi i }  \lim_{T \rightarrow \infty } \int_{\gamma -iT}^{\gamma +iT} e^{\lambda t} \hat{g}(\lambda) d\lambda \quad \gamma \geq \alpha
	\]
	is well defined for every $t \in \mathbb R$, and does not depend on $\gamma$. Moreover, $f(t)=0$ if $t<0$, while $f$ is continuous in $t$ for $t>0$. Finally $\hat{f}(\lambda)=\hat{g}(\lambda).$
\end{lemma}
\begin{proof}[Proof of Theorem \ref{thm:asymptotic behaviour AC}] 
Also in this case, the proof is very similar to a proof in \cite{heijmans1986dynamical}, viz. the proof of Corollary 8.3. 
We write the main steps here. 

Since the function $\hat{b}$ belongs to $H_1(\alpha,X^{\mathbb C}) $ for every $\alpha > r $, we deduce, by Lemma \ref{lem:laplace inversion} and by the uniqueness of the Laplace transform \cite[Theorem 6.2.3]{phillips1957functional} that 
\begin{align}\label{inverse laplace}
b(t)= \frac{1}{2\pi i } \int_{\alpha - i \infty }^{\alpha + i \infty } e^{\lambda t} \hat{b}(\lambda) d\lambda. 
\end{align} 
Consider $0 < v < \varepsilon $ where $\varepsilon $ is given by Lemma \ref{cor:epsilon} and
notice that 
\begin{align}\label{closed int}
 &\int_{\alpha - i T }^{\alpha + i T } e^{\lambda t} \hat{b}(\lambda)=  \oint_{\Gamma} e^{\lambda t} \hat{b(\lambda)} d\lambda -  \lim_{T \rightarrow \infty } \int_{\Gamma_3}  e^{\lambda t} \hat{b}(\lambda) d\lambda
  -  \lim_{T \rightarrow \infty } \int_{\Gamma_2}  e^{\lambda t} \hat{b}(\lambda) d\lambda \\
 & -  \lim_{T \rightarrow \infty } \int_{\Gamma_4}  e^{\lambda t} \hat{b}(\lambda) d\lambda \nonumber
\end{align}
where $\Gamma:=\cup_{i=1}^4 \Gamma_i$ and $\Gamma_1$ is the segment in the complex plan connecting the point $\alpha -i T $ to $\alpha + iT $, $\Gamma_2$ is the segment connecting $\alpha + i T $ with $r-v + iT $, $\Gamma_3 $ is the segment connecting 
$r-v + iT $ with $r-v -iT $ and, finally, $\Gamma_4 $ is the segment connecting $r-v-iT$ with $\alpha - iT.$
 
From the Cauchy theorem for vector valued functions (\cite{phillips1957functional}), equality \eqref{closed int} and Lemma \ref{lemma:rl} and the Laplace inversion formula \eqref{inverse laplace}, we deduce that
\begin{align*}
b(t)=\frac{1}{2\pi i } \oint_{\Gamma} e^{\lambda t} \hat{b(\lambda)} d\lambda   + \frac{1}{2\pi i }  \lim_{T \rightarrow \infty } \int_{r - v -iT}^{r - v +iT} e^{\lambda t} \hat{b}(\lambda) d\lambda.
\end{align*} 

\begin{figure}[H] 
\centering
\includegraphics[scale=1]{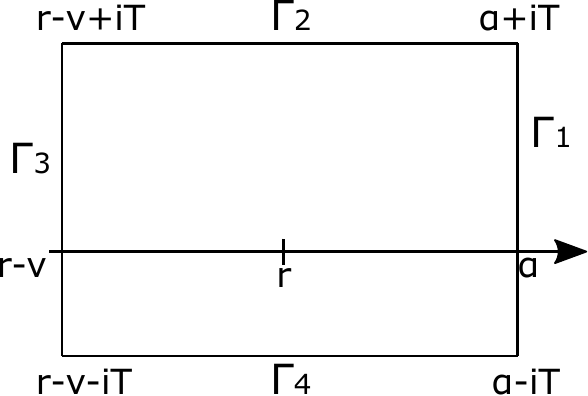}
\caption{Graphic representation of the set $\Gamma$}\label{fig:residue}
\end{figure}

By the residue theorem we have that
\begin{align*}
\frac{1}{2\pi i } \oint_{\Gamma} e^{\lambda t} \hat{b(\lambda)} d\lambda={Res}_{\lambda=r}{ e^{\lambda t} \hat{b}(\lambda)}= e^{r t } R_{-1} \hat{b_0}(r) = e^{r t} \frac{\langle F_r, \hat{b_0}(r) \rangle }{ \langle F_r, - K_1 \psi_r \rangle } \psi_r.
\end{align*}

We conclude the proof by noting that
\begin{align*}
\left\| \frac{1}{2\pi i }  \lim_{T \rightarrow \infty } \int_{r- v -iT}^{r- v +iT} e^{\lambda t} \hat{b}(\lambda) d\lambda  \right\|_{1} \leq M e^{(r-v) t }
\end{align*}
with 
\[
M = \frac{1}{2 \pi} \int_{- \infty}^\infty \left\|  \hat{b}(r- v+ i \eta) \right\|_1 d \eta 
\]
and explaining why 
\[
\int_{- \infty}^\infty \left\|  \hat{b}(r- v+ i \eta) \right\|_1 d \eta < \infty. 
\]
As in the proof of Proposition \ref{prop:hardy space} we know that thanks to Lemma \ref{lemma:rl}, there exists a $\eta_0>0$ and a constant $C>0$ such that 
\[
\| (I-\mathbb K_{r-v + i \eta} )^{-1} \|_{op} \leq C  
\]
for every $\eta$ with $|\eta |> \eta_0.$
On the other hand since the function $\eta \mapsto (I -\mathbb K_{r-v + i \eta })^{-1} $ is continuous on the compact set $[-\eta_0, \eta_0]$
we deduce that for every $\eta \in \mathbb R$
\[
\|(I-\mathbb K_{r-v + i \eta} )^{-1} \|_{op} \leq C .
\]
By equality \eqref{lap transform of b} we deduce that 
\[
\|\hat{b}(r-v + i \eta ) \|_1 \leq C \|\hat{b_0}(r-v + i \eta ) \|_1. 
\]
Since from the proof of Proposition \ref{prop:hardy space} we know that for every $\xi > z_0 $ the function $\eta \mapsto \hat{b_0}(\xi + i \eta)$ belongs to  $L^1((- \infty , \infty ) , X^\mathbb C))$ we deduce, possibly adjusting $C$, that the function $\eta \mapsto \hat{b}(r-v + i \eta )$ belongs to $L^1((- \infty , \infty ),   X^\mathbb C)). $

\qed\end{proof}

We next present here two sufficient conditions on $\tilde{K}$ that guarantee that for every $\lambda \in \Sigma$, with $r \neq \lambda,$ we have that $\Re \lambda < r$. 
As will be shown in Section \ref{sec:examples}, either Assumption \ref{rotation} or Assumption \ref{rot density} can be easily checked for all the model examples we consider. 
\begin{assumption}\label{rotation} 

\eu{There exists a measurable function $\gamma: \mathbb R_+ \times \Omega_0 \rightarrow \mathbb R $ such that for every $x \in \Omega_0$ the function $a \mapsto \gamma(a,x)$ is piecewise monotone and there exists a measurable function $c: \mathbb R_+ \times \Omega_0 \rightarrow \mathbb R_+$ such that 
\[
\tilde{K}(a) \varphi(\cdot) = c(a, \cdot) \varphi(\gamma(a, \cdot))  \quad \forall \varphi \in X_+ \text{ and } \forall a  \in \mathbb R_+.
\]}
\end{assumption}
\begin{assumption}\label{rot density}
There exists a function $\tilde{k}: \mathbb R_+ \times \Omega_0  \times \Omega_0 \mapsto \mathbb R_+$, that is measurable in each variable and such that 
\[ \sup_{a \in [0,T]} \sup_{x \in \Omega_0}\int_{\Omega_0} \tilde{k}(a,x,y) dy< \infty
\] 
and such that 
\[
(\tilde{K}(a)\varphi)(y):= \int_{\Omega_0} \tilde{k}(a,x,y) \varphi (x) dx \quad \forall \varphi \in X \text{ and }  \forall a  \in \mathbb R_+.
\]
\end{assumption}

The aim of the following proposition is to show that each of the preceding two assumptions guarantees that $\Re \lambda < r$ for every $\lambda \in \Sigma$ with $\lambda \neq r$.

\begin{proposition}\label{lambda_d largest}
	Let $\mathbb K_\lambda$ be compact for every $\lambda \in \Delta$ and non-supporting for every $\lambda \in \Delta \cap \mathbb R$. 
	Let $\tilde{K}$ satisfy either Assumption \ref{rotation} or Assumption \ref{rot density}.
	Let $r$ be the Malthusian parameter.
	If $\lambda \in \Sigma$, where $\Sigma $ is given by \eqref{sigma}, and $\lambda \neq r$, then $\Re \lambda < r$.
\end{proposition}
To prove Proposition \ref{lambda_d largest} we need the following Theorem, which corresponds to Theorem 1.39 in \cite{rudin2006real}.
\begin{theorem} \label{thm:rudin} 
	Let $\varphi \in X^{\mathbb C}$ and assume that 
	\[
	\int_{\Omega_0} |\varphi (x)| dx = \left| \int_{\Omega_0} \varphi (x) dx \right| 
	\]
	then there exists a constant $\beta $ such that $\beta \varphi = |\varphi |$ a.e. on $\Omega_0.$
\end{theorem}
\begin{proof}[Proof of Proposition \ref{lambda_d largest}]
	The proof is an adaptation of the proof of Theorem 6.13 in\ \cite{heijmans1986dynamical}.

	Assume that there exists a $\lambda \in \Sigma$ with $\lambda \neq r$ such that $\mathbb K_\lambda \psi=\psi$ for some $\psi \in X^{\mathbb C}.$ It follows that 
	\begin{equation}\label{eq:psi}
	|\psi|= |\mathbb K_{\lambda}\psi|\leq  \mathbb K_{\Re \lambda}|\psi|.
	\end{equation}
	Taking duality parings with $F_{\Re\lambda }$ on both sides of the inequality we deduce that 
	\[
	\langle F_{\Re \lambda}, |\psi | \rangle \leq \rho(\mathbb K_{\Re\lambda})  \langle F_{\Re \lambda}, |\psi | \rangle.
	\] 
	This implies that $\rho(\mathbb K_{\Re_\lambda}) \geq 1= \rho(\mathbb K_r).$ Since, by Proposition \ref{prop:Klambda decreasing} we know that  the function $\mu \rightarrow \rho(\mathbb K_\mu)$ is decreasing when $\mu$ varies in $\mathbb R$ we deduce that $\Re\lambda \leq r.$
	
	Assume now that $\Re \lambda = r$, hence, since $\lambda \neq r$, it must hold that $Im \lambda >0$.
	From \eqref{eq:psi} we know that $ \mathbb K_{r}|\psi|\geq |\psi|.$
	If we assume that 
	$ \mathbb K_{r}|\psi| \neq |\psi|$ 
	then taking duality parings with $F_r$ we deduce that 
	$
	\langle F_r ,|\psi|\rangle  > \langle  F_r ,|\psi|\rangle.
	$
	This is a contradiction, hence it must hold that $ \mathbb K_{r}|\psi| =|\psi|$. 
	Since $(r, \psi_r)$ is the unique (up to normalization) solution of the non-linear eigenproblem \eqref{eigenproblem}, we deduce that there exists an $\ell: \Omega_0 \rightarrow \mathbb R$ such that $\psi(x)=e^{i \ell(x)} \psi_r(x)$ and, as a consequence of \eqref{eq:psi}, we have that $\mathbb K_r \psi_r = | \mathbb K_\lambda \psi |$ i.e. 
	\begin{align}\label{differentiate assumptions} 
	\int_0^\infty e^{- r a} \tilde{K}(a) \psi_r da = \left| \int_0^\infty e^{- r a} e^{i a Im \lambda } \tilde{K}(a) e^{i \ell(\cdot)}\psi_r da \right|
	\end{align}

If we make Assumption \ref{rotation} this implies that
	 	\begin{align*}
	 \int_0^\infty e^{- r a} \tilde{K}(a) \psi_r da = \left|   \int_0^\infty e^{- r a} e^{i a Im \lambda } e^{i \ell(\gamma(a,\cdot))}\tilde{K}(a) \psi_r da \right|
	 \end{align*}
	  Since 
	  \begin{align*}
	  	 \int_0^\infty e^{- r a} \tilde{K}(a) \psi_r da =   \int_0^\infty \left| e^{- r a} e^{i a Im \lambda } e^{i \ell(\gamma(a,\cdot))}\tilde{K}(a) \psi_r \right| da 
	  \end{align*}
	  we deduce, by Theorem \ref{thm:rudin}, that 
	there exist a $\beta \in \mathbb R$ such that
	\[
	a Im \lambda + \ell(\gamma(a,x) )= \beta.
	\]
	As a consequence we have that 
	\begin{align*}
	& e^{i \ell(x) }\psi_r (x)= \int_0^\infty e^{-\left(\Re \lambda+ i Im \lambda \right) a } \tilde{K}(a) \psi(x) da \\
	&= \int_0^\infty e^{- a \Re \lambda }  e^{i\beta - i \ell(\gamma(a,x)) } \tilde{K}(a) \psi(x) da  \\
	& =e^{i\beta} \int_0^\infty e^{- a \Re \lambda }   \tilde{K}(a) \psi_r(x) da = e^{i\beta} \psi_r(x).
	\end{align*}
	This implies that $\ell(x)=\beta(\mod 2 \pi)$ for a.e. $x \in \Omega_0$ and hence, the piecewise monotonicity of $\gamma (\cdot,x) $  implies that $Im \lambda =0$.
This is a contradiction and the desired conclusion follows. 

If, instead, we make Assumption \ref{rot density} equality \eqref{differentiate assumptions} implies that
\begin{align*}
&	\int_0^\infty e^{- r a} \int_{\Omega_0} \tilde{k}(a,y,x) \psi_r(y) dy da \\
&= \left| \int_0^\infty  e^{- r a}  \int_{\Omega_0} e^{i a Im \lambda + i \ell(y)}   \tilde{k}(a,y,x)  \psi_r(y) dy da \right|
\end{align*}

Since 
\begin{align*}
&	\int_0^\infty e^{- r a} \int_{\Omega_0} \tilde{k}(a,y,x) \psi_r(y) dy da \\
&=\int_0^\infty  \int_{\Omega_0} \left| e^{- r a}   e^{i a  Im \lambda + i \ell(y)}   \tilde{k}(a,y,x)  \psi_r(y)  \right| dy da,
\end{align*} 
we deduce by Theorem \ref{lemma:rl} that there exists a $\beta \in \mathbb R$ such that 
\[
	a Im \lambda + \ell(x )= \beta.
\]

As a consequence we have that 
	\begin{align*}
	& e^{i \ell(x)} \psi_r (x)= \int_0^\infty e^{-\left(\Re \lambda+ i Im \lambda \right) a } \tilde{K}(a) \psi(x) da \\
	&= \int_0^\infty e^{- a \Re \lambda }  e^{i\beta - i \ell(x) } \tilde{K}(a) \psi(x) da  \\
	& =e^{i\beta} \int_0^\infty e^{- a \Re \lambda }   \tilde{K}(a) \psi_r(x) da = e^{i\beta} \psi_r(x).
	\end{align*}
	This implies that $\ell(x)=\beta(\mod 2 \pi)$ for a.e. $x \in \Omega_0$ and hence that $Im \lambda =0$. This is a contradiction and the desired conclusion follows. 
\qed\end{proof}

\subsection{An alternative approach} \label{sec:gripenberg} 
In this section we present an alternative approach for proving Theorem \ref{thm:asymptotic behaviour AC}.
We make the same assumptions on $\tilde{K}$ and $b_0$ as made in Section \ref{sec:asympt AC} and we keep the same notation. 
 
We plan to deduce the asymptotic behaviour of the solution of equation \eqref{eq b} by obtaining estimates on the resolvent operator from the following theorem, which is Theorem 2 in \cite{gripenberg1987asymptotic}. 
\begin{theorem}\label{thm:grip}
Let $w \in \mathbb R$ and let $ \tilde{K} \in L_{- w}^1(\mathbb R_+, \mathcal L(X^{\mathbb C}))$ be an operator kernel. Then its resolvent $\tilde{R}$ belongs to $ L_{-w}^1(\mathbb R_+, \mathcal L(X^{\mathbb C}))$ if and only if $I- \mathbb K_\lambda $ is invertible for every $\lambda \in \mathbb C $ such that $\Re \lambda \geq w $.
\end{theorem}

To be able to apply this theorem, we have to assume that the function $a \mapsto \tilde{K}(a)$ is measurable with respect to the topology induced by the operator norm on $ \mathcal L(X).$ 
This measurability assumption is stronger than the measurability assumption we ask for the operator kernels in Section \ref{sec:RE for density}.

\begin{theorem} \label{thm:asymptotic behaviour AC Griep}
	Assume that for every $\lambda \in \Delta \cap \mathbb R$ the operator $\mathbb K_\lambda $ is non-supporting and that its complexification $\mathbb K_\lambda$ is compact for every $\lambda \in \Sigma$. 
 Additionally assume that $\tilde{K}: \mathbb R_+ \rightarrow \mathcal L(X^{\mathbb C})$ is measurable and satisfies either Assumption \ref{rotation} or Assumption \ref{rot density}.
	Let $(r, \psi_r) $ be the eigencouple solving \eqref{eigenproblem}. 
	Then there exists $v>0$ such that
	\[
	\| e^{-r t } b(t) - c \psi_r(\cdot) \|_{1} \leq L e^{- v t} , \quad t >0
	\]
	for some constants $L,c >0 $. 
\end{theorem}
 \begin{proof} 
We already know that for every $\lambda \in \mathbb C$ with $\Re \lambda \geq \sigma > r$ the operator $I-\mathbb K_\lambda $ is invertible. Moreover $\tilde{K} \in L^1_{- \sigma}(\mathbb R_+, \mathcal L(X^{\mathbb C})) $  for any $\sigma > r$. 
Hence we deduce from Theorem \ref{thm:grip} that  $\tilde{R} \in L^1_{- \sigma }(\mathbb R_+, \mathcal L(X^{\mathbb C}))$. 
We denote the Laplace transform of $\tilde{R}$ as follows 
\[
\mathbb R_\lambda := \int_0^\infty e^{-\lambda a}  \tilde{R}(a) da < \infty \quad \Re \lambda > \sigma. 
\]

Similarly as in Proposition \ref{prop:hardy space} we deduce that $(I- \mathbb K_\lambda)^{-1} \mathbb K_\lambda \hat{b_0}(\lambda) \in H(\sigma, X^{\mathbb C})$ for $\sigma > r$. 
Hence, from the Laplace inversion formula we deduce that 
\[
\int_0^t \tilde{R}(a)b_0(t-a) da =  \frac{1}{2\pi i }  \lim_{T \rightarrow \infty } \int_{\sigma -iT}^{\sigma  +iT} e^{\lambda t} \left( I - \mathbb K_\lambda \right)^{-1} \mathbb K_\lambda \hat{b_0}(\lambda) d\lambda
\] 
if $\sigma >r.$

Consider a $w \in \mathbb R$ with $w < r$ such that the operator $(I- \mathbb K_\lambda)$ is invertible for every $\lambda \in \mathbb C$ with $w \leq \Re \lambda <r $.
This $w$ exists thanks to Lemma \ref{cor:epsilon}. 
Define the operator $Q$ as
\[
\int_0^t Q(a)b_0(t-a) da :=  \frac{1}{2\pi i }  \lim_{T \rightarrow \infty } \int_{w -iT}^{w +iT} e^{\lambda t} \left( I - \mathbb K_\lambda \right)^{-1} \mathbb K_\lambda \hat{b_0} (\lambda)d\lambda. 
\] 

Similarly as in the proof of Theorem \ref{thm:asymptotic behaviour AC} we can deduce, by the residue theorem, that 
\[
\int_0^t \tilde{R}(a)b_0(t-a) da - \int_0^t Q(a) b_0 (t-a) da = e^{rt} R_{-1} \mathbb K_r \hat{b_0}(r)
\]

As a consequence 
\begin{align*}
& b(t)=\int_0^t \tilde{R} (a) b_0(t-a) da + b_0(t) \\
&= \int_0^t Q(a) b_0(t-a) da  + e^{rt} R_{-1} \mathbb K_r \hat{b_0}(r)  + b_0(t) \\
&= \int_0^t Q(a) b_0(t-a) da  +   e^{rt} \frac{\langle F_r, \mathbb K_r \hat{b_0}(r)  \rangle }{\langle F_r, - K_1 \psi_r \rangle } \psi_r + b_0(t) 
\end{align*}

Notice that by the definition of $F_r $ we have that 
\[
\frac{\langle F_r, \mathbb K_r \hat{b_0}(r)   \rangle }{\langle F_r, - K_1 \psi_r \rangle }= \frac{\langle \mathbb K_r^* F_r,  \hat{b_0}(r) \rangle }{\langle F_r, - K_1 \psi_r \rangle }= \frac{\langle  F_r,  \hat{b_0}(r) \rangle }{\langle F_r, - K_1 \psi_r \rangle }
\]

Moreover, by the definition of $Q$ we have $\|\int_0^t Q(a) b_0(t-a) da  \|_1\leq c e^{w t } $ where the constant $c$ is equal to 
\[
c:=\int_{-\infty}^{\infty} \| (I- \mathbb K_{w + i \eta })^{-1} \mathbb K_{ w + i \eta}  \hat{b_0}(w + i \eta ) \|_1 d \eta
\]
The fact that $c< \infty $ follows by an adaptation of the final part of the proof of Theorem \ref{thm:asymptotic behaviour AC}. 
 Hence, using \eqref{bound b0}, it follows that 
\begin{align*}
\left\|e^{-rt}b(t) -  \frac{\langle F_r,  \hat{b_0}(r)  \rangle }{\langle F_r, - K_1 \psi_r \rangle } \psi_r\right\|_1 
\leq c_1 e^{(w-r)t} + e^{-rt } \| b_0(t) \|_1
 \leq c_2 e^{-v t},
\end{align*} 
for some positive constants $v, c_1, c_2$. 
\qed\end{proof}
This approach is not very different from the approach developed in Section \ref{sec:asympt AC}, but here we have to make stronger measurability assumptions on $\tilde{K}$. These correspond to stronger assumptions on the model parameters and therefore we decided to focus on Heijmans' approach.

\subsection{Asymptotic behaviour of the measure-valued solution} \label{sec:asympt measure} 
In this section we deduce the asymptotic behaviour of the measure $B$ from the behaviour of the density of its absolutely continuous component. 
This type of technique has been applied in \cite{diekmann2018waning} and in \cite{pichor2020dynamics}. 

\begin{lemma} \label{b and B}
Let $K$ be a $z_0$-bounded regularizing kernel and let $B_0$ be given by \eqref{b0 function of m0} as a function of $K$ and $M_0 \in \mathcal M_{+,b}(\Omega).$
Let $\tilde{K}: \mathbb R^*_+ \rightarrow \mathcal L(X)$ be the operator defined by 
\begin{equation}\label{eq:ktilde} 
\tilde{K}(a)f = d_{K, f} (a) \text{ for every } a \geq 0 \text{ and } f \in X
\end{equation}
where $d_{K,f}(a) \in X$ is the density of the measure \eqref{int AC} with $t$ replaced by $a$.
Let us denote with $B $ the solution of equation \eqref{RE}. Then 
\begin{equation}\label{b density BAC}
B^{AC}(t, \omega) =  \int_{\omega} b(t)(x) dx \quad \forall \omega \in \mathcal B(\Omega_0)
\end{equation}
where $b$ is the solution of \eqref{eq b} with respect to $\tilde{K}$ and the function $b_0: \mathbb R_+ \rightarrow X $ mapping $t$ to the density of $B_0^{AC}(t, \cdot)+\mathcal L_K(B^s)^{AC}(t, \cdot)$. 
\end{lemma}
\begin{proof}
First of all we need to check that $\tilde{K}$ is an operator kernel, so in particular that for every $f \in X$ the map
\begin{equation}\label{boch meas function}
a \mapsto \tilde{K}(a) f 
\end{equation}
is Bochner measurable. 
Since $X$ is separable, Bochner measurability and weak measurability coincide. So it suffices to show that the map \eqref{boch meas function} is weakly measurable, i.e. that for every $g \in L^\infty(\Omega_0)$ the map 
\[
a \mapsto \int_\Omega g(x) \tilde{K}(a) f (x)  dx = \int_\Omega g(x) d_{K,f}(a) (x)  dx 
\]
is measurable. This is a consequence of the fact that for every $\omega $ the map
\[
a \mapsto \int_\omega d_{K,f}(a) (x)dx = \int_\Omega K(a,y,\omega ) f(y) dy 
\]
is measurable. 
 We refer to \cite{franco2021one} for the details.
 Similarly one shows that $b_0: \mathbb R_+ \rightarrow X$ is Bochner measurable. 

Moreover, thanks to the fact that $K$ is a locally bounded kernel
\[
\sup_{a \in[0,T] } \sup_{ f \in X} \| \tilde{K}(a) f\|_1 < \infty. 
\]
Hence, thanks to Lemma \ref{lem:for b}, equation \eqref{eq b}, with respect to $\tilde{K}$ and $b_0$ has a unique solution $b$. 
Integrating all the terms in the equation over the set $\omega $ we deduce that 
\[
\tilde{B}(t, \omega):= \int_\omega b(t) (x) dx 
\]
is a solution of \eqref{RE AC int}. By uniqueness it follows that $\tilde{B}=B^{AC}$
\qed\end{proof} 

\begin{theorem}\label{cor:measure asympt AC}
Let $K$ be a $z_0$-bounded regularizing kernel such that the operator $\mathbb K_\lambda $, defined by \eqref{NGO lambda} with $\tilde{K}$ given by \eqref{eq:ktilde}, is compact for every $\lambda \in \Delta $ and non-supporting for every $\Delta \cap \mathbb R$.
Assume also that $\tilde{K} $  satisfies either Assumption \ref{rotation} or Assumption \ref{rot density}.
Let us denote with $B $ the solution of equation \eqref{RE} and
let $\Psi_r(dx)=\psi_r(x) dx$, with $\psi_r$ the stable distribution, and $r$ the Malthusian parameter.
Then there exist constants $M, k >0$ such that
\begin{equation}\label{speed of convergence}
\left\| e^{-r t }B(t, \cdot) -  c \Psi_r(\cdot)  \right\| \leq M e^{-k t} \quad \forall t >0.
\end{equation}  
where $c>0$ is the same constant as in Theorem \ref{thm:asymptotic behaviour AC} and $\|\cdot  \|= \| \cdot  \|_{TV} =\| \cdot \|_{\flat}$.
\end{theorem}
\begin{proof}
Since the density of $B^{AC}(t, \cdot) $, $b(t)$, solves \eqref{eq b} with respect to the $\tilde{K}$ and $b_0$ given by Lemma \ref{b and B}, we deduce that 
\begin{align*}
& \left\| e^{-r t } B(t, \cdot) - c \Psi_r(\cdot)  \right\| \leq \left\| e^{-r t } B(t, \cdot) - e^{-r t } B^{AC}(t, \cdot) + e^{-r t } B^{AC}(t, \cdot)- c \Psi_r(\cdot)  \right\|  \\
& \leq \left\| e^{-r t } B(t, \cdot) - e^{-r t } B^{AC}(t, \cdot) \right\|  + \left\| e^{-r t }B^{AC}(t, \cdot) -c \Psi_r(\cdot)  \right\| \\
& \leq e^{-r t } \left\| B^s(t, \cdot)  \right\|  + \left\| e^{-r t } b(t, \cdot) -  c \psi_r(\cdot)  \right\|_1 \leq e^{-r t } \left\| B^s(t, \cdot)  \right\|  + L e^{-v t} \\
& \leq c_1 e^{(z_0-r) t }+ c_2 t e^{(z_0-r) t} + L e^{-v t},
\end{align*} 
where in the last inequality we have applied \eqref{Bs tends to zero exponentially}.
From this chain of inequalities we deduce that \eqref{speed of convergence} holds. 
\qed\end{proof} 

We stress that in Corollary \ref{cor:measure asympt AC} we prove balanced exponential growth and we also provide an exponential estimate of the remainder, as is done in \cite{bertoin2019feynman} to which we refer for yet another approach.

\section{Kernels arising from structured population models}\label{sec:examples}
The aim of this section is to present three classes of $z_0$-bounded regularizing kernels that, as we shall show in the next sections satisfy the assumptions of Theorems \ref{thm:eigencouple} and \ref{thm:asymptotic behaviour AC}, i.e. the corresponding operator kernel $\tilde{K} $ satisfies either Assumption \ref{rotation} or Assumption \ref{rot density} and the corresponding operator
$\mathbb K_\lambda $ is non-supporting for every $\lambda \in \Delta \cap \mathbb R$ and compact for every $\lambda \in \Delta$. 

Since the classes of kernels that we present are motivated by structured population models, we interpret the mathematical  assumptions by describing their meaning in the context of the corresponding models. To help the reader we also provide the more classical PDE formulation of the models in the next section.
\eu{In all of this section we assume that the $i$-state space $\Omega $ is a subset of $\mathbb R_+^*$.} 
\subsection{The kernel as a modelling ingredient} \label{sec:kernel modelling ingr}
The main modelling ingredient of the renewal equation is the kernel, which summarises the effect of the individual level mechanisms determining the population evolution. 
The individual level mechanisms modelled via the renewal equation \eqref{RE} are
\begin{itemize}
\item deterministic smooth development of the individual state, as growth or waning.
\item giving birth, with offspring appearing at a different position (i.e. having a different state), or jumping to another position, in which case we say that the individual in the old state died while an individual in the new state was born.
 We assume that this happens at a position dependent rate $\Lambda$.
\end{itemize}
Therefore we assume that the kernel is 
\begin{equation}\label{kernel}
K(a,\xi,\omega):= \mathcal F(a,\xi) \Lambda(X(a,\xi)) \nu(X(a,\xi), \omega) 
\end{equation}

where 
\begin{itemize}
	\item $X(a,\xi)$ is the state of an individual that survived up to the current time and that, $a$ time ago, had state $\xi $.
\item $\mathcal F(a,\xi)$ is the probability that an individual that $a$ time ago had state $\xi$ survives up to the present time.
\item $\nu(z, \omega)$ denotes the expected number of individuals born with size in $\omega $ when an individual with size $z$ reproduces or jumps.
\end{itemize}
We want to find sufficient conditions on $\mathcal F$, $X$ and $\nu $ that guarantee that $K$ is a $z_0$-bounded regularizing kernel, that the corresponding operator kernel $\tilde{K}$ satisfies either Assumption \ref{rotation} or Assumption \ref{rot density} and that the corresponding operator $\mathbb K_\lambda $ is compact for every $\lambda \in \Delta $ and non-supporting for every $\lambda \in \Delta \cap \mathbb R$.
 To this end we start by writing the following basic assumptions on $\mathcal F$, $X$, $\nu $ implying that $K$ is a $z_0$-bounded kernel.
\begin{assumption} \label{ass:2nd level} 
	Assume that
\begin{itemize}
	\item the map $(a,\xi) \mapsto \mathcal F(a, \xi)$ is measurable; 
	\item  the map $(a,\xi) \mapsto X(a, \xi)$ is measurable; 
	\item for every $\omega \in \mathcal B(\Omega_0 ) $ the map $x \mapsto \nu(x, \omega )$ is measurable and  
\[ \sup_{x \in \Omega } \nu (x, \Omega_0 ) \leq M 
\] 
 for some $M>0$; 
	\item there exists a $z_0<0$ and a constant $c>0$ such that 
	\begin{equation}\label{z0 k}
	\sup_{x \in \Omega_0} \mathcal F(t,x) \Lambda(X(t,x))  \leq  c  e^{ z_0 t}, \quad t\geq 0; 
	\end{equation}
\end{itemize} 
\end{assumption} 
If Assumption \ref{ass:2nd level} holds, then the kernel $K $, defined by \eqref{kernel}, is a $z_0$-bounded kernel. 
What additional assumptions on $\nu $ and $X$ guarantee that $K$ is also a regularizing kernel?
\begin{proposition} \label{prop:fission ass}
Let $\mathcal F $, $\nu$, $X$ satisfy Assumption \ref{ass:2nd level}. 

If, additionally, $\nu(x, \cdot) \in \mathcal M_{+, AC}(\Omega_0)$, then the kernel $K$ defined by \eqref{kernel} is a $z_0 $-bounded regularizing kernel.  
\end{proposition} 
\begin{proof}
The absolute continuity, with respect to the Lebesgue measure, of the measures \eqref{int AC} and \eqref{AC m_xt} is a consequence of the fact that $K(a,x,\cdot) \in \mathcal I$, with $\mathcal I$ defined in Definition \ref{ideal}, for every $a>0$ and $ x \in \Omega_0$. 
\qed\end{proof}
We aim at finding milder conditions on $\nu $ that still guarantee that the kernel $K$ is regularizing. Motivated by biological applications (see the upcoming sections) we focus on the following type of measures
	\begin{equation}\label{measure nu generalized} 
	\nu(x, \omega)= \beta (x) \delta_{q(x)} (\omega) \quad x \in \Omega, \  \omega \in \mathcal B(\Omega_0) . 
	\end{equation}
where $q $ and $\beta$ are suitable functions.

We first present an example of a measure $\nu$, satisfying \eqref{measure nu generalized}, and a function $X$, that give rise to a kernel $K$ which is \textbf{not} regularizing.
\begin{example}\label{example eq fission}
	If 
	\[
	\nu(x, \omega ) = 2 \delta_{\frac{x}{2}} (\omega)
	\] 
	and if we assume that the development is  exponential, i.e. $X(a,\xi) = \xi e^{a}$, then Assumption \ref{ass:AC} does not hold. Hence $K $ is not a regularizing kernel. 
	Indeed, 
	\begin{align*}
	&2  \int_0^t \int_\Omega K(s,x,d\xi) \mathcal F(t-s, \xi ) \delta_{\frac{1}{2}X(t-s, \xi)} (\omega) \Lambda (X(t-s, \xi))ds \\
	&= 4  \int_0^t \int_\Omega \mathcal F(s,x) \Lambda(xe^s) \delta_{\frac{x}{2} e^s }(d\xi)   \mathcal F(t-s, \xi ) \delta_{\frac{1}{2}\xi e^{t-s}} (\omega) \Lambda (X(t- s, \xi)) ds \\
	&= 4 \delta_{\frac{x}{4} e^t} (\omega)  \int_0^t \mathcal F(s,x) \Lambda(xe^s)  \mathcal F\left(t-s,  \frac{x}{2} e^s \right) \Lambda \left(X\left(t-s, \frac{x}{2} e^s\right)\right) ds. 
	\end{align*} 
\end{example}
The take home message of this example is that it is not only the shape of $\nu $ that determines whether the kernel is regularizing or not, but also the development rate.

We now state sufficient assumptions on $q $, $\beta$, $\mathcal F$ and $X$ that guarantee that the kernel $K$ defined by \eqref{kernel} is a $z_0$-bounded regularizing kernel.
\begin{proposition} \label{prop:regularization}
	Let $\mathcal F $, $\nu$, $X$ satisfy Assumption \ref{ass:2nd level}. 	
	Assume that $\nu $ is of the form \eqref{measure nu generalized} for a measurable function $\beta: \Omega \rightarrow \mathbb R_+$ and a measurable function $q: \Omega \rightarrow \mathbb R_+ $. 
	Additionally, assume that $q$ is such that the function
	\begin{equation} \label{Fa}
	F_a: x \mapsto q(X(a,x))
	\end{equation}
is invertible and such that if $|\omega|=0$, then $|F_a^{-1}(\omega)|=0$, where we are denoting with $|\cdot |$ the Lebesgue measure, see Appendix \ref{sec:notation}.
Finally assume that $q$ and $X$ are such that the function
\begin{equation}\label{pTx}
p_{t,x} : a \mapsto q(X(t-a,q(X(a,x))))
\end{equation}
is invertible and such that $|\omega|=0$ implies  $|p_{t,x}^{-1}(\omega)|=0$.
Then the kernel $K$ defined by \eqref{kernel} is a $z_0$-bounded regularizing kernel.
\end{proposition} 
\begin{proof}
 The kernel $K$ is  $z_0$-bounded because  $\mathcal F $, $X$ and $\nu$ satisfy Assumption \ref{ass:2nd level}.
	
We now prove that, for every $f$, the measure \eqref{int AC} is absolutely continuous with respect to the Lebesgue measure.
For notational convenience we rewrite $K$ as
\[
K(a,x,\omega)=j(a,x) \delta_{q(X(a,x))}(\omega).
\]

	Let $A \in \mathcal B(\Omega )$  be a set of zero Lebesgue measure, then
\begin{align*}
&\int_{\Omega_0} f(x) K(a,x,A) dx = \int_{\Omega_0} f(x) j(a,x) \delta_{q(X(a,x)) } (A)dx \\
&= \int_{F_a^{-1}(A)} f(x) j(a,x) dx =0. 
\end{align*}
We now prove that also \eqref{AC m_xt} is an absolutely continuous measure with respect to the Lebesgue measure. Indeed
	\begin{align*}
&K^{*2}(T,x,A)\\
&= \int_0^T \int_\Omega K(s,x,d \xi) j(T-s, \xi ) \delta_{q \left( X(T-s, \xi)\right) }(A)  ds \\
&= \int_0^T  j(s,x)  j(T-s, q(X(s,x)))  \delta_{p_{T,x}(s)}(A)  ds \\
& =\int_{[0,T] \cap p_{T,x}^{-1}(A ) } j(s,x) j(T-s, q(X(s,x)) )  ds.
\end{align*}
The assumptions on $p_{T,x}$ then guarantee that $K^{*2}(T,x,\cdot)$ is absolutely continuous with respect to the Lebesgue measure. 

\qed\end{proof}

\section{Asymptotic behaviour of the population birth rate for the model examples} \label{sec:asy models} 
We now motivate the above assumptions on $\nu$ by presenting the models that we are going to study with the results presented in Section \ref{sec:method}.

\subsection{Two applications to structured population models}\label{sec:two applications}

\subsubsection{Cell growth and fission} 
The first example is the model of cell growth and fission that is classically formulated via the PDE 
\begin{equation}\label{eq:equal fission den} 
\partial_t n (t, x)+ \partial_x \left( g(x)  n(t,x) \right) = - \left[ \Lambda(x) + \mu(x) \right] n(t,x) +4  \Lambda(2x) n(t,2x) 
\end{equation}
 or alternatively via the PDE
\begin{equation}\label{PDE growth-fission den} 
\partial_t n (t, x)+ \partial_x \left( g(x)  n(t,x)  \right) = - \left[ \Lambda(x) + \mu(x) \right] n(t,x) + \int_\Omega h(y,x)  \Lambda(y) n(t,y) dy ,
\end{equation}
These PDEs describe the evolution in time of a population of cells, structured by size, growing at rate $g$, dying at rate $\mu$ and dividing into two smaller cells at rate $\Lambda$.
 The type of equation depends on how the cells divide. More precisely, if cells divide into equal parts, then the density of cells of size $x$ at time $t$, $n(t,x)$, is the solution of equation \eqref{eq:equal fission den}. 
If, instead, the expected number of cells with size in $[y, y+ dy]$, produced by the division of a cell of size $x$, is equal to $h(x,y)dy$, then $n(t,x)$ is the solution of equation \eqref{PDE growth-fission den}.

The model described above fits into the class of models introduced in Section \ref{sec:kernel modelling ingr}. 
Hence, the population birth rate, which in this case is the rate at which individuals are born due to fission, has to satisfy \eqref{RE}, with $K$ given by \eqref{kernel} and $X(a,\xi)$ is the solution at time $a$ of the following ODE
	\begin{equation} \label{eq:size evolution} 
	\frac{dx}{dt}= g(x) \quad x(0)=\xi,
	\end{equation}
	while
\begin{equation}\label{survival probability} 
\mathcal F(t, \xi) := \exp\left( - \int_0^t \tilde \mu(X(s, \xi) ) ds\right) =  \exp\left( - \int_\xi^{X(t, \xi)} \frac{\tilde \mu(x)}{g(x)} dx \right) 
\end{equation}
where $\tilde \mu(x)=\mu(x)+ \Lambda(x)$, and with 
\[
\nu(x, \omega ) = \int_\omega h(x,y) dy \quad \text{ or } \quad  \nu(x, \omega )=2 \delta_{\frac{x}{2}}(\omega).
\] 

Now the question is, what are the assumptions on the parameters $g$, $\Lambda$ and $\nu$ that ensure that $K$ is a $z_0$-bounded regularizing kernel, that the corresponding operator $\tilde{K} $ satisfies Assumption \ref{rotation} or Assumption \ref{rot density} and that $\mathbb K_\lambda $ is non-supporting for every $\lambda \in \Delta \cap \mathbb R$, compact for every $\lambda \in \Delta$? In other words, what are the assumptions on the parameters that allow us to study the evolution of the population by using the results presented in Section \ref{sec:method}? 
Below we present two collections of assumptions, one for the case of fission into equal sizes and one for the case of fission into unequal sizes. 
We start with the latter.

\begin{assumption} [Unequal fission model]\label{ass: nu AC}
We assume that 
\begin{enumerate}
\item $\Omega=\mathbb R_+^* $; 
\item the growth rate $g: \Omega \rightarrow \mathbb R_+^*$ is a continuous function such that for every $z \in \Omega$
\begin{equation}\label{infinite travel time}
\int_z^{\infty} \frac{1}{g(s)} ds=\infty; 
\end{equation}
\item the fission rate $\Lambda: (0,\infty) \rightarrow \mathbb R_+$ is a measurable function such that either $ supp(\Lambda)=[M, \infty ) $, where $M> 0$, or $ supp(\Lambda)=\mathbb R^*_+$, and such that 
  $\lim_{ z \rightarrow \infty }\Lambda(z)$ exists and is strictly positive; 
\item the death rate $\mu: \Omega \rightarrow \mathbb R_+$ is measurable;  
\item  for every $y \in \Omega $
\begin{equation} \label{density h}
\nu(y, \cdot) \in \mathcal M_{+,AC}(\Omega),
\end{equation}
 with density $h(y,\cdot)$ such that $ h(y,x) =0$ when $y<x$ and $h(y,x)>0$ if $y>x $
\begin{equation}\label{eq:consistency condition} 
   \int_0^y h(y,x) dx =2,  \quad h(y,x)=h(y,y-x); 
\end{equation} 
\item the set of the states at birth is
\[
\Omega_0:= \bigcup_{y \in supp (\Lambda)} \textit{supp} \left( h(y, \cdot)\right)=(0,\infty).
\]
We assume that for every $\varepsilon >0 $ there exists a $\delta_\varepsilon >0$ such that for every $0 <\delta < \delta_\varepsilon$ we have 
\begin{equation}\label{ass:comp}
\left| \int_{\Omega_0} \left( h(y,x) - h(y, x+\delta) \right) dx \right|  < \frac{\varepsilon}{y} \quad  \text{ for every }  y >0
\end{equation}
where $h(y, x+ \delta ):=0$ if $x +\delta \notin \Omega_0$. 
\item 
 Finally we assume that 
 \begin{equation}\label{ass:comp2}
 \int_0^1 \frac{\Lambda(y)}{y g(y)} dy < \infty.
 \end{equation}
\end{enumerate}
\end{assumption}

We now explain the interpretation of these requirements on $g$, $h$, $\Lambda$. 
By the definition of $g$,
\[
\tau(x,y):=\int_x^y \frac{1}{g(z) } dz
\]
is the time that it takes to develop from size $x $ to size $y$. 
Hence, the fact that $g$ satisfies \eqref{infinite travel time} implies that the time that it takes to grow up to size equal to $\sup \Omega$ is equal to infinity.  
 This, together with the assumption on the limiting large size behaviour of the fission rate, guarantees that the probability that a cell reaches size equal to infinity is zero.

The first assumption on $h$ in \eqref{eq:consistency condition} guarantees that a cell always divides into two cells. The second assumption in the same line is a consequence of the fact that mass is conserved during fission and hence a cell of size $x$ that divides into a cell of size $y$ produces also a cell of size equal to $x-y$. 

In many works the $i$-state space $\Omega$ is assumed to be a compact subset of $\mathbb R_+^*$, see for instance \cite{diekmann1984stability} and \cite{heijmans1984stable}. 
Here we relax this assumption and assume that $\Omega =\mathbb R^*_+$. 
The price of this generalisation is that we need to impose assumptions on the model parameters $g, \Lambda, h$ that exclude gelation (i.e. escape of mass at infinity, in the "fragmentation" terminology) and shattering (i.e. escape of mass at zero). 
This is why we introduce conditions \eqref{ass:comp} and \eqref{ass:comp2}. 
These tightness assumptions guarantee the compactness of the operator $\mathbb K_\lambda $, as we will see in Section \ref{sec:asymp unequal fission}, proof of Proposition \ref{K compact and nonsupporting}.

Condition \eqref{ass:comp} holds for a broad class of self-similar kernels. In particular it holds for uniform fragmentation, $h(y,x)=\frac{2}{y} \chi_{(0, x )}$, but also for some of the self-similar kernels considered in \cite{stadler2020analyzing}. 
Indeed assume that
\[
h(y,x)=\frac{2}{y} p\left(\frac{x}{y}\right)
\] 
 where $p:[0,1] \mapsto \mathbb R_+$ is s.t. $p \in L^\infty ([0,1])$ with $\int_0^1 p(z) dz =1 $ and $p(1-z)=p(z)$ for every $z \in [0,1]$. 
Then 
\begin{align*}
&\left| \int_{\Omega_0} \left[ h(y,x)- h(y,x+\delta) \right] dx \right| \leq \left| \int_0^{y -\delta} \left[ h(y,x)- h(y,x+\delta) \right] dx \right| \\
& + \left| \int_{y -\delta}^y h(y,x) dx \right| \leq \frac{2}{y} \left| \int_0^{y -\delta} \left[ p\left(\frac{x}{y} \right)- p\left(\frac{x+\delta}{y}\right) \right] dx \right| \\
& + \frac{2}{y} \left| \int_{y -\delta}^y  p\left( \frac{x}{y} \right) dx \right| \leq 2 \left| \int_0^{1- \frac{\delta}{y}} \left[ p\left(z\right)- p\left(z+ \frac{\delta}{y}\right) \right] dz \right| \\
& + 2\left| \int_{1 -\frac{\delta}{y}}^1 p\left( z \right) dz \right|   \\
&\leq 2 \left| \int_0^{1- \frac{\delta}{y}}  p\left(z\right)dz -  \int_{\frac{\delta}{y}}^{1}   p\left(z\right)dz \right| +  \frac{2}{y}  \delta\| p \|_{L^\infty} \\
&\leq 2 \left| \int_{1- \frac{\delta}{y}}^1  p\left(z\right)dz \right| + 2 \left|  \int_0^{\frac{\delta}{y}}   p\left(z \right)dz \right| +  \frac{2}{y}  \delta\| p \|_{L^\infty}  \leq \frac{ 6 }{y}\delta\| p \|_{L^\infty}. 
\end{align*}
Hence $h$ satisfies  (\ref{ass:comp}).

Finally condition \eqref{ass:comp2} guarantees that $\Lambda (y) \rightarrow 0$ as $y \rightarrow 0$ quickly, namely faster that $y$ itself. 
\eu{We expect that it is possible to weaken considerably the condition $h(x, y) > 0$ if $x > y$. This condition is however attractive, because it allows for a straightforward proof of the non-supportingness of the operator $\mathbb K_\lambda$.}

For the model in which cells divide into equal parts we make the following assumptions. 
\begin{assumption}[Equal fission model]\label{ass:equal fission}
In this case we assume that 
\begin{enumerate} 
\item $\Omega= \Omega_0:=(0,\infty)$, 
\item $g$ satisfies point $2$ of Assumption \ref{ass: nu AC}, is Lipschitz continuous and $g(2x) < 2 g(x)$ for every $x \in \Omega$ and $0 <  \sup_{x \in \Omega} \frac{1}{g(x)} < \infty$, 
\item  $\Lambda $ satisfies point $3$  of Assumption \ref{ass: nu AC}, 
\item  $\mu$ satisfies point $4$  of Assumption \ref{ass: nu AC}, 
\item for every $y \in \Omega $ we have that 
\begin{equation}\label{nu AC fission}
\nu(y, \omega)= 2\delta_{y/2} (\omega) \quad  \omega \in \mathcal B(\Omega_0). 
\end{equation}  
\end{enumerate}
\end{assumption} 

The requirements on the parameters listed in Assumption \ref{ass:equal fission} are needed to deduce the asymptotic behaviour of the population with the method presented in Section \ref{sec:asympt measure}.
Indeed the assumptions on the growth rate $g$ exclude the possibility of having cyclic solutions, see \cite{metz2014dynamics}, \cite{bernard2016cyclic}, being a sufficient assumption to guarantee that the operator $\mathbb K_\lambda $ is compact and non-supporting, as we will see in Section \ref{sec:asymp unequal fission}.

\begin{lemma} \label{lem:examples work}
	Let either Assumption \ref{ass: nu AC} or Assumption \ref{ass:equal fission} hold.
	Then the kernel $K$ defined by \eqref{kernel} is a $z_0$-bounded regularizing kernel for some $z_0 <0$.
\end{lemma}
\begin{proof} 
Thanks to \eqref{infinite travel time}, 
	\[
	K(t,x,\Omega_0) \sim e^{-\lim_{z \rightarrow \infty} \Lambda(z) t} 
	\]
	as time tends to infinity. Hence for every $z_0 >- \lim_{z \rightarrow \infty} \Lambda(z)$ we have that $K$ satisfies \eqref{condition on k for reduction}.
	If $\nu $ satisfies Assumption \ref{ass: nu AC}, then this concludes the proof thanks to Proposition \ref{prop:fission ass}.
	
	Assume, instead, that $\nu$ is given by \eqref{nu AC fission} and that $g(2x) < 2 g(x) $. 
	Then the function $p_{T,x}$ introduced in Proposition \ref{prop:regularization} is equal to
	\[
	p_{T,x} : s \mapsto \frac{1}{2} X\left(T-s,  \frac{1}{2} X(s,x) \right). 
	\]
	This map is differentiable and by the chain rule
\[
	2 p_{T,x}'(s) = - \partial_1 X\left(T-s,  \frac{1}{2} X(s,x) \right) + \partial_2 X\left(T-s,  \frac{1}{2} X(s,x) \right)  \frac{1}{2} \frac{d}{ds} X(s,x)
\]

Using \eqref{eq:size evolution} we deduce that for every $a>0$, $\xi >0$ and $s>0$
\[
\frac{d X(a,(X(s,\xi))}{ds} = \partial_2 X \left(a,X(s,\xi) \right) g(X(s, \xi)). 
\]
On the other hand
\[
\frac{d X(a,(X(s,\xi))}{ds}=\frac{d X(a+s,\xi)}{ds} = g(X(a+s,\xi)). 
\] 
Hence substituting $s=0$ we deduce that 
\[
\partial_2 X\left(a,\xi  \right) = \frac{g\left(  X\left(a,\xi  \right) \right)}{g(\xi)}. 
\]

Therefore using $g(2x) < g(x)$ we deduce that
	\begin{align*}
	2 p_{T,x}'(s) &= - g\left( X\left(T-s,  \frac{1}{2} X(s,x) \right)\right) + \frac{g( X\left(T-s,  \frac{1}{2} X(s,x) \right)) }{g\left(\frac{1}{2} X(s,x) \right) } \frac{1}{2} g (X(s,x)) \\
	& = - g\left( X\left(T-s, \frac{1}{2} X(s,x) \right)\right)  \left( 1 - \frac{g (X(s,x))}{2g\left( \frac{1}{2} X(s,x) \right) }  \right) <0.
	\end{align*}
	As a consequence $p_{T,x}$ is monotone, hence invertible and such that $|A|=0$ implies $|p_{T,x}^{-1}(A)|=0$

The function $F_a $, given by \eqref{Fa}, is invertible and such that if $|A |=0$, then $|F_a^{-1}(A)|=0$ as the map $x \mapsto X(a,x)$ is monotone. 
\qed \end{proof}

%%%

\subsubsection{Waning and boosting} 
Consider a population of individuals structured by their level of immunity against a pathogen. 
Assume that the level of immunity decreases with rate $g$ and is boosted by infection and that the force of infection equals a constant $\gamma.$
We assume that the time that it takes the immune systems to clear the infection is negligible compared to the time in between two infections and consider, accordingly, boosting as instantaneous. 

Assume that the immunity level after the boosting is determined by the immunity level before the boosting event via the boosting function $f$.
\eu{We assume that $f$ is as in Figure \ref{fig:boosting} and we denote with $f_1$ the restriction of $f$ to the set $(0, x_c)$ and with $f_2 $ the restriction to $f$ on $(x_c, M]$. } 
The density of individuals with immunity level $x$ at time $t$, $n(t,x)$, satisfies the following PDE
	\begin{equation}\label{eq:waning and boosting den} 
	\partial_t n (t, x)+ \partial_x \left(  g(x) n(t,x) \right) = - \gamma n(t,x) +Sn(t,x)
	\end{equation} 
where 
\begin{equation}\label{op S}
S \varphi(x)=\begin{cases} 0 & x < m \\
 - \gamma \frac{1}{f'(f_1^{-1}(x))} \varphi ( f_1^{-1} (x)) +  \gamma \frac{1}{f'(f_2^{-1}(x))} \varphi ( f_2^{-1} (x)) & m<x<r  \\ 
\gamma \frac{1}{f'(f_2^{-1}(x))} \varphi ( f_2^{-1} (x)) & r<x<M.
\end{cases}
\end{equation}

The term  $S n(t, x)$ in equation \eqref{eq:waning and boosting den} represents the individuals that (re)appear in the population at time $t$ with state $x$ after boosting.
\eu{Since the function $f$ has a local minimum, an individual with immunity level $x \in [m,r]$ can be obtained as the result of the boosting of an individual in any one of the sets $(0, x_c)$, $(x_c,M)$ while an individual with state at birth $x \in [r, M]$ is produced by the boosting of an individual with state in $(x_c,M)$.}

The backward reformulation of equation \eqref{eq:waning and boosting den} is 
\begin{equation}\label{eq:waning and boosting den back} 
	\partial_t m (t, x)- g(x) \partial_x  m(t,x) = - \gamma m(t,x) + S^*m(t,x) 
\end{equation} 
where $ S^*$ is the (pre)dual operator of $S $ and is given by
\[
S^* \varphi(x) = \gamma \varphi(f(x)).
\]

	This model fits into the class of models described in Section \ref{sec:kernel modelling ingr}. 
	Hence, the population birth rate $B$, which in this case is the rate at which individuals appear in the population with a higher immunity level due to boosting, \eu{is the solution of equation \eqref{RE} with a  kernel $K$ given by formula \eqref{kernel}. 
The factor $X(a,\xi)$ in  \eqref{kernel} is the solution of the ODE \eqref{eq:size evolution}, with $g$ the rate of waning. 
The factor $\Lambda(x)=\gamma >0$ is the boosting rate. 
Since we assume the death rate to be equal to zero, we have that the term $\mathcal F$ in  \eqref{kernel}  is equal to
	\begin{equation}\label{survival probability w} 
	\mathcal F(t, \xi) := e^{- \gamma t }
	\end{equation}
and
\begin{equation}
\label{nu AC immune}
\nu(y, \omega):= \delta_{f(y)} (\omega) 
\end{equation}
for every $i$-state $y$ and for every set of states at birth $\omega $.}

\eu{We have chosen a specific form of $f$ in order to make the computations in Section \ref{sec:asympt waning boosting} not too demanding for the reader. For sure the result holds for a much wider class of boosting functions $f$ (see for instance \cite{diekmann2018waning}, but note that in that paper there is no proof that convergence is exponential).} 
We now specify the assumptions on the parameters that guarantee that we can apply the results presented in Section \ref{sec:method}. 
\begin{assumption}[Waning and boosting model, see Figure \ref{fig:boosting}]\label{ass:immunity}
We assume that 
\begin{enumerate}
\item $\Omega =(0,M] $;
\item \eu{the boosting function $f: \Omega \rightarrow [m, M] =:\Omega_0$} is such that $f(x)=f_1(x)$ if $x \in (0, x_c]$ while 
$f(x)=f_2(x)$ if $x \in ( x_c, M ]$ where 
\[
f_1(x)= -\alpha_1 x + q_1 \text{ and }  f_2(x)= \alpha_2 x + q_2 
\]
with 
\[
\alpha_1 =\frac{r-m}{x_c }, \quad  q_1 =r\textit{ where } 0<m<M, 0<r<M
\] 
\[
\alpha_2 =\frac{M-m}{M-x_c }, \quad  q_2 =m - x_c \frac{M-m}{M-x_c}; 
\] 
\item  $g: \Omega \rightarrow \mathbb (-\infty , 0)$ is a continuous function and  such that
\begin{equation}\label{assumption on jumps}
\frac{\alpha_2 g(y)}{g(f(y))}<1  \quad \text{ for a.e. } \ y \in \Omega_0.
\end{equation}
\end{enumerate}
\end{assumption}
The conditions on the parameters $g$, $\nu$, $\Lambda $ listed in Assumption \ref{ass:immunity} guarantee that the model is well defined and allow to apply the results of Section \ref{sec:asympt measure} as we will see in Section \ref{sec:asympt waning boosting}. 

In this work we focus on Assumption \ref{ass:immunity} and we assume that the set of the possible immunity levels is a compact set, but this assumption can be relaxed as for instance in \cite{diekmann2018waning}.

\begin{figure}[H] 
\centering
\includegraphics[scale=0.6]{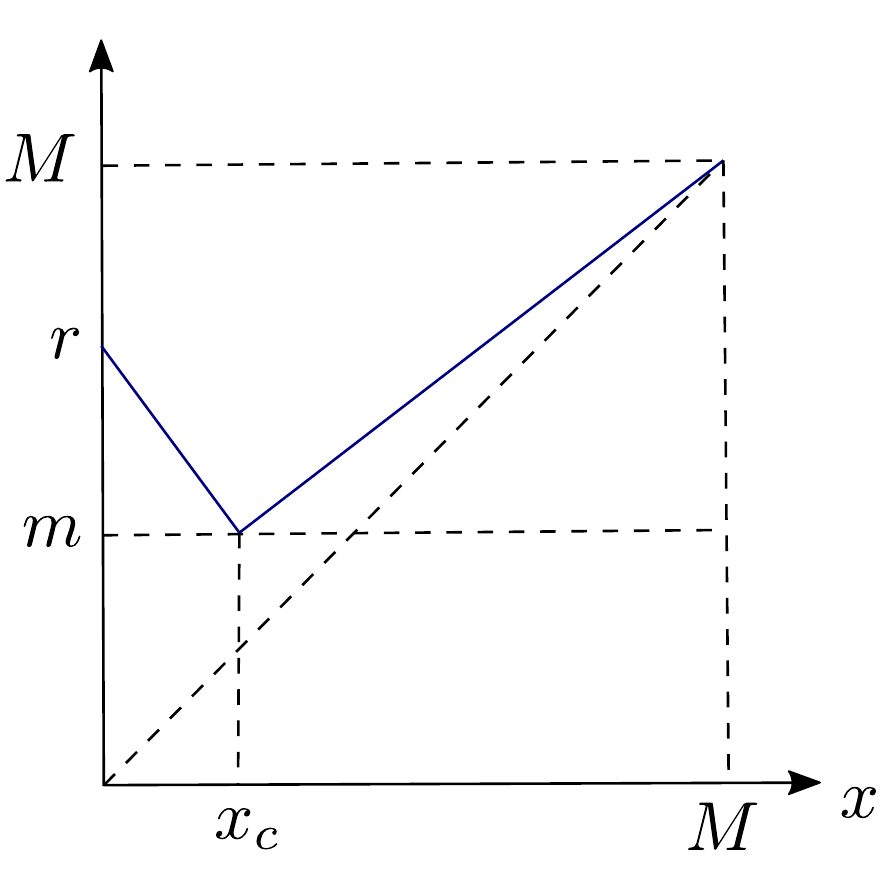}
\caption{Boosting function}\label{fig:boosting}
\end{figure}

Condition \eqref{assumption on jumps} is sufficient to guarantee that the kernel $K$ defined by \eqref{kernel} is regularizing. The meaning of this assumption is the following. The immunity level of an individual who boosts at time $t$ and then wanes for a time interval of length $dt$ is lower than the immunity level of an individual who wanes for $dt$ and then boosts at time $t+dt.$ This assumption can be seen as a congener of the assumption $g(2x) < 2 g(x)$ in the case of fission into equal sizes. We refer to \cite{diekmann2018waning} for more explanations.

\begin{lemma} \label{lem:examples work imm}
	Let $g, \mu, \Lambda, \nu$ satisfy Assumption \ref{ass:immunity}. 
	Then the kernel $K$ defined by \eqref{kernel} is a $- \gamma$-bounded regularizing kernel.
\end{lemma}
\begin{proof} 
	The fact that $K$ is a $-\gamma $ kernel follows simply by noting that 
	\[
	\mathcal F(a,x) \Lambda(X(a,x) ) = \gamma e^{- \gamma a}.
	\]
	We next investigate whether $K$ is a regularizing kernel. 
	To this end we notice that the function $F_a$ introduced in Proposition \ref{prop:regularization},
	\[
	F_a : x \mapsto f(X(a,x))
	\]
	 is invertible because it is piecewise monotone.
	
	On the other hand, the function $p_{T,x}$ now reads 
	\[
	p_{T,x}: a \mapsto f(X(T-a, f(X(a,x))))
	\]
As in the proof of Lemma \ref{lem:examples work}, using the chain rule, condition \eqref{assumption on jumps} and the definition of $X$ as the solution of the ODE \eqref{eq:size evolution} we prove that 
	\begin{align*}
&	p_{T,x}'(a)= - f'(X(T-a, f(X(a,x)))) g(X(T-a, f(X(a,x)))) \cdot \\
	& \cdot \left[ 1 - \frac{f'(X(a,x)) g(X(a,x))}{g(f(X(a,x)))}\right] \quad a.e.\  a >0.
	\end{align*}
	Thanks to \eqref{assumption on jumps} we deduce that $p_{T,x}$ is piecewise monotone. Hence the desired conclusion follows. 	
\qed\end{proof}

\eu{\subsection{Asymptotic behaviour for the model examples} }\label{sec:asympt examples} 
\eu{
In this section we apply the results presented in Section \ref{sec:method} to the model examples. To this end we proceed as follows
\begin{enumerate}
\item  we use Lemma \ref{b and B} to associate to the kernel $K$ an operator kernel $\tilde{K}$; 
\item we define the discounted next generation operator $\mathbb K_\lambda$ as a function of $\tilde{K}$, using \eqref{NGO lambda};  
\item then we check that the operator $\mathbb K_\lambda$ is compact and non supporting. 
\end{enumerate} }

\subsubsection{Cell growth and fission (into unequal parts)}\label{sec:asymp unequal fission} 
In this section we assume that the parameters $g, \mu, \Lambda , \nu $ satisfy Assumption \ref{ass: nu AC}. Hence there exists a $z_0$ such that the kernel $K$, given by \eqref{kernel}, is a $z_0$-regularizing kernel. 

It remains to prove that, under the assumptions of the unequal fission model, $\mathbb K_\lambda$ satisfies the assumptions of Theorem \ref{thm:asymptotic behaviour AC}.
Recall that $K \in \mathcal I$. We denote with $k $ its density, given by
\begin{align}\label{kernel density AC}
k (a,x,y):=\mathcal F(a,x) \Lambda(X(a,x)) h(X(a,x), y ).
\end{align}
The operator $\tilde{K}$ introduced in \eqref{eq:ktilde}, is given by
\begin{equation}\label{Ktilde unequal}
\left(\tilde{K}(a) \varphi \right)(y) := \int_{\Omega_0} k(a,x,y) \varphi(x) dx \quad a \geq 0, \quad y \in \Omega_0
\end{equation}
and as a direct consequence we have the following result.
\begin{lemma}\label{AC rotation}
The kernel $\tilde{K}$ satisfies Assumption \ref{rot density}. 
\end{lemma}

The following theorem, see e.g \cite{kufner1977function}, is fundamental to prove the compactness of the operator $\mathbb K_\lambda$ in the model examples.  
\begin{theorem}[Fr\'echet-Kolmogorov]\label{lemma2}
	Let $T: X^{\mathbb C} \rightarrow X^{\mathbb C}$ be linear and bounded. 
	If for every $\varepsilon >0$ there exists a $\delta >0$ such that for every $0 < |h | < \delta $
	\[
	\int_{\Omega_0} \left| T \varphi(x+ h) - T \varphi (x) \right|dx \leq \varepsilon \| \varphi \|_1 
	\]
	for every $\varphi \in X^{\mathbb C}$, where $T\varphi(x+h)=0$ if $x+ h \notin \Omega_0$, then the operator $T$ is compact. 
\end{theorem}

\begin{proposition}\label{K compact and nonsupporting} 
The operator $\mathbb K_\lambda$ is compact for every $\lambda \in \Delta $ and the operator $\mathbb K_\lambda $ is non-supporting for every $\lambda \in \Delta \cap \mathbb R$.
\end{proposition}
\begin{proof}
To prove compactness we apply Theorem \ref{lemma2}. 
Recalling Assumption \ref{ass:comp} we deduce that for every $\varepsilon >0$ there exists a $\delta_\varepsilon>0$ such that for every $\delta < \delta_\varepsilon$ and for every $\varphi \in X_+$
\begin{align*}
& \int_{\Omega_0}  \left| \left( \mathbb K_\lambda \varphi \right) (x+\delta)- \left(\mathbb K_\lambda \varphi \right) (x) \right| dx  \\
& \leq \int_0^\infty \int_{\Omega_0} \varphi(y) e^{- a \Re\lambda } \mathcal F(a,y) \Lambda(X(a,y))  \cdot \\
&\cdot \left| \int_{\Omega_0} h(X(a,y), x+\delta) - h(X(a,y), x ) dx \right| dy  da  \\
& \leq \varepsilon \int_0^\infty \int_0^\infty \varphi(y) e^{-a \Re \lambda } \mathcal F(a,y) \frac{\Lambda(X(a,y))}{X(a,y)} dy da, 
\end{align*}
where for the last inequality we have used \eqref{ass:comp}.  

Using Fubini's theorem and performing the change of variables $X(a,y)=x$ we deduce that
\begin{align*}
&\int_0^\infty \int_0^\infty  \varphi(y) e^{- a \Re \lambda } \mathcal F(a,y) \frac{\Lambda(X(a,y))}{X(a,y)} dy da \\
&= \int_0^\infty \int_0^\infty \varphi(y) e^{-a \Re \lambda } \mathcal F(a,y) \frac{\Lambda(X(a,y))}{X(a,y)} da dy \\
&= \int_0^\infty \int_y^\infty \varphi(y)  e^{- \tau(y,x)\Re \lambda  } \hat{\mathcal F}(y,x) \frac{\Lambda(x)}{x g(x) }dx dy  \\
&\leq  \int_0^1 \int_0^x \varphi(y) e^{- \tau(y,x) \Re \lambda } \hat{\mathcal F}(y,x) \frac{\Lambda(x)}{x g(x) } dy dx \\
& + \int_1^\infty \int_0^x \varphi(y) e^{-\tau(y,x) \Re \lambda} \hat{\mathcal F}(y,x) \frac{\Lambda(x)}{x g(x) } dy dx
\end{align*}
Now, using Fubini's theorem, the change of variables $x=X(a,y)$ and the bound \eqref{z0 k}, we estimate the second term in the following way 
\begin{align*}
&\int_1^\infty \int_0^x  \varphi(y) e^{- \tau(y,x) \Re \lambda } \hat{\mathcal F}(y,x) \frac{\Lambda(x)}{x g(x) } dy dx \\
&\leq \int_1^\infty  \int_0^x  \varphi(y) e^{-\tau(y,x) \Re \lambda } \hat{\mathcal F}(y,x) \frac{\Lambda(x)}{ g(x) } dy dx \\
& \leq \int_0^\infty  \int_0^x     \varphi(y) e^{-\tau(y,x) \Re \lambda } \hat{\mathcal F}(y,x) \frac{\Lambda(x)}{ g(x) } dy dx \\
& \leq \int_0^\infty  \int_y^\infty   \varphi(y) e^{-\tau(y,x) \Re \lambda } \hat{\mathcal F}(y,x) \frac{\Lambda(x)}{ g(x) } dx dy\\
& \leq \int_0^\infty  \int_0^\infty  \varphi(y) e^{-a \Re \lambda } \mathcal F(a,y) \Lambda(X(a,y)) da dy \\
& \leq C \int_0^\infty   e^{\left( -\Re \lambda + z_0\right)a }da  \int_0^\infty  \varphi(y) dy
\leq \frac{C}{\Re \lambda - z_0 } \|\varphi \|_1 
\end{align*}
On the other hand, thanks to \eqref{ass:comp2} 
\begin{align*}
\int_0^1 \int_0^x \varphi(y) e^{-\tau(y,x)\Re \lambda } \hat{\mathcal F}(y,x) \frac{\Lambda(x)}{x g(x) } dy dx \leq e^{-\tau(0,1)\Re \lambda } \| \varphi \|_1 \int_0^1 \frac{\Lambda(x)}{x g(x) } dx . 
\end{align*}

It follows that for every $\varepsilon >0$ there exists a $\delta$ such that 
\[
 \left\| \left( \mathbb K_\lambda \varphi \right) (\cdot+\delta)- \left(\mathbb K_\lambda \varphi \right) (\cdot) \right\|_{1}  \leq \varepsilon  \| \varphi \|_{1} .
\]
Applying Lemma \ref{lemma2} we conclude that $\mathbb K_\lambda$ is compact for every $\lambda \in \mathbb C$ with $\Re \lambda > z_0$.

We now prove that $\mathbb K_\lambda$ is non-supporting. 
To this end we firstly prove a stronger property. Indeed, thanks to \eqref{ass:comp} we can prove that $\mathbb K_\lambda \varphi (x)>0$ for every $x \in \Omega_0$, because for every $\varphi \in L_+^1(\Omega_0)$ and every $x \in \Omega_0$
\begin{align*}
(\mathbb K_\lambda \varphi )(x)=& \int_0^\infty \int_0^\infty \varphi(y) e^{-\lambda a} \mathcal F(a,y) \Lambda(X(a,y)) h(X(a,y), x) da dy \\
& = \int_0^\infty \varphi(y) \int_y^\infty  e^{-\lambda \tau(y,z)} \hat{\mathcal F}(z,y) \frac{\Lambda(z)}{g(z)} h(z, x) dz dy \\
\end{align*} 
Since we assume that $h(z,x)>0$ for every $z>x $, then \[
\int_y^\infty \hat{\mathcal F}(y,z) e^{-\lambda \tau(y,z)} \frac{ \Lambda(z) }{g(z)} h(z, x) dz >0.\] 
\eu{Since $\varphi$ belongs to $X_+$ there exists a set of positive Lebesgue measure $U \subset \Omega_0$, on which $\varphi$ is strictly positive.  
} 
\eu{The integration is over $\Omega_0$.  We conclude that $(\mathbb K_\lambda \varphi )(x)>0$ for every $x \in \Omega_0.$}
\qed\end{proof} 

\begin{proposition}
	Let $g, \mu, \Lambda, \nu $ be such that Assumption \ref{ass: nu AC} holds.
	Let $r, \psi_r$ be respectively the Malthusian parameter and the stable distribution.
The solution $B$ of \eqref{RE} satisfies 
\[
\| e^{-rt } B(t, \cdot) - c \Psi_r  \|\leq M e^{-v t}
\] 
where $\Psi_r(dx)=\psi(x) dx $ and $c, M>0$ and $v>0$ and the norm $\| \cdot \|= \| \cdot \|_{TV} = \| \cdot \|_{\flat}$.
Moreover
\[
sign(r)=sign(R_0 -1 )
\]
where $R_0 $ is defined in \eqref{basci repro n} and
\[
\mathbb K_0 \mu (\cdot) = \int_0^\infty \int_{\Omega_0} k(t,x,\cdot) \mu (dx) dt. 
\]
\end{proposition}

\subsubsection{Cell growth and fission (into equal parts)}
In this section we assume that the parameters $g, \mu, \Lambda , \nu $ are such that Assumption \ref{ass:equal fission} holds. 
Also in this case we have to check whether the assumptions of Lemma \ref{b and B} hold. 
In this case 
\begin{equation}\label{Ktilde equal} 
(\tilde{K}(s)\varphi) (z):= 4 \frac{g(X(-s,2z))}{g(2z)}\mathcal F(s, X(-s,2z)) \Lambda(2z) \varphi(X(-s, 2z))  
\end{equation} 
for every $z>0$ such that $s < \tau(0, 2z)$ while $(\tilde{K}(s)\varphi) (z)=0$ otherwise. 
Indeed 
\begin{align*}
& \int_{\Omega} K(s,x,\omega) \varphi(x) dx =2  \int_0^\infty \mathcal F(s,x) \Lambda(X(s,x)) \delta_{\frac{1}{2} X(s,x)} (\omega) \varphi(x) dx \\
& = 4  \int_{\omega} \chi_{[0, \tau(0, 2y)]}(s) \mathcal F(s,X(-s,2y)) \Lambda(y) \varphi(x) \frac{g(X(-s, 2y))}{g(2y)} dy 
\end{align*}

\begin{lemma}\label{lem:equal rotation}
The kernel $\tilde{K}$ satisfies Assumption \ref{rotation}. 
\end{lemma}
\begin{proof}
%Let $\lambda \in \Delta  $ and let $\psi \in X_+$. Assume that $\mathbb K_\lambda\psi=|\mathbb K_{\Re \lambda+ i Im \lambda }\psi e^{ i \ell(\cdot)}| $ for some $\ell: \Omega_0 \rightarrow \Omega_0$. 
%Then, by the definition of $\mathbb K_\lambda$, this implies that 
%\begin{align*}
%& \int_0^\infty e^{- \Re \lambda  a } \tilde{K}(a) \psi da = \left|  \int_0^\infty e^{-\Re \lambda a -i Im \lambda  a} \tilde{K}(a) \psi e^{ i \ell(\cdot)} da \right|.
%\end{align*}
%By the definition of $\tilde{K}$, \eqref{Ktilde equal}, we deduce that 
%\begin{align*}
%&\int_0^\infty \left| e^{-\Re \lambda a -i Im \lambda  a} \tilde{K}(a) \psi e^{ i \ell(\cdot)} \right| da  = \int_0^\infty \left| e^{- \Re \lambda a -i Im \lambda  a} e^{i \ell(X(-a, 2 \cdot))}  \tilde{K}(a) \psi \right| da \\
%&=  \int_0^\infty e^{- \Re \lambda a } \tilde{K}(a) \psi da
%\end{align*}
%we apply Theorem 1.39 in \cite{walter1987real} to deduce that 
%\[
% \ell(X (-a,x))-\eta a =\beta \text{ for a.e. } a>0 \text{ and for a.e. } x \in \Omega_0.
%\]
%Hence we deduced the desired conclusion.
The statement follows by the definition of $\tilde{K}$.
\qed\end{proof}
The following theorem can be found in \cite{krasnoselskii1976integral} and will be important to prove the compactness of the operator $\mathbb K_\lambda$. 
\begin{theorem}\label{lemma1}
	Let $T: X^{\mathbb C} \rightarrow X^{\mathbb C}$ be linear and bounded and of the form 
	\[
	(T \varphi )(x)=\int_{\Omega_0} h(x,y) \varphi (y) dy. 
	\]
	Suppose that there exists an $h_+ $ such that 
	\[
	|h(x,y)| \leq h_+(x,y) \quad x,y \in \Omega_0
	\]
	and that the operator $T_+: X_+ \rightarrow X_+$
	\[
	(T_+ \varphi )(x):=\int_{\Omega_0} h_+(x,y) \varphi (y) dy 
	\]
	is compact. Then $T$ is compact.
\end{theorem}

\begin{lemma}\label{lem:comp non supp equal}
The operator $\mathbb K_\lambda$ is compact for every $\lambda \in \Delta $ and $\mathbb K_\lambda$ is non-supporting for every $\lambda  \in \Delta \cap \mathbb R.$
\end{lemma}
\begin{proof}
Using \eqref{Ktilde equal} and \eqref{z0 k}, we deduce that for every $\varphi \in X_+$ 
\begin{align*} 
&  \left|\left(\mathbb K_\lambda \varphi \right)(y)\right| \\
& \leq  4  \int_0^{\tau(0,2y)} e^{-s \Re \lambda}   \frac{g(X(-s, 2y))}{g(2y)}\mathcal F(s, X(-s,2y)) \Lambda(2y) \varphi(X(-s, 2y))  ds \\
& \leq c \int_0^{\tau(0,2y)}  e^{-s \Re \lambda} e^{z_0 s }  \frac{g(X(-s, 2y))}{g(2y)} \varphi(X(-s, 2 y))  ds \\
\end{align*}
Let $p >0$. The operator $\mathbb K^{+}(p)$ defined by 
\[
\left( \mathbb K^{+}(p) \varphi \right)(y) = \int_0^{\tau(0, 2y)}  e^{-p s} \frac{g(X(-s, 2y))}{g(2y)} \varphi(X(-s,2y)) ds  
\]
is a linear bounded map from $X_+$ to $X_+$. 
Indeed, thanks to the second assumption in \eqref{ass:comp2}, if we assume that $\varphi \in X_+$, then 
\begin{align*}
&\int_0^\infty \left( \mathbb K^{+}(p) \varphi\right) (y) dy =\int_0^\infty  \int_0^{\tau(0, 2y)}  e^{-p s} \varphi(X(-s,2y)) \frac{g(X(-s, 2y))}{g(2y)}ds dy \\
&= \int_0^\infty  \int_0^{2y} e^{-p \tau(z,2y)} \frac{\varphi(z)}{g(2y)} dz dy \leq  \int_0^\infty  \int_{z/2}^\infty  \frac{e^{-p \tau(z,2y)}}{g(2y)} dy \varphi(z) dz \\
&=\frac{1}{p}  \| \varphi \|_1\\
\end{align*}

We want to prove that $\mathbb K_\lambda$ it is compact. To this end, we apply Lemma \ref{lemma2}. 
Consider $\delta >0$, then 
\begin{align*}
&\left| \mathbb K^{+}(p) \varphi(\delta+y ) - \mathbb K^{+}(p) \varphi(y) \right| \\
& \leq \left| \int_0^{\tau(0, 2y+ 2\delta)}  e^{-p s} \frac{g(X(-s,2y+ 2\delta))}{g(2y+2\delta)} \varphi(X(-s,2y+ 2\delta )) ds  \right. \\
& \left. -  \int_0^{\tau(0, 2y)}  e^{-p s} \frac{g(X(-s,2y))}{g(2y)} \varphi(X(-s,2y)) ds \right| \\
& \leq \left| \int_0^{2y+2\delta }  e^{-p \tau(z,2y+2\delta)}  \frac{\varphi(z)}{g(2y+2\delta )} dz - \int_{0}^{2y}  e^{-p \tau(z, 2y)} \frac{\varphi(z)}{g(2y)} dz \right| \\
& \leq  \left| \int_{2y}^{2y+2\delta }  e^{-p \tau(z,2y+2\delta)}  \frac{\varphi(z)}{g(2y+2\delta)} dz \right| \\
&+ \int_{0}^{2y}  \left|\frac{e^{-p \tau(z, 2y)}}{g(2y)} - \frac{e^{-p \tau(z,2y+2 \delta)}}{{g(2y+2\delta)}} \right| \varphi(z) dz . \\
\end{align*} 
\eu{Moreover}
\begin{align*}
 & \int_{2y}^{2y+ 2\delta}  \left|e^{-p \tau(z,2y+2\delta)} \frac{\varphi(z)}{g(2y+2 \delta)} \right| dz \\
& \leq   \sup_{x \in \Omega}\frac{1}{g(x)}  e^{-p \tau(2y ,2y+2\delta)}  \int_{2y}^{2y+ 2\delta}  e^{-p \tau(z,2y)} \varphi(z) dz . \\
\end{align*}
Consequently, 
\begin{align*}
&\| \mathbb K^+ (p) \varphi (\cdot+\delta) - \mathbb K^+ (p) \varphi(\cdot ) \|_1 \\
& \leq \sup_{x \in \Omega}\frac{1}{g(x)} \int_0^\infty    e^{-p \tau(2y ,2y+2\delta)} \int_{2y}^{2y+ 2\delta}  e^{-p \tau(z,2y)} \varphi(z) dz   dy  \\
&+ \int_0^\infty  \int_{0}^{2y}  \left| \frac{e^{-p \tau(z, 2y)}}{g(2y)} - \frac{e^{-p \tau(z,2y+2 \delta)}}{g(2y+2\delta)} \right| \varphi(z) dz  dy \\
& \leq \sup_{x \in \Omega}\frac{1}{g(x)} \int_0^\infty \int_{z/2-\delta}^{z/2}   e^{-p \tau(2y ,2y+2\delta)}  e^{-p \tau(z,2y)}  dy \varphi(z) dz \\
&+\left( \sup_{x \in \Omega}  \frac{1}{g(x)} \right)^2 \int_0^\infty  \int_{z/2}^\infty  \left| g(2y+2\delta)e^{-p \tau(z, 2y)} - e^{-p \tau(z,2y+2 \delta)} g(2y ) \right| \varphi(z)  dy dz  \\
& \leq \sup_{x \in \Omega}  \frac{1}{g(x)} \hat{c} \delta \|\varphi \|_1  +\left(\sup_{x \in \Omega}  \frac{1}{g(x)}\right)^2 c' \delta \|\varphi \|_1  \\
& \leq \sup_{x \in \Omega}  \frac{1}{g(x)} \left[ \hat{c} \delta \|\varphi \|_1 + c^*  2 \delta \|\varphi \|_1  \right] \\
\end{align*} 
where $\hat{c}, c^*, c'>0$ and where we have used the Lipschitz continuity of $g$ and \eu{of $\tau$ with respect to its second argument.} 
We deduce that $\mathbb K^{+} (p)$ is compact for every $p>0$, hence by Lemma \ref{lemma1} we have that $\mathbb K_\lambda$ is compact for every $\lambda \in \mathbb C$ with  $\Re \lambda > z_0$.

We now check that $\mathbb K_\lambda $ is non-supporting for every $\lambda \in \mathbb R \cap \Delta$. 
For every $\varphi \in X_+$, there exists a set $S$ of positive Lebesgue measure with $\varphi(x) \neq 0$ for every $x \in S.$ Therefore for every $2 x> x_1:=\inf S$
\[
 \mathbb K_\lambda \varphi (x)=4\Lambda (2x) \int_0^{2x} e^{-\lambda \tau(z, 2x)} \hat{\mathcal F}(z,2 x)  \frac{\varphi(z)}{g(2x)} dz>0.
\] 
This implies that for every $x > x_1/4 $ we have that $\mathbb K_\lambda^2 \varphi (x) >0$, indeed 
\[
\mathbb K_\lambda^2 \varphi(x)= 4 \int_0^{2x} e^{-\lambda \tau(z, 2x)} \hat{\mathcal F}(z,2 x)  \frac{ \mathbb K_\lambda \varphi (z) }{g(2x)} dz. 
\]
Iterating this argument we deduce that for every $\varphi \in X_+$ there exists an $x_1>0$ and an $n \in \mathbb N$ such that $\mathbb K^n_\lambda \varphi (x)>0$ for every $x > \frac{x_1}{2^n}$. 
As a consequence, this implies that for every $F \in L^\infty_+(\Omega_0)$ and every $\varphi \in X_+$ there exists an $n$ such that $\langle F, \mathbb K_\lambda^n \varphi \rangle >0.$
\qed\end{proof}

\begin{proposition}
	Let $g , \mu, \Lambda, \nu $ be such that Assumptions \ref{ass: nu AC} holds. 
		Let $r, \psi_r$ be respectively the Malthusian parameter and the stable distribution.  
The solution $B$ of \eqref{RE} satisfies 
\[
\| e^{-rt } B(t, \cdot) -c  \Psi_r  \|\leq M  e^{-v t}
\] 
where $c, M, v>0$ and $\Psi_r(dx)=\psi_r dx $ and $\| \cdot \|= \| \cdot \|_{TV} = \| \cdot \|_{\flat}$, 
\[
sign(r)=sign (R_0-1 ), 
\]
and $R_0$ is the spectral radius of $\mathbb K_0$ defined by
\[
\mathbb K_0 \varphi (x) = \int_0^\infty \frac{g(X(-t,2x))}{g(2x)} \mathcal F(t, X(-t, 2x)) \Lambda(2x) \varphi(X(-t, 2x)) dt . 
\]
\end{proposition}

\subsubsection{Waning and boosting}\label{sec:asympt waning boosting}
\eu{In this section we make Assumption \ref{ass:immunity}.
In this case the operator kernel $\tilde{K}$ is such that for every $a \geq 0$ the operator $\tilde{K}(a)$ belongs to $ \mathcal L(L_+^1(\Omega_0))=\mathcal L(L_+^1([m,M]))$ and is equal to 
\begin{equation}\label{Ktilde imm} 
 (\tilde{K}(s)\varphi) (z):=\begin{cases} 
 \gamma  e^{-\gamma s} \left[- \beta_1(z)  \varphi(X(-s, f_1^{-1}(z) ))g(X(-s, f_1^{-1}(z)) \right. &\\ 
\left. + \beta_2(z) \varphi(X(-s, f_2^{-1}(z) ))g(X(-s, f_2^{-1}(z)) \right], &z<r  \\
 \gamma  e^{-\gamma s} \beta_2(z) \varphi(X(-s, f_2^{-1}(z) ))g(X(-s, f_2^{-1}(z)), &z>r
\end{cases} 
\end{equation} 
for $z \in [m,M]$, where 
\begin{equation}\label{eq:Di}
\beta_1(z)= \frac{1}{g(f^{-1}_1(z)) f'(f_1^{-1}(z))}>0 
\end{equation}
while
\begin{equation}\label{eq:Di}
\beta_2(z)= \frac{1}{g(f^{-1}_2(z)) f'(f_2^{-1}(z))}<0. 
\end{equation}}

\eu{Indeed the measure \eqref{int AC} is equal to 
\begin{align*}
& \int_\Omega K(s,x, \omega) \varphi(x) dx = \gamma e^{-\gamma s} \int_{\Omega_0} \delta_{f(X(s,x))}(\omega) \varphi(x) dx\\
&=  \gamma e^{-\gamma s} \int _{\Omega} \delta_{f(y)}(\omega)\frac{g(X(-s,y))}{g(y)} \varphi(X(-s,y) ) dy \\
&= \gamma  e^{-\gamma s} \left[ \int _{\omega \cap [m,r] } \left( - \frac{\varphi(X(-s, f_1^{-1}(z) )) g(X(-s,f_1^{-1}(z)))}{g(f_1^{-1}(z)) f'(f_1^{-1}(z))} \right. \right.  \\
&+ \left.  \frac{\varphi(X(-s, f_2^{-1}(z) )) g(X(-s,f_2^{-1}(z)))}{g(f_2^{-1}(z)) f'(f_2^{-1}(z))}\right) dz \\
&+  \left.  \int _{\omega \cap [r,M] }   \frac{\varphi(X(-s, f_2^{-1}(z) )) g(X(-s,f_2^{-1}(z)))}{g(f_2^{-1}(z)) f'(f_2^{-1}(z))}dz\right] 
\end{align*} 
Hence the density of the measure \eqref{int AC} is \eqref{Ktilde imm}.}

\eu{By the definition of $\tilde{K}$ we deduce the following. }
\begin{lemma}
The kernel $\tilde{K} $  satisfies Assumption \ref{rotation}.
\end{lemma}

\begin{lemma}
The operator $\mathbb K_\lambda$ is compact for every $\lambda \in \Delta $ and non-supporting for every $\lambda \in \Delta \cap \mathbb R$.
\end{lemma}
\begin{proof}
By the definition of $\tilde{K}$ and by the change of variables 
\[y=X(-s,f^{-1}_1(z))\] we deduce that $\mathbb K_\lambda $ is equal to 
\eu{\begin{align*}
\frac{1}{\gamma} \mathbb K_\lambda\varphi(z) = \begin{cases}
 D_1(z) \int^M_{f_1^{-1}(z) } e^{-(\lambda + \gamma) \tau(x, f^{-1}_1(z))} \varphi(y ) dy &\\
+ D_2(z) \int^M_{f_2^{-1}(z)}  e^{-(\lambda + \gamma) \tau(x, f^{-1}_2(z) )} \varphi(y) dy &  m<z<r \\
D_2(z) \int^M_{f_2^{-1}(z)}  e^{-(\lambda + \gamma) \tau(x, f^{-1}_2(z) )} \varphi(y) dy & r<z<M
\end{cases} 
\end{align*}
where 
\[
D_1(z) = \beta_1(z)>0 \text{ and } D_2(z) = - \beta_2(z) >0. 
\]
}
\eu{Hence $\mathbb K_\lambda $ is the sum of the three operators $\mathbb K^i_\lambda $ defined as 
\[
\mathbb K^1_\lambda \varphi(z)=
\begin{cases}
\gamma D_1(z)  \int^M_{f_1^{-1}(z) }  e^{-(\lambda + \gamma) \tau(x, f^{-1}_1(z))} \varphi(y)dy &  m < z < r  \\
0 & r<z<M
\end{cases} 
\]
\[
\mathbb K^2_\lambda \varphi(z)=
\begin{cases}
\gamma D_2(z)  \int^M_{f_2^{-1}(z) }  e^{-(\lambda + \gamma) \tau(x, f^{-1}_2(z))} \varphi(y)dy  &  m < z < r  \\
0 & r<z<M
\end{cases} 
\]
and 
\[
\mathbb K^3_\lambda \varphi(z)=
\begin{cases}
\gamma D_2(z) \ \int^M_{f_2^{-1}(z) } e^{-(\lambda + \gamma) \tau(x, f^{-1}_2(z))} \varphi(y)  dy &  r< z < M \\
 0 &  z < r\\
\end{cases} 
\]
Since ${\mathbb K_\lambda} = \sum_{i =1}^3 \mathbb K_\lambda^i $ if we prove that for $i=1,2,3$ the operator ${\mathbb K^{i}_\lambda}$ is compact, then we deduce that $\mathbb K_\lambda$ is compact. }

\eu{To prove that each ${\mathbb K_\lambda^i} $ is compact we apply Lemma \ref{lemma2}. 
We describe in detail how to prove that $\mathbb K_\lambda^1 $ is compact. 
Consider $\delta < \min\{m ,r-m\}$. 
If  $z \in [m, r-\delta]$,  then since $f^{-1}_1 (z) > f_1^{-1}(z+\delta)$
\begin{align*}
 & \frac{1}{\gamma}  \left|  {\mathbb K_\lambda^1} \varphi(z+ \delta) - {\mathbb K_\lambda^1} \varphi(z)  \right| \\
&  =  \left| D_1(z) \int_{f^{-1}_1(z)}^M e^{-(\lambda + \gamma) \tau(x,  f_1^{-1}(z))} \varphi(x) dx \right.   \\
&\left. - D_1(z+ \delta ) \int^M_{f_1^{-1}(z+ \delta)} e^{-(\lambda + \gamma) \tau(x,  f_1^{-1}(z+\delta))} \varphi(x) dx \right| \\
&  \leq  \int_{f^{-1}_1(z)}^M   \left| D_1(z)  \ e^{-(\lambda + \gamma) \tau(x,  f_1^{-1}(z))} - D_1(z+ \delta ) e^{-(\lambda + \gamma) \tau(x,  f_1^{-1}(z+\delta))} \right| \varphi(x)  dx \\
&+  \int^{f^{-1}_1(z)}_{f^{-1}_1(z+\delta)} D_1(z+ \delta ) e^{-(\lambda + \gamma) \tau(x,  f_1^{-1}(z+\delta))} \varphi(x) dx. 
\end{align*}
\eu{On the other hand if $z \in [ r-\delta, r]$ we have that 
\begin{align*}
 & \frac{1}{\gamma}  \left|  {\mathbb K_\lambda^1} \varphi(z+ \delta) - {\mathbb K_\lambda^1} \varphi(z)  \right| \\
&  =   D_1(z) \int_{f^{-1}_1(z)}^M e^{-(\lambda + \gamma) \tau(x,  f_1^{-1}(z))} \varphi(x) dx 
\end{align*}}
Hence 
\begin{align}\label{horrible ineq} 
 & \frac{1}{\gamma}  \left\|  {\mathbb K_\lambda^1} \varphi(\cdot + \delta) - {\mathbb K_\lambda^1} \varphi(\cdot )  \right\|_1  \nonumber \\
&\leq \int_m^{r-\delta} \int_{f^{-1}_1(z)}^M   \left| D_1(z)  \ e^{-(\lambda + \gamma) \tau(x,  f_1^{-1}(z))} \right. \\
&\left.  - D_1(z+ \delta ) e^{-(\lambda + \gamma) \tau(x,  f_1^{-1}(z+\delta))}  \right| \varphi(x)dx dz \nonumber \\
&+  \int_m^{r-\delta} \int^{f^{-1}_1(z)}_{f^{-1}_1(z+\delta)} D_1(z+ \delta ) e^{-(\lambda + \gamma) \tau(x,  f_1^{-1}(z+\delta))}  \varphi(x) dx dz \nonumber \\
&+  \int_{r-\delta}^r D_1(z) \int_{f^{-1}_1(z)}^M e^{-(\lambda + \gamma) \tau(x,  f_1^{-1}(z))} \varphi(x) dx dz. \nonumber
\end{align} 
Since $f_1^{-1}(z) \in [0,x_c],$ and  $f_1^{-1}(z+ \delta) \in [0,x_c] $ if $z \in [m,r]$, then  
\[
D_1(z)=- \frac{1}{\alpha_1}  \frac{1}{g(f_1^{-1}(z))} \text{ and } D_1(z+ \delta)=- \frac{1}{\alpha_1} \frac{1}{g(f_1^{-1}(z+\delta))}. 
\]
 Using similar arguments to the one used in the proof of Lemma \ref{lem:comp non supp equal} we estimate the first term of inequality \eqref{horrible ineq} in the following way
\begin{align*}
&\int_m^{r-\delta} \int_{f^{-1}_1(z)}^M   \left| D_1(z)  e^{-(\lambda + \gamma) \tau(x,  f_1^{-1}(z))} \right. \\
& \left.  - D_1(z+ \delta ) e^{-(\lambda + \gamma) \tau(x,  f_1^{-1}(z+\delta))}  \right| \varphi(x) dx dz \\
& \leq \frac{1}{\alpha_1} \int_m^{r-\delta} \int_{f^{-1}_1(z)}^M e^{-(\lambda + \gamma) \tau(x,  f_1^{-1}(z+\delta))}  \left| \frac{1}{g(f^{-1}_1(z))}   e^{-(\lambda + \gamma) \tau(f^{-1}_1(z),  f_1^{-1}(z+\delta))} \right. \\
& \left. - \frac{1}{g(f^{-1}_1(z+\delta ))} \right|  \varphi(x) dx dz \leq c \delta \|  \varphi \|_1
\end{align*}
where we have used the uniform continuity of $g$ on compact intervals and the Lipschitz continuity of the function $\tau$ in the second argument. 
On the other hand, using the expression for $f_1 $ we deduce that the second term in inequality \eqref{horrible ineq} can be estimated in the following way
\begin{align*}
&  \int_r^{m-\delta} \int^{f^{-1}_1(z)}_{f^{-1}_1(z+\delta)} D_1(z+ \delta ) e^{-(\lambda + \gamma) \tau(x,  f_1^{-1}(z+\delta))}  \varphi(x)  dx dz \\
& \leq c \sup_{ x \in [m,M]}{\frac{1}{g(x)} }  \int_m^{r-\delta} \int^{-\frac{z + \delta }{\alpha_1}+ \frac{q_1}{\alpha_1}}_{-\frac{z  }{\alpha_1}+ \frac{q_1}{\alpha_1}} \varphi(x)  dx dz\\ 
& \leq c \sup_{ x \in [m,M]}{\frac{1}{g(x)}}  \int_m^{r} \int_{q_1 -\delta - \alpha_1 x }^{q_1 -x\alpha_1} dz  \varphi(x)   dx \leq c \delta \|\varphi \|_1 \\ 
\end{align*}
for a suitable constant $c>0.$
Finally the third term in inequality \eqref{horrible ineq} can be estimated in the following way
\[
\int_{r-\delta}^r D_1(z) \int_{f^{-1}_1(z)}^M e^{-(\lambda + \gamma) \tau(x,  f_1^{-1}(z))} \varphi(x) dx dz \leq c' \delta \| \varphi\|_1.  
\]
}

As a consequence, for every $\varepsilon >0$ there exists a $\delta_\varepsilon >0$ such that, for every $\delta < \delta_\varepsilon$
\[
\| {\mathbb K_\lambda^1} \varphi(\cdot+ \delta) - {\mathbb K_\lambda^1} \varphi(\cdot) \|_1 < \varepsilon \|\varphi\|_1
\]
and, hence the operator ${\mathbb K_\lambda^1}$ is  compact. 
The same technique can be used to prove that the operators $\mathbb K_\lambda^2 $ and $\mathbb K_\lambda^3$  are compact, hence the operator ${\mathbb K_\lambda}$ is compact for every $\lambda \in \mathbb C$ with $\Re \lambda > -\gamma $.

Now we have to prove that the operator $\mathbb K_\lambda$ is non-supporting for every $\lambda \in \Delta \cap \mathbb R$.
For every $\varphi\in X_+ =L^1_+([m,M])$ there exists a $ x_0 \in (m, M)$ such that $\varphi \neq 0$ for a subset $S$ of $[x_0, M]$ with positive measure. 
On the other hand for every $F \in L^\infty_+(\Omega_0)= L^\infty_+([m, M])$ there exists a set $S_F$ of positive Lebesgue measure on which $F$ is strictly positive.
If $x< x_0$ we have that 
\begin{align*}
&(\mathbb K_\lambda \varphi )(x) \geq  \frac{-1}{\alpha_2g(f_2^{-1}(x))} \int_{f_2^{-1}(x)}^{M} e^{-(\gamma +\lambda)\tau(z,f_2^{-1}(x))}\varphi(z)dz \\
& \geq  \frac{-1}{\alpha_2g(f_2^{-1}(x))} \int_{x}^{M} e^{-(\gamma +\lambda)\tau(z,f_2^{-1}(x))} \varphi(z) dz \\
&\geq  \frac{-1}{\alpha_2g(f_2^{-1}(x))}\int_{x_0}^{M} e^{-(\gamma +\lambda)\tau(z,f_2^{-1}(x))}\varphi(z)dz\\ 
\end{align*}
then $\mathbb K_\lambda \varphi(x)>0$. 

Assume now that $x > x_0$. Also in this case we have that 
\begin{align*}
&(\mathbb K_\lambda \varphi )(x) \geq \frac{-1}{\alpha_2 g(f_2^{-1}(x))} \int_{f_2^{-1}(x)}^{M} e^{-(\gamma +\lambda)\tau(z,f_2^{-1}(x))}\varphi(z) dz . 
\end{align*}
Thanks to the fact that $f_2$ is monotonically increasing we deduce that, if $x< f_2(x_0)$ then $f_2^{-1}(x)< x_0$ and hence 
$\mathbb K_\lambda \varphi (x)>0$ for every $x< f_2(x_0) $. 
Iterating this argument we deduce that $\mathbb K_\lambda \varphi (x)>0$ for every $x< f^n_2(x_0) $. 

Since for every $z<M$ there exists an $\overline n $ such that $f_2^{\overline n}(x_0)> z.$ 
we deduce that  there exists an $\overline n$ such that the set $\{x < f^{\overline n}(x_0)\} \cap S_F $ has positive measure and therefore $\mathbb K_\lambda$ is non-supporting. 
\qed\end{proof}
\begin{proposition}
	Let $g , \mu, \Lambda, \nu $ be such that Assumptions \ref{ass:immunity} hold. 
	Let $\psi_0$ be the stable distribution.
	The Malthusian parameter is equal to $0$ and $R_0=1$.  
	The solution $B$ of \eqref{RE} satisfies 
	\[
	\|  B(t, \cdot) - c\Psi_0  \|\leq M e^{-v t}
	\] 
	where $c, v, M >0$ and $\Psi_0(dx)=\psi_0(x) dx $ the norm $\| \cdot \|= \| \cdot \|_{TV} = \| \cdot \|_{\flat}$. 
\end{proposition}

\section{Relation between the PDE formulation and the RE} \label{sec:PDE}
In this section we prove asynchronous exponential growth/decline for the population distribution for the model examples introduced in Section \ref{sec:asy models}. 

We assume that the kernel $K$ is defined by \eqref{kernel} for parameters satisfying one among the three Assumptions \ref{ass: nu AC}, \ref{ass:equal fission}, \ref{ass:immunity}; hence $K$ is a $z_0$-bounded regularizing kernel and induces via  formula (\ref{eq:ktilde})  an operator kernel $\tilde{K}$ that satisfies either Assumption \ref{rotation} or Assumption \ref{rot density}.  The operator kernel $\tilde{K}$ in turn induces  the discounted next generation operator $\mathbb K_\lambda$ through the Laplace transform (\ref{NGO lambda}). 
We assume that  $\mathbb K_\lambda$  is non-supporting for every $\lambda \in \Delta \cap \mathbb R$ and compact for every $\lambda \in \Delta$.

\eu{\subsection{From the population birth rate to the population distribution}}
\eu{We start by making the connection between the renewal equations and the partial differential equations formalising the model examples presented in Section \ref{sec:examples}. }
\eu{Let $M(t, \omega)$ be the number of individuals in the population with state in the set $\omega$ at time $t.$ 
Assume that at time $t=0$ we have  $M(0,\cdot) =M_0$ with $M_0 \in \mathcal M_{+,b}(\Omega)$. 
Then the number of individuals, born before time zero, with state in the set $\omega $ at time $t$ is equal to
 \[
\int_{\Omega } \mathcal F(t,x) \delta_{X(t,x)}(\omega ) M_0(dx). 
\] 
On the other hand, the number of individuals, born after time zero, with state in the set $\omega $ at time $t$  equals
\[
\int_0^t \int_{\Omega_0} B(t-a, d\xi) \mathcal F(a, \xi) \delta_{X(a,\xi)}(\omega) da, 
\]
where $B$ is the population birth rate.
The two above observations lead to the following expression of $M$ in terms of $B $ and $M_0$}
\begin{equation}\label{eq:m in terms of b}
M(t, \omega)=\int_0^t \int_{\Omega_0} B(t-a, d\xi) \mathcal F(a, \xi) \delta_{X(a,\xi)}(\omega) da + \int_{\Omega } \mathcal F(t,x) \delta_{X(t,x)}(\omega ) M_0(dx).
\end{equation}
Once $B$ has been solved from the renewal equation \eqref{RE}, the formula \eqref{eq:m in terms of b} is an explicit formula for $M$.

\eu{
An alternative way to define $M $, see for instance \cite{canizo2021spectral}, is to define it by duality with the solution of the backward equation corresponding to \eqref{PDE growth-fission den}, \eqref{eq:equal fission den}, \eqref{eq:waning and boosting den}, that reads, in its general form, as 
\begin{equation}\label{backward PDE} 
\partial_t \varphi(t, x) =  g(x)  \partial_x \varphi(s,x)-\tilde{\mu}(x)  \varphi(s,x)  +  \int_{\Omega_0} \varphi(s,\eta) \nu(x, d\eta)  \Lambda(x). 
\end{equation}
}

\eu{If $M$ is differentiable this amounts to define $M$ as the function that satisfies 
	\begin{align}\label{PDE dual}
	& \int_\Omega \frac{d}{dt} M(t, dx) \varphi( x) =   \int_\Omega \left(  g(x)  \partial_x \varphi(x)-\tilde{\mu}(x)  \varphi(x)  \right) M(t, dx)   \\
	&+\int_\Omega \left( \int_{\Omega_0} \varphi(\eta) \nu(x, d\eta)\right)  \Lambda(x)   M(t , dx ) \nonumber
	\end{align}
for every $\varphi \in C^1_c(\mathbb R_+). $
We provide more details on how to interpret the term $\int_\Omega \frac{d}{dt} M(t, dx) \varphi( x)$ in Appendix \ref{sec:well posed PDE} (proof of Proposition \ref{prop:PDE exists}). }

\eu{If $M$ is not differentiable, as in the present case, this alternative way cannot be used.
However, $M$ can still be defined as the solution of an equation, namely equation \eqref{PDE} below, which can be seen as a weak version of the PDEs \eqref{PDE growth-fission den}, \eqref{eq:equal fission den} , \eqref{eq:waning and boosting den} when $\nu $ is  equal to \eqref{density h}, \eqref{nu AC fission} or \eqref{nu AC immune}, respectively.}

\begin{proposition}\label{prop:PDE}
Assume $\mu, g , \Lambda, \nu $ are either as in Assumptions \ref{ass: nu AC},  Assumption \ref{ass:equal fission} or Assumption \ref{ass:immunity}. 

The function $M$, defined by equation \eqref{eq:m in terms of b}, is the unique function mapping $\mathbb R_+ \times \mathcal B(\Omega) $ into $\mathbb R_+$, that satisfies the following  equation for every $\varphi \in C^1(\mathbb R_+ , C^1_c(\Omega))$ 
	\begin{align}\label{PDE}
& \int_\Omega \varphi(t, x) M(t, dx) -\int_\Omega \varphi(0, x) M_0( dx)  
- \int_0^t  \int_\Omega \partial_s \varphi(s, x) M(s, dx) ds \\
	&=   \int_0^t \int_\Omega \left(  g(x) \partial_x \varphi(s,x) -\tilde{\mu}(x)  \varphi(s,x)  \right) M(s, dx) ds  \nonumber \\
	&+ \int_0^t \int_\Omega \left( \int_{\Omega_0} \varphi(s,x) \nu(\eta, dx)\right)  \Lambda(\eta)   M(s , d\eta ) ds \nonumber
	\end{align}
	and the initial condition $M(0, \cdot)=M_0(\cdot).$
\end{proposition} 
We refer to the appendix for the proof of this proposition. 

\eu{Equation \eqref{PDE} is used in the literature, see for instance \cite{debiec2018relative} and also \cite{gabriel2018measure} in a slightly different situation. It is not entirely intuitive because it relies  on  what one could call a double duality (both in the state variable and in the time variable). }
In our approach we use the interpretation to define $M$ by \eqref{eq:m in terms of b}. In the next section we shall deduce the large time behaviour of $M$ from that of $B$ in a simple natural way. So there is no need to formulate a PDE for $M$ and to specify in which sense we solve it. Our motivation to, nevertheless, formulate and prove Proposition \ref{prop:PDE} is simply to show that our constructively defined $M$ does indeed coincide with $M$ as defined in other works.

By the interpretation of  $\mathcal F(t, x)$  as the survival probability we expect $\mathcal F(t, x)$ to tend to zero as time tends to infinity. 
\eu{Indeed, it follows from the assumption on the model parameters that we have made in this section that there exists a constant $C>0$ such that
\begin{equation}\label{bound F}
\sup_{ x \in \Omega_0} \mathcal F(t, x) \leq C e^{z_0 t} \quad  {\rm for \,  all} \,\, t>0.
\end{equation}
} 
The exponential bound \eqref{bound F} is crucial in deducing the asymptotic behaviour of the solution $B$ of the renewal equation \eqref{RE} and subsequently of $M$ defined in terms of $B$ in  \eqref{eq:m in terms of b}. This will be done in the next subsection.

\subsection{Asymptotic behaviour of the solution of the PDE}

We now state and prove our result on asynchronous exponential growth/decline of the population distribution $M$.

\begin{theorem}\label{gener asympt}
Assume that $\mu, g , \Lambda, \nu $ are either as in Assumptions \ref{ass: nu AC} ,   Assumption \ref{ass:equal fission}, or Assumption \ref{ass:immunity}. 
	Let $M$ be given by \eqref{eq:m in terms of b} and let $\psi_r$ and $r$ be as in Corollary \ref{cor:measure asympt AC}. 
	Then there exist a constant $C>0$ and a constant $\ell>0 $ such that
	\begin{equation} \label{eq m expo convergence}
	\left\| e^{-r t} M(t, \cdot ) - c M_{\psi_r} ( \cdot) \right\| \leq C e^{- \ell t} \quad t >0 .
	\end{equation}
	where $c>0 $ and $M_{\psi_r} $ is defined by
	\begin{equation}\label{Mpsi}
	M_{\psi_r}(\omega) := \int_0^\infty  \int_{\Omega_0} e^{-a r } \psi_r(\xi) \mathcal F(a, \xi) \delta_{X(a,\xi)}(\omega) d\xi da \quad \omega \in \mathcal B(\Omega)
	\end{equation}
and 	where $\|\cdot  \|= \| \cdot  \|_{TV} =\| \cdot \|_{\flat}. $
\end{theorem} 
\begin{proof} 
\eu{We start by introducing some useful notation.}
	Let $\tilde{M}^{AC}$ be the measure defined by 
	\[
	\tilde{M}^{AC}(t, \omega): =\int_0^t \int_{\Omega_0} b(t-a)(\xi) \mathcal F(a, \xi) \delta_{X(a,\xi)}(\omega) d\xi da
	\] 
	where $b$ is the density of $B^{AC}.$
	In analogy we define $\tilde{M}^s $ as follows
	\[
	\tilde{M}^{s}(t, \omega): =\int_0^t \int_{\Omega_0} B^s(t-a, d \xi) \mathcal F(a, \xi) \delta_{X(a,\xi)}(\omega) da. 
	\] 
 With $\tilde{M}$ we denote the measure defined by
	\[
	\tilde{M}(t,\omega) : = \int_{\Omega } \mathcal F(t,x) \delta_{X(t,x)}(\omega ) M_0(dx).
	\] 
	Notice that 
	\begin{align*}
	& \left\| M(t, \cdot ) - c e^{r t} M_{\psi_r} (\cdot)\right\| \leq \left\| M(t, \cdot ) - \tilde{M}^{AC}(t,\cdot)  \right\|+ 
	\left\| \tilde{M}^{AC}(t,\cdot)  - c e^{r t} M_{\psi_r} ( \cdot)\right\| \\
	& \leq \left\| \tilde{M}(t, \cdot) \right\| + \left\| \tilde{M}^s(t, \cdot ) \right\|+ \left\| \tilde{M}^{AC}(t,\cdot)  -c  e^{r t} M_{\psi_r} (\cdot)\right\|. 
	\end{align*} 
\eu{Now we estimate all these terms.} In the estimates that follow we shall denote by $C$ a suitably chosen positive constant the value of which may change from line to line.
	Thanks to \eqref{bound F}  we have
	\begin{align*}
	& \left\| \tilde{M}(t, \cdot ) \right\|= \int_{\Omega_0} M_0(d\xi) \mathcal F(t, \xi) \delta_{X(t,\xi)}(\Omega) \leq  \int_{\Omega_0} M_0(d\xi) \mathcal F(t, \xi) \leq C \| M_0 \| e^{z_0 t}
	\end{align*}
for every $t \geq 0.$
	It follows from \eqref{Bs tends to zero exponentially} and \eqref{bound F}  that 
	\begin{align*}
	& \left\| \tilde{M}_s(t, \cdot ) \right\|=\int_0^t \int_{\Omega_0} B^{s}(t-a, d\xi) \mathcal F(a, \xi) \delta_{X(a,\xi)}(\Omega) da  \\
	& \leq \int_0^t \left( c_1 e^{ z_0 (t-a)} + c_2 t e^{z_0 (t-a)} \right) \sup_{\xi \in \Omega } \mathcal F(a, \xi) da \\
	&\leq \left( c_1  e^{ z_0 t} + c_2  t e^{z_0 t} \right) t. 
	\end{align*}
Let  $m$ be the density of $\tilde{M}^{AC}$ and let $m_{\psi_r}$ be the density of $M_{\psi_r}$.
\eu{This means that
\[
m_{\psi_r}(y):=\int_0^\infty e^{-r a } \psi_r(X(-a,y)) \mathcal F(a, X(-a,y)) \frac{g(y)}{g(X(-a,y))}  da
\]
and 
\[
m(t,y):=\int_0^t b(t-a)(X(-a,y)) \mathcal F(a, X(-a,y))  \frac{g(y)}{g(X(-a,y))}  da, 
\]
where $b(t-a)(X(-a,y))$ is the evaluation in $X(-a, y)$ of the function $b(t-a)$. } 
By the definition of the total variation norm and of the flat norm, we then have 
	\[
	\left\| \tilde{M}^{AC}(t,\cdot) - c e^{r t} M_{\psi_r} ( \cdot)\right\| = \left\| m(t,\cdot)  - c e^{r t}m_{\psi_r} ( \cdot)\right\|_1.
	\] 
\eu{	Notice that by the change of variable $X(-a,y)=x$ we get }
	\begin{align*}
	&\left\|m(t,\cdot)  - c e^{rt} m_{\psi_r} (t, \cdot)\right\|_1 \leq  e^{r t} \int_t^\infty \int_{\Omega_0} \psi_r(x) e^{-r a}\mathcal F(a,x) dxda \\
	&+  \int_0^t \int_{\Omega_0} \left| b(a)(x) - c e^{r a} \psi_r(x) \right| \mathcal F(t-a,x) dx da 
	\end{align*} 
\eu{	The fact that $\mathcal F$ satisfies \eqref{bound F} and the fact that $r > z_0$ imply 
	\[
     e^{r t} \int_t^\infty \int_{\Omega_0} \psi_r(x) e^{-r a}\mathcal F(a,x) dxda \leq C e^{z_0 t}.
	\]
} 

	Recall that by Corollary \ref{cor:measure asympt AC}  there exists a constants $C>0$ and $ r>k>0$ such that
	\[ 
	\left\| b(t) - c e^{r t } \psi_r \right\|_1 \leq C e^{-k t+ r t } \text{ for every } t >0.
	\] 
	\eu{Therefore, using \eqref{bound F}, as well as the fact that $r-k>0$ and, hence, $\max_{a \in [0,t]} e^{(r-k) a} =  e^{(r-k) t} $ we deduce that} 
	\begin{align*}
	& \int_{0}^t \int_{\Omega_0} \left| b (a)(x) - c e^{r a} \psi_r(x) \right| \mathcal F(t-a,x) \ dx da \\
	& \leq C \int_{0}^t  e^{(r-k) a} \sup_{x \in \Omega} \mathcal F(t-a,x) da \leq C  e^{(r-k) t} \int_0^\infty e^{z_0 a } da.
	\end{align*} 
	Combining all the bounds that we have for $\| \tilde{M}(t, \cdot)\|$, $\|m^s(t, \cdot)\|$ and $\|\tilde{M}^s(t, \cdot) - e^{- r t} \psi_r (\cdot) \|$ we deduce that 
	\[
	\left\|  M(t, \cdot ) - c e^{rt} M_{\psi_r} ( \cdot) \right\| \leq C e^{(r-\ell)  t}
	\]
	fore some $\ell >0$, that is,  \eqref{eq m expo convergence} holds.
\qed\end{proof}

\section{Concluding remarks}
Models of physiologically structured populations can be formulated from first principles as renewal equations for the population birth rate $B$, which takes on values in the space of measures on the set of admissible states-at-birth \cite{DGG2007,franco2021one}.
In this paper we proved the asynchronous exponential growth of the measure-valued solution $B$ of the renewal equation, \eqref{RE} under a regularisation assumption on the kernel $K.$
This  assumption enabled us to derive the asymptotic behaviour of $B $ from the behaviour of its absolutely continuous part  $B^{AC}$. 
 \eu{Moreover, using the regularisation assumption, we proved that also the density of $B^{AC}$ satisfies a renewal equation}. 
We studied the long term behaviour of this density by way of Laplace transform methods.  

We applied our results to a model of cell growth and fission (either into equal or unequal parts) and to a model of waning and boosting of the level of immunity against a pathogen. 
For these examples we then used the interpretation to express in Equation \eqref{eq:m in terms of b} the population state $M$, that is, the distribution of individual states, in terms of the population birth rate $B$.   If we assume that the values of $B(t, \omega)$ for $t<0$ are given, we can write \eqref{eq:m in terms of b} as follows:
\eu{\begin{equation} \label{M function of B tr inv}
M(t, \omega)=\int_0^\infty \int_{\Omega_0} B(t-a,d\xi) \mathcal F(a,\xi) \delta_{X(a,\xi)}(\omega) da.
\end{equation}
}
Vice versa,  we can express $B$ in terms of of $M$:
\begin{equation} \label{B function of M tr inv}
B(t, \omega)= \int_{\Omega} \Lambda(\eta) \nu(\eta, \omega) M(t, d\eta). 
\end{equation}
Combining Equations \eqref{M function of B tr inv} and \eqref{B function of M tr inv} we deduce the translation invariant formulation  
\begin{equation}\label{tr inv B}
B(t, \omega)=\int_0^\infty \int_{\Omega_0} B(t-a, d\xi) K(a,\xi, \omega) da 
\end{equation} 
of Equation \eqref{RE}.

When we solved \eqref{tr inv B} and then used \eqref{M function of B tr inv} to define $M$, we actually solved a PDE, the weak version of which is \eqref{PDE} (see Section 6 of \cite{franco2021one} for general remarks about the way $RE$ arise when solving certain types of PDE). 
 However, as noted above, there  is no need to write down the PDE itself and to specify the notion of solution, nor to rigorously prove the existence of such a solution, since the interpretation justifies our conclusions. 
So guided by the interpretation we determined the asymptotic behaviour of the population distribution efficiently using \eqref{M function of B tr inv} and thus avoided demanding technicalities associated with PDEs.

When the measure $M(t, \cdot)$ is absolutely continuous with respect to the Lebesgue measure,  it is simpler to write down the PDE.  For the model of cell growth and fission into equal parts it takes the form \eqref{eq:equal fission den}, for fission into unequal parts it becomes \eqref{PDE growth-fission den} and for the model of waning and boosting we have \eqref{eq:waning and boosting den}. These PDEs have been treated for instance in \cite{diekmann1984stability}, \cite{heijmans1984stable} and \cite{diekmann2018waning}, respectively.

\eu{The corresponding backward formulation of these equations is \eqref{backward PDE}. 
When $\Omega \subset \mathbb R$, then the measure $M(t, \cdot) $ can be represented by the NBV function $N$  defined by
\[
N(t,x):=\int_{[0,x]} M(t, d\eta). 
\]
The function $N$ satisfies the following forward equation
\begin{equation}
\partial_t N(t,x)= - g(x) \partial_x N(t,x)- \int_{[0,x]} \tilde{\mu}(\xi) N(t,d\xi) + \int_\Omega \Lambda(\eta) \nu(\xi, [0,x]) N(t, d\xi). 
\end{equation} 
}

﻿ ﻿The regularization assumption on the kernel $K$ entails, of course, a restriction concerning the class of models that is covered. For the special example of fission into two equal parts, it is shown in Section II.12 of \cite{metz2014dynamics} that one can establish convergence to an absolutely continuous stable distribution under a relaxed regularity condition. So there is definitely room for deriving sharper results. On the other hand, it is known that a stable distribution may have a non-zero singular component. Indeed, this can happen in the context of the selection-mutation balance as analyzed in \cite{ackleh2016population,burger1996stationary}. We now briefly comment on the similarities and differences of the models considered here and those treated in \cite{ackleh2016population,burger1996stationary}.
	In \cite{burger1996stationary}, the nonlinearity is of the ‘replicator’ type, meaning that it is due to dividing by the total population size, so due to working with relative magnitudes. In the context of our framework we can, if we wish, do the same. In \cite{ackleh2016population} the per capita birth and death rates are allowed to depend on the total population size (see below for a more general setup). Neither \cite{ackleh2016population} nor \cite{burger1996stationary} considers a dynamical trait and both implicitly assume that the survival probability as a function of age is an exponential function. In these respects the model considered here is far more general. In \cite{burger1996stationary} mutation is incorporated as a random change of trait, while in \cite{ackleh2016population} it is incorporated as production of offspring with a different trait. If in our framework we put $K=K_1+K_2$, with $K_1( a,\xi, . )$ equal to a $(a,\xi)$ dependent multiple of the Dirac measure concentrated in $\xi$ (describing production of offspring with exactly the same trait) and $K_2(a,\xi, . )$ equal to a $(a,\xi)$ dependent multiple of a fixed absolutely continuous probability distribution, we obtain a “house-of-cards” type model in the spirit of Section 4 of \cite{burger1996stationary}. Perhaps one can do a lot of more or less explicit calculations for this special case of a one-dimensional range perturbation of a rather degenerate kernel $K_1$, but this has not been done so far.

It is a challenge to extend the analysis developed in this paper to the case of a structured population embedded in a non-constant environment which influences the evolution of the population and which in turn is influenced by feedback from the population.
An example of an environment for cell growth and fission is the amount of nutrient resources, as it is known that the availability of nutrients affects both the growth and fission rates \cite{Monahan2014}.
In the waning and boosting context, the force of infection $\gamma $ is the most relevant environmental variable. 

Let us denote the environment by $E$.  
The evolution in time of $(B(t), E(t))$ is given by the following system of equations 
\begin{equation}  \label{foor RE} 
B(t, \omega)=\int_0^\infty \int_{\Omega_0} B(t-a, d\xi) K(a,\xi,E_t,E(t), \omega) da  
\end{equation} 
\begin{equation}  \label{food}
\frac{d}{dt}E(t)=f(E(t)) - \int_0^\infty \int_{\Omega_0} B(t-a, d\xi ) c(a,\xi, E_t, E(t)) da
\end{equation} 
where  $\frac{ d E(t)}{dt} = f(E(t))$ describes the evolution in time of the  environment in the absence of a consumer population and  $c(a,\xi, E_t, E(t))$ represents the influence on the environment of an individual born with state $\xi$, that at time $t$ has age $a.$
Both $c$ and the kernel $K$ in equation \eqref{foor RE} depend on $E(t)$ as well as on the {\em history} $E_t$, that is on all the values of $E$ before time $t$ and this dependence on $E$ introduces, via \eqref{food}, a non-linearity in  equation \eqref{foor RE}. 

It is an open problem to study the asymptotic behaviour of $(B(t),E(t)) $ under an (adapted) regularisation assumption on the kernel $K$ and to uncover the connection with the corresponding PDE formulation. 
The special case $\Omega_0=\{ x_0\}$ is elaborated in \cite{barril2021nonlinear}.

\appendix

\section{Notation}\label{sec:notation}
In this appendix we introduce the notation used in the paper.
We denote by $\mathbb R_+$ the set $[0, \infty)$ and by $\mathbb R_+^*$ the set $(0, \infty)$. Given a Borel measurable subset $A$ of $\mathbb R$ we denote by $\mathcal B(A)$  the $\sigma-$algebra of all Borel subsets of $A$.
$\mathcal M (A) $ is the set of the signed Borel measures on $A$,  $\mathcal M _{+}(A) $ is the cone of the positive measures and  $\mathcal M_{+,b} (A) $ the set of the positive and bounded measures on $A$. Furthermore,  $\mathcal M_{+,AC}(A)$ is the subset of the measures which are absolutely continuous with respect to the Lebesgue measure. 
We denote by $\mu^s$ the singular part of the measure $\mu$ and by $\mu^{AC}$ its absolutely continuous part, again with respect to the Lebesgue measure. We have $\mu = \mu^{s}+ \mu^{AC}.$
Finally, we denote by $|A|$ the Lebesgue measure of the Borel set $A$. 

The total variation norm  $\|\mu  \|_{TV}$ of a measure $\mu \in \mathcal M(A)$ is
defined by  
\[
\| \mu \|_{TV} = \sup_{\Pi} \sum_{i=1}^n |\mu(A^i) |,
\] 
where the supremum is taken over all the finite measurable partitions $\Pi:=\{A^1, \dots, A^n\} $ of the set $A$.
We denote by $BL(A)$ the space of the real valued bounded Lipschitz functions, endowed with the norm 
\[ 
\| f \|_{BL}:= \sup_{x \in A } |f (x)| + \sup_{x,y \in A:  x \ne y} \frac{|f(x)-f(y)|}{|x-y |}.
\]  
Finally,  the flat norm $\| \mu \|_{\flat}$ of a measure $\mu \in M(A)$ is defined by  
\[
\| \mu \|_{\flat} = \sup\left\{ \left|\int_{A} f d\mu \right| : f \in BL(A) \text{ such that } \| f \|_{BL} \leq 1 \right\}. 
\] 
For positive measures $\mu$ the equality$\| \mu \|_{\flat} = \| \mu \|_{TV}$ holds, see \cite{gwiazda2018measure}. 

We denote by $L^1_+(A) $ the set of the positive functions belonging to $L^1(A)$. The set $L^1_+(A) $ is a cone in $L^1(A)$. Similarly, we denote by $L^\infty_+(A) $ the cone of the positive functions belonging to $L^\infty(A)$.

For real numbers $\rho$, $L^1_\rho(A)$ is the space of the measurable functions $f: A \rightarrow \mathbb R$ such that 
\[
\int_A \left| f(a) e^{\rho a}\right| da < \infty.
\]

 $\mathcal L (X)$ is the space of the bounded linear operators from the normed linear space $X$ into itself equipped with the operator norm by $\| \cdot \|_{op}$.
The spectral radius of the linear operator $T$ is denoted by  $\rho(T)$.

\section{Proof of Proposition \ref{prop:PDE}} \label{sec:well posed PDE}

Let the assumptions made in the beginning of Section \ref{sec:PDE} hold.

\eu{To prove that $M$ defined by \eqref{eq:m in terms of b} is the unique solution of equation \eqref{PDE}, we use the Riesz–Markov–Kakutani representation theorem and \eqref{eq:m in terms of b} to identify $M(t, \cdot)$ with the  element 
\begin{equation}\label{M frechet}
T_B(t)+ T_{M_0}(t)
\end{equation} 
of $(C_c(\Omega))^*$,
where the function $T_B: \mathbb R_+ \rightarrow (C_c(\Omega))^*$ is defined by 
	\begin{equation}\label{eq:tb}
	T_B: t \mapsto\left( \varphi \mapsto \int_0^t \int_{\Omega_0} B(t-a, d\xi) \mathcal F(a, \xi) \varphi(X(a,\xi))  da  \right) \quad t \geq 0
	\end{equation}
while the function $T_{M_0}: \mathbb R_+ \rightarrow (C_c(\Omega))^*$ is defined by 
\begin{equation} \label{eq:tm0}
	T_{M_0}: t \mapsto\left( \varphi \mapsto \int_{\Omega} M_0( d\xi) \mathcal F(t, \xi) \varphi(X(t,\xi)) \right) \quad t \geq 0. 
	\end{equation}
}
We use this representation to  compute $\frac{d}{dt } M(t, \cdot)$.  To this end we will identify $T_B(t)$ and $T_{M_0}(t)$ with their restrictions on $C^1_c(\Omega). $
We start with an auxiliary lemma that explain how to compute $\frac{d}{dt} T_B(t)$ and $\frac{d}{dt} T_{M_0}(t).$
\begin{lemma}\label{lem:m a.e. differentiable} 
	The functions $T_B: \mathbb R_+ \rightarrow (C^1_c(\Omega))^*$ and $T_{M_0}: \mathbb R_+ \rightarrow (C^1_c(\Omega))^*$  defined  by \eqref{eq:tb} and \eqref{eq:tm0}, respectively,  are  a.e. differentiable and differentiable.
	For the values of time $t \geq 0$ for which $T_B$ is differentiable, its derivative $\frac{d}{dt} T_B (t) \in (C^1_c(\Omega))^*$ is given by
	\begin{align*}
	\frac{d}{dt} T_B (t) \varphi = F_B (t) \varphi \quad \text{ for every } \varphi \in C^1_c(\Omega),
	\end{align*} 
	where 
	\begin{align*}
	&F_B(t) \varphi:= \int_{\Omega_0} B(t, d\xi) \varphi(\xi) + \int_0^t \int_{\Omega_0} B(t-a, d\xi)  \mathcal F(a, \xi) G(\varphi)(X(a,\xi)) da  
	\end{align*} 
	with
	\begin{align*}
	G(\varphi) (x):= - \tilde{\mu} (x)\varphi(x)+g(x) \varphi'(x).
	\end{align*} 
The derivative $\frac{d}{dt} T_{M_0}(t) \in (C^1_c(\Omega))^*$ of  $T_{M_0}$  is given by
	\begin{align*}
	\frac{d}{dt} T_{M_0} (t) \varphi = F_{M_0} (t) \varphi \quad \text{ for every } \varphi \in C^1_c(\Omega),
	\end{align*} 
	where 
	\begin{align*}
	&F_{M_0}(t) \varphi:= \int_{\Omega}   \mathcal F(t, \xi) G(\varphi)(X(t,\xi)) M_0(d\xi). 
	\end{align*} 
\end{lemma} 
\eu{The proof of this lemma is technical as it deals with function with values in $(C^1_c(\Omega))^*$, 
but the result  is intuitively credible as it is formally obtained by simply applying Leibniz rule for differentiating under the integral sign.
}
\begin{proof}
	We start by proving that $T_B$ is  differentiable. 
	Notice that for every $\varphi \in C^1_c(\Omega)$,
\eu{	\begin{align*}
	& \left\| \frac{T_B (t) - T_B (t+h)}{h}-  F_B (t) \right\|_{op} = \sup_{\| \varphi \|_{C_c^1(\Omega)} \leq 1}  \left| \frac{T_B (t)\varphi  - T_B (t+h) \varphi}{h}-  F_B (t)\varphi  \right|.
	\end{align*} }
	By the Lebesgue point theorem we have that for almost every $t>0$
	\begin{align*}
	\lim_{h\rightarrow 0} \frac{1}{h} \int_t^{t+h} \int_\Omega B(a,dy) \mathcal F(t-a, y) \varphi(X(t-a,y)) da = \int_\Omega B(t,dy) \varphi(y)
	\end{align*} 
\eu{for every $\varphi \in C^1_c(\Omega)$ with $\|\varphi\|_{C_{c}^1(\Omega)}\leq 1. $}
	The fact that $\varphi$ is Lipschitz continuous and that for every $y \in \Omega_0$ the map $X(\cdot, y)$ is continuous implies that the convergence is uniform in $\varphi$.
	Let us illustrate why.
	Notice that 
	\begin{align*} 
&	 \frac{1}{h}\left| \int_t^{t+h} \int_\Omega B(a,dy) \mathcal F(t-a, y) \varphi(X(t-a,y)) da - \int_\Omega B(t,dy) \varphi(y)\right| \\
	& \leq \frac{1}{h} \int_t^{t+h} \left|\int_\Omega B(a,dy) \mathcal F(t-a, y) \varphi(X(t-a,y)) \right. \\
& \left. - \int_\Omega B(a,dy) \mathcal F(t-a, y) \varphi(y) \right| da \\
	& +  \frac{1}{h} \left|\int_t^{t+h} \int_\Omega B(a,dy) \mathcal F(t-a, y) \varphi(y) da - \int_\Omega B(t,dy)\varphi(y)\right| da \\
	 	&\leq \frac{1}{h} \int_t^{t+h}\int_\Omega B(a,dy) \left| \varphi(X(t-a,y)) - \varphi(y) \right| da \\
	 & +\frac{1}{h} \int_t^{t+h} \left| \int_\Omega B(a,dy) \mathcal F(t-a, y)  da - \int_\Omega B(t,dy) \right| da\\ 
	\end{align*}
The second term goes to zero a.e., uniformly in $\varphi$. 
Since $\varphi$ is compactly supported, and Lipschitz continuous, we have that for every $\varepsilon>0$ there exists a $\delta>0$ such that for every $h<\delta $ we have $\left| \varphi(X(h,y)) - \varphi(y) \right| \leq \|\varphi'\|_\infty \left| X(h,y)-y\right|< \varepsilon$.
It follows that
\begin{align*}
 \frac{1}{h} \int_t^{t+h}\int_\Omega B(a,dy) \left| \varphi(X(t-a,y)) - \varphi(y) \right| da \rightarrow 0
 \end{align*}
 uniformly in $\varphi$ with $\|\varphi \|_{C_c^1(\Omega)} \leq 1$ as $h \rightarrow 0.$
 
	On the other hand, by the dominated convergence theorem we deduce that
	\begin{align*} 
	&\lim_{h \rightarrow 0} \int_0^{t} \int_{\Omega_0} \frac{ \Delta_h \mathcal F \varphi(t,a,y)}{h}  G(\varphi)(X(a,y))  ) B(a, dy) da  \\
	&= \int_0^{t} \int_{\Omega_0}  \mathcal F(a, y) G(\varphi)(X(a,y))  ) B(a, dy) da  
	\end{align*} 
	where
	\[
	\Delta_h \mathcal F \varphi(t,a,y)= \mathcal F(t+h-a, y) \varphi(X(t+h-a, y)) -  \mathcal F(t-a, y) \varphi(X(t-a, y)). 
	\] 
	Since for every $y \in \Omega_0$ the map $X(\cdot, y)$ is continuous and since $\varphi \in C_c^1(\mathbb R)$, hence Lipschitz continuous function, the convergence is uniform in $\varphi.$

The proof of the fact that
\[
\frac{d T_{M_0}(t) \varphi}{dt}= F_{M_0}(t) \varphi
\]
is analogous and we omit it. 
\qed\end{proof}
\begin{proposition}\label{prop:PDE exists}
	The function $M$, defined by equation \eqref{eq:m in terms of b} satisfies \eqref{PDE} for every $\varphi \in C^1(\mathbb R_+, C^1_c(\Omega))$. 
\end{proposition}
\begin{proof}
	\eu{Integrating by parts we find that}
	\begin{align*}
	&\int_\Omega \varphi(t, x) M(t, dx) -\int_\Omega \varphi(0,x) M_0(dx)- \int_0^t  \int_\Omega \partial_s \varphi(s, x) M(s, dx) ds \\
	&= \int_0^t \int_\Omega \varphi(s, x) \frac{d}{ds} M(s, dx) ds 
	\end{align*} 
\eu{where the term $\int_\Omega \varphi(s, x) M(s, dx)$ is equal to $F_B(s) \varphi(s,\cdot)+ F_{M_0}(s)\varphi(s,\cdot)$. Hence thanks to Lemma \ref{lem:m a.e. differentiable}}
	\begin{align*}
	& \int_0^t \int_\Omega \varphi(s, x) \frac{d}{ds} M(s, dx) ds = \int_0^t \int_{\Omega_0} B(t, d\xi) \varphi(s, \xi) ds \\
	&+ \int_0^t \int_0^s \int_{\Omega_0} B(s-a, d\xi)  \mathcal F(a, \xi) G(\varphi(s,\cdot))(X(a,\xi))da ds  \\
	&+ \int_0^t \int_\Omega \mathcal F(s,x) G(\varphi(s,\cdot))(x) M_0( dx) ds.
	\end{align*} 
\eu{This implies that 
\begin{align*}
	&\int_\Omega \varphi(t, x) M(t, dx) -\int_\Omega \varphi(0,x) M_0(dx)- \int_0^t  \int_\Omega \partial_s \varphi(s, x) M(s, dx) ds \\
&=  \int_0^t \int_{\Omega_0} B(s, d\xi) \varphi(s, \xi) ds \\
	&+ \int_0^t \int_0^s \int_{\Omega_0} B(s-a, d\xi)  \mathcal F(a, \xi) G(\varphi(s,\cdot))(X(a,\xi))da  ds \\
	&+ \int_0^t \int_\Omega \mathcal F(s,x) G(\varphi(s,\cdot))(x) M_0( dx) ds.
\end{align*}
Hence, to deduce that $M$ satisfies \eqref{PDE}, we have to prove that 
\begin{align}\label{eq for M}
&\int_0^t \int_{\Omega_0} B(s, d\xi) \varphi(s, \xi) ds \\
	&+ \int_0^t \int_0^s \int_{\Omega_0} B(s-a, d\xi)  \mathcal F(a, \xi) G(\varphi(s,\cdot))(X(a,\xi))da ds \nonumber \\
	&+ \int_0^t \int_\Omega \mathcal F(s,x) G(\varphi(s,\cdot))(x) M_0( dx) ds \nonumber \\
&= \int_0^t \int_\Omega  G(\varphi(s,\cdot))  M(s, dx) ds  \nonumber \\
	&+ \int_0^t \int_\Omega \left( \int_\Omega  \varphi(s,x) \nu(\eta, dx) \right) \Lambda(\eta)  M(s , d\eta ) ds\nonumber  
\end{align}
} 
	\eu{Using \eqref{eq:m in terms of b} 
to compute 
\[
 \int_0^t \int_\Omega G(\varphi(s,\cdot))(x )M(s, dx) ds 
\] 
 we deduce that }
	\begin{align}\label{eq for M2}
	&\int_0^t \int_0^s \int_{\Omega_0} B(s-a, d\xi)  \mathcal F(a, \xi) G(\varphi(a,\cdot ))(X(a,\xi)) da ds \\
&=  \int_0^t \int_\Omega G(\varphi(s,\cdot))(x )M(s, dx) ds \nonumber \\
	&- \int_0^t \int_\Omega \mathcal F(s,x)  G(\varphi(s,\cdot))(x ) M_0( dx) ds\nonumber 
	\end{align}
\eu{Let
\[
L(s, x ):=\Lambda(x) \int_\Omega \nu(x, dy) \varphi(s,y). 
\]
We integrate the function $L(s,\eta) $ against the measure $M(s,d \eta) ds $ on $\mathbb R_+ \times \Omega$ and deduce that 
\begin{align*}
&\int_0^t \int_\Omega L(s, \eta) M(s,d \eta) ds = 
& \int_0^t \int_\Omega \int_\Omega \varphi(s,x)  \nu(\eta, dx) \Lambda(\eta)  M(s, d\eta )  ds \\
\end{align*}
On the other hand, integrating $L(s,\eta) $ against the following measure on $\mathbb R_+ \times \Omega$ 
\[
\int_0^s \int_{\Omega_0} B(s-a, d\xi) \mathcal F(a, \xi) \delta_{X(a, \xi)}(\cdot) da ds +  \int_\Omega \mathcal F(s,x) \delta_{X(s,x)}(\cdot)  M_0(dx) ds, 
\] 
and, additionally, using the fact that $B$ satisfies \eqref{RE}, as well as the formula \eqref{eq:m in terms of b}, we deduce that 
	\begin{align}\label{eq for M3}
&	\int_0^t \int_{\Omega_0} B(s, d\xi) \varphi(s, \xi) ds \\
&=\int_0^t \int_\Omega L(s, \eta) M(s,d \eta) ds \nonumber \\
& =  \int_0^t \int_\Omega \int_\Omega \left(\varphi(s,y) \nu(\eta, dy) \right) \Lambda(\eta)  M(s, d\eta )  ds \nonumber
	\end{align}}
\eu{Combining \eqref{eq for M} with \eqref{eq for M2} and with \eqref{eq for M3} we find that $M $ satisfies \eqref{PDE}. }
\qed\end{proof}
Finally we prove that there exists a unique solution for equation \eqref{PDE}. 
\begin{proposition}\label{prop:PDE uniq}
	If both $M_1$ and $M_2 $ solve \eqref{PDE} with the same initial condition $M_1(0, \cdot)=M_2(0,\cdot)=M_0(\cdot)$, then  $M_1(t, \cdot)=M_2(t, \cdot)$ for every $t>0$. 
\end{proposition} 
\begin{proof}
	Let  $M=M_1-M_2.$
	Since $M_1$ and $M_2$ satisfy equation \eqref{PDE}, it follows that
	\[
	\int_\Omega \varphi(t, x) M(t, dx)=   \int_0^t \int_\Omega \mathcal G \varphi (s,x) M(s, dx) ds,
	\]
	where 
	\begin{align}
	\mathcal G  \varphi(s,x):= \partial_s \varphi(s, x)+  g(x)  \partial_x \varphi(s,x) -\tilde{\mu}(x)  \varphi(s,x) + \int_\Omega \varphi(s,\eta)  \Lambda(x)  \nu(x, d\eta).
	\end{align}
	We prove that for every $\psi \in C_c(\Omega_0) $ there exists a $\varphi \in C^1([0,t],C^1_c(\Omega_0))$ such that $\mathcal G \varphi =0$ and $\varphi(t,x)=\psi(x).$
	This implies that $M(t,\text{supp} \psi)=0$. Making the function $\psi$ vary we deduce that $M(t,A)=0$ for every $A \in \mathcal B(\Omega_0)$. From this we find that $M_1=M_2.$
	
	Let us prove that for every $\psi$ there exists a solution to the equation $\mathcal G \varphi=0$ with final condition $\varphi(t,x)=\psi(x).$
	Thanks to the definition of $\mathcal G$, this is equivalent to prove that there exists a unique solution to  
	\begin{equation}\label{eq uniqueness}
	\partial_s \varphi(s, x)=-  g(x)  \partial_x \varphi(s,x) +\tilde{\mu}(x)  \varphi(s,x) - \int_\Omega \varphi(s,\eta)  \Lambda(x)  \nu(x, d\eta).
	\end{equation}
	Integrating along the characteristic we can rewrite the equation in a fixed point form:
	\begin{align*} 
	&\varphi(s,X(s,x))=\mathcal T \varphi(s,x)
	\end{align*} 
	with $X(s,x)$ being the solution of the ODE $\frac{dy}{ds}=g(y)$ with initial datum $y(0)=x$ and with
	\begin{align*}
	\mathcal T \varphi(s,x):=&\psi(X(t,x)) e^{- \int_s^t \tilde{\mu}(X(v,x)) dv} \\
	&+ \int_s^t \int_\Omega \varphi(v,\eta)  \Lambda(X(v,x))  \nu(X(v,x), d\eta)  e^{- \int_s^v \tilde{\mu}(X(v,x)) dv}dv
	\end{align*}
	Since 
\begin{align*}
	& \| \mathcal T \varphi_2- \mathcal T \varphi_2\|_\infty  \leq 2 \cdot \| \varphi_1-\varphi_2 \|_\infty \cdot \sup_{x \in \Omega} \int_s^t \tilde{\mu}(X(v,x))  e^{- \int^s_v \tilde{\mu}(X(a,x)) da} dv \\
&  \leq 2 \cdot \| \varphi_1-\varphi_2 \|_\infty \cdot  \sup_{x \in \Omega} \left(1- e^{- \int_0^t \tilde{\mu}(X(a,x)) da}  \right),
	\end{align*}
	we deduce that for sufficiently small $\overline {t}>0$ the operator $\mathcal T$, that maps $C^1([0,\overline{t}], C^1_c(\Omega_0))$ in itself, is a contraction. Hence there exists a unique solution of equation \eqref{eq uniqueness} as $s $ varies between $0$ and $\overline t.$ A solution for every time $t$ can be proven to exists by repeating the above reasoning for every interval of time of length $\overline t.$
\qed\end{proof}
\begin{proof}[Proof of Proposition \ref{prop:PDE}]
It is enough to combine the statement of Proposition \ref{prop:PDE uniq} and \ref{prop:PDE exists}. 
\qed\end{proof}

\section*{Declarations}
\textbf{Funding}: The research was supported by the ERC grant 741487\\
\textbf{Conflicts of interest}: The authors have no conflicts of interest to declare that are relevant to the content of this article \\
\textbf{Availability of data and material}: Not applicable \\
\textbf{Code availability}: Not applicable 

\vspace{1cm}

\textbf{Eugenia Franco} University of Helsinki, Department of Mathematics and Statistics, 

\hspace{0.5 cm} Helsinki, Finland,
eugenia.franco@helsinki.fi,  orcid:0000-0002-5311-2124  \\

\textbf{Odo Diekmann} Utrecht University, Mathematical Institute, Utrecht, Netherlands,

\hspace{0.5 cm}orcid:0000-0003-4695-7601 \\

\textbf{Mats Gyllenberg} 
University of Helsinki, Department of Mathematics and Statistics,

\hspace{0.5 cm} Helsinki, Finland, orcid:0000-0002-0967-8454 \\

\bibliographystyle{plain}
			\bibliography{biblio2}
			\nocite{*}
%\begin{acknowledgements}
%\end{acknowledgements}


\end{document}